\documentclass[12pt]{amsart}

\usepackage{amsmath,amssymb}
\usepackage{mathrsfs} 
\usepackage{enumerate}
\usepackage{appendix}
\usepackage{wrapfig}

\usepackage{tablefootnote}

\usepackage{pgf,tikz}
\usetikzlibrary{arrows}
\usepackage[all]{xy}
\usepackage{appendix}

\usepackage{array,booktabs,arydshln,xcolor}

\usepackage{multicol}

\usepackage{geometry} 
\usepackage{tikz-cd}
\tikzcdset{arrow style=tikz,diagrams={>=stealth}}

\RequirePackage{hyperref}
\hypersetup{pdfpagemode={UseOutlines},
bookmarksopen=true,
bookmarksopenlevel=0,
hypertexnames=false,
colorlinks=true, 
citecolor=teal, 
linkcolor=blue, 
urlcolor=magenta, 
pdfstartview={FitV},
unicode,
breaklinks=true,
}

\newtheorem{theorem}{Theorem}[section]
\newtheorem{theoremintro}{Theorem}
\newtheorem{corollaryintro}[theoremintro]{Corollary}
\newtheorem{lemma}[theorem]{Lemma}
\newtheorem{proposition}[theorem]{Proposition}
\newtheorem{corollary}[theorem]{Corollary}
\newtheorem{conjecture}[theorem]{Conjecture}
\newtheorem{cproposition}[theorem]{Proposition *}

\newtheorem*{theorem*}{Theorem}
\newtheorem*{ques*}{Question}
\newtheorem*{prop*}{Proposition}

\theoremstyle{definition}
\newtheorem{definition}[theorem]{Definition}
\newtheorem{example}[theorem]{Example}
\newtheorem{examples}[theorem]{Examples}

\newtheorem{ques}[theorem]{Question}

\newtheorem*{definition*}{Definition}

\theoremstyle{remark}
\newtheorem{remark}[theorem]{Remark}

\numberwithin{equation}{section}

\title[Log-Coarse structures of Lie groups and hyperbolic spaces]{On the logarithmic coarse structures of Lie groups and hyperbolic spaces}

\author{Gabriel Pallier}
\address{D\'epartement de Math\'ematiques,
Universit\'e de Fribourg,
Chemin du Mus\'ee, 23 --
CH-1700 Fribourg, Switzerland
}

\thanks{Author supported by the European Research Council, ERC Starting Grant 713998 GeoMeG ‘Geometry of Metric Groups’.}

\email{gabriel@pallier.org}


\date{\today}
\dedicatory{}

\begin{document}



\maketitle

\begin{abstract}

We characterize the Lie groups with finitely many connected components that are $O(u)$-bilipschitz equivalent (almost quasiisometric in the sense that the sublinear function $u$ replaces the additive bounds of quasiisometry) to the real hyperbolic space, or to the complex hyperbolic plane. 
The characterizations are expressed in terms of deformations of Lie algebras and in terms of pinching of sectional curvature of left-invariant Riemannian metrics in the real case.
We also compare sublinear bilipschitz equivalence and coarse equivalence, and prove that every coarse equivalence between the logarithmic coarse structures of geodesic spaces is a $O(\log)$-bilipschitz equivalence.  
The Lie groups characterized are exactly those whose logarithmic coarse structure is equivalent to that of a real hyperbolic space or the complex hyperbolic plane. 
Finally we point out that a conjecture made by Tyson about the conformal dimensions of the boundaries of certain hyperbolic buildings holds conditionally to the four exponentials conjecture.
\end{abstract}

\tableofcontents

\begingroup

\renewcommand{\thetheoremintro}{{\rm \Alph{theoremintro}}}

\section{Introduction}

\subsection{Background}
Let $X$ and $Y$ be metric spaces. A map $\phi: X \to Y$ is a quasiisometry if there exists $\lambda \geqslant 1$ and $c \geqslant 0$ such that  
$\lambda^{-1} d(x,x') - c \leqslant d(\phi(x), \phi(x')) \leqslant \lambda d(x,x') + c$ and for every $y$ in $Y$, $d(y, \phi(X)) \leqslant c$.
Let a locally compact, compactly generated group $G$ act continuously co-compactly properly by isometries on a locally compact geodesic space $X$; we call $X$ a geometric model of $G$. Every such $G$ has a geometric model (e.g. Cayley graphs if it is finitely generated, Riemannian metrics if it is connected Lie), and two geometric models of a given $G$ will always be equivariantly quasiisometric. Thus one can speak of quasiisometries between compactly generated locally compact groups.

Quasiisometries arose from the interpretation by Margulis of the work of Mostow on the rigidity of locally symmetric spaces \cite{MargulisMostow}.
Specifically, Margulis conjectured that a quasiisometry of a higher rank symmetric space $X$ should lie at bounded distance from an isometry, implying Mostow rigidity for the co-compact lattices in $X$, but also the fact that any finitely {generated group} $G$ quasiisometric to $X$ must surject with finite kernel onto such a uniform lattice.
This was first proved by Kleiner and Leeb using asymptotic cones, a tool formerly introduced by Gromov, in the form recast by van den Dries and Wilkie \cite{KleinerLeebQI}.
The interplay of quasiisometries and asymptotic cones can actually be expressed in the following way: between geodesic metric spaces, a map is a quasiisometry if and only if it goes through any asymptotic cone (with possibly moving observation centers); see \S\ref{subsec:going-through-cones} for a precise statement.
Kleiner and Leeb's theorem is part of a more general principle which, in contrast with Mostow rigidity, makes sense (and is stated below) for locally compact compactly generated groups.

\begin{theorem}[Many authors, see {\cite[Theorem 19.25]{CornulierQIHLC}} and the references there]
Let $G$ be a compactly generated locally compact group and let $X$ be a Riemannian symmetric space of non-compact type. The following are equivalent:
\begin{enumerate}[{\rm (1)}]
    \item
    $G$ is quasiisometric to $X$.
    \item 
    $X$ is a Riemannian geometric model for $G$.
\end{enumerate}
Moreover, if $G$ is a Lie group isomorphic to a closed subgroup of upper triangular real matrices (call such groups completely solvable), then the former are equivalent to:
\begin{enumerate}[{\rm (1)}]
    \setcounter{enumi}{2}
    \item 
    $G$ is isomorphic to a maximal completely solvable\footnote{Beware that the maximal solvable subgroups of $\operatorname{Isom}(X)$ (which is a real Lie group) are not always completely solvable; they only have a co-compactly embedded such subgroup.}  subgroup of $\operatorname{Isom}(X)$.
\end{enumerate}
\end{theorem}

The case $G$ finitely generated and $X = \mathbb H_{\mathbf R}^n$, $n \geqslant 3$ is up to formulation due to Tukia \cite{TukiaQC2mob} and was among the early results motivating the first formulation of quasiisometric rigidity by Gromov \cite{GromovQIprogram}. 
Gromov almost simultaneously proposed a vast programme of classifying finitely generated groups and isometrically homogeneous spaces up to quasiisometry \cite{GromovAsymptoticHomogeneous}. 
For nonsemisimple connected or nonarchimedean Lie groups and their lattices, this is far from being achieved today.

Between geodesic metric spaces, quasiisometries are exactly the coarse equivalences, that is, they respect the bounded coarse structure described as the family of entourages
\begin{equation}
    \notag
    \mathcal{E}^{O(1)} = \left\{ E \subseteq X \times X: \exists D > 0,\, \sup_{(x,x') \in E} {d_X(x,x')} \leqslant D \right\}.
\end{equation}

A broad interpretation of Gromov's programme is the following: classify the coarse structures generated by compactly generated groups, and characterize those that are generated by particular geometric models, especially the Riemannian symmetric or homogeneous spaces.
Recently, certain extensions of Gromov's questions have been addressed where coarse surjectivity is relaxed.
These are the study of the rigidity of quasiisometric embeddings (see \cite{fisher2015quasi} and \cite{FisherWhyteQIEmbedSym} for symmetric spaces) and of
the (non)-existence of coarse embeddings (see \cite{hume2020poincare} for connected Lie groups).

\subsection{Main results}
In this paper, we are interested in maps more general than quasiisometries.
In contrast with quasiisometries, these can still be characterized as going through asymptotic cones, though not through asymptotic cones for any sequence of basepoints (we elaborate on \cite{CornulierCones11} for this; see \S\ref{subsec:going-through-cones} for a precise statement). The coarse surjectivity assumption is not exactly relaxed, but adapted accordingly.

For the needs of the next definition, say that a function $u: [0,+\infty) \to (0,+\infty)$ is admissible if 
$\limsup_{r \to + \infty} u(r)/r = 0$ (that is, $u$ is sublinear)
and for every $A \geqslant 1$ there exists $B<+\infty$ such that for all sequences $(r_n, s_n)$ with $1/A \leqslant \inf s_n/r_n \leqslant \sup s_n/r_n \leqslant A$, 
$
\sup u(s_n)/u(r_n) \leqslant B
$. 
Examples of admissible function include $u(r) = r^\alpha \log^\beta (r) $ for $r \geqslant 2$ (and $u(r)=1$ otherwise) when $\alpha \in(-\infty, 1)$ and $\beta \in (-\infty, +\infty)$.

\begin{definition}[After\footnote{
This is \cite[Definition 2.1]{cornulier2017sublinear} with a mild difference in the definition of the class of admissible functions that we make in order to include functions with limit $0$ at $\infty$ (see \S \ref{sec:coars-geometry} for why).} {\cite{cornulier2017sublinear}}]
\label{def:sbe}
Let $u$ be an admissible function.
A map $\phi: (X,o_X) \to (Y,o_Y)$ between pointed metric spaces realizes a (large-scale) $O(u)$-bilipschitz equivalence if there are $\kappa \geqslant 1$ and $c \geqslant 0$ such that, for all $x,x' \in X$ and $y \in Y$,

\begin{align}
        \label{eq:sbe-1}
       -cu(\vert x \vert \vee \vert x' \vert) + \frac{d_X(x,x') }{\kappa}  \leqslant d_Y(\phi(x), \phi(x')) & \leqslant \kappa d_X(x,x') + cu(\vert x \vert \vee \vert x' \vert) \\ 
       \label{eq:sbe-2}
       d_Y(y, \phi(X)) & \leqslant c u (\vert y \vert),
\end{align}
where $\vert x \vert$ denotes $d_X(o_X,x)$, and ``$\vee$'' denotes max.
\end{definition}

We also call $o(r)$-bilipschitz equivalence, or sublinear bilipschitz equivalence (abbreviated SBE in some places), a $\phi$ such that \eqref{eq:sbe-1} and \eqref{eq:sbe-2} hold with some unspecified strictly sublinear function in lieu of $cu$. 

Quasiisometries correspond to $u\equiv 1$. 
Of particular importance in this paper is $u = \log$. 
Given an admissible function $u$, we consider the coarse structure on metric spaces with the following entourages:
\begin{equation}
    \notag
    \mathcal{E}^{O(u)} = \left\{ E \subseteq X \times X: \limsup_{r \to +\infty} \sup_{(x,x') \in E,\,  \sup (d(o_X,x), d(o_X,x')) \geqslant r}  \frac{d_X(x,x')}{u (\vert x \vert)} < +\infty  \right\}.
\end{equation}
These are quantitative refinements of the coarse structure introduced in \cite{DranishnikovSmith}.
$O(u)$-bilipschitz equivalences are always $\mathcal E^{O(u)}$-coarse equivalences. We prove that the converse holds between geodesic spaces when $u= \log$:
\begin{theoremintro}
\label{th:coarse-is-sbe}
Assume that $X$ and $Y$ are geodesic. Then $\phi : X \to Y$ is $O(\log)$-bilipschitz if and only if it is a coarse equivalence of $\mathcal{E}^{O(\log)}$.
\end{theoremintro}

This is a variant of the well-known fact that coarse equivalences between geodesic spaces are quasiisometries, however the proof is significantly more involved.

Keeping quasiisometric rigidity and classification in mind, it is natural to ask:

\begin{ques}[Rigidity]
\label{ques:rigidity}
Let $u$ be as above, $u \geqslant 1$. Which compactly generated locally compact groups $G$ are $O(u)$-bilipschitz equivalent to a given symmetric space $X$?
\end{ques}

\begin{ques}[Classification]
\label{ques:classification}
Given $u$ as above, $u \geqslant 1$, classify isometrically homogeneous spaces up to $O(u)$-bilipschitz equivalence.
\end{ques}

The following theorem was stated in the introduction of the author's thesis. 
While essentially following from the combination of \cite{DranishnikovSmith}, \cite{HigesPeng} and the coarse interpretation of $o(r)$-bilipschitz equivalences, it was not extracted at first sight from the literature, so we provide a proof here (relying on the above cited works). 
Recall for the statement that all the maximal compact subgroups of a connected Lie groups are conjuguated \cite{BorelMaximaux}.

\begin{theoremintro}[After {\cite{DranishnikovSmith}} and {\cite{HigesPeng}}]
\label{th:geodim}
Let $G$ and $H$ be connected Lie groups. If there exists a $o(r)$-bilipschitz equivalence $\phi: G \to H$, then
\begin{equation}
    \operatorname{geodim}(G) = \operatorname{geodim}(H),
\end{equation}
where $\operatorname{geomdim}(G)$ denotes $\dim G/K$ if $K$ is any maximal compact subgroup of $G$.
Especially, if $G$ and $H$ are solvable and simply connected, then $\dim G = \dim H$.
\end{theoremintro}
The theorem actually holds for every $o(r)$-coarse equivalences $\phi$, see {\S\ref{subsec:assouad-nagata}}. If $G$ and $H$ are nilpotent, then $\operatorname{geodim}$ is the covering dimension of their asymptotic cones and Theorem \ref{th:geodim} also follows from \cite{PanCBN}. 

Next, building on \cite{CornulierCones11}, \cite{CoTesContracting} and \cite{pallier2019conf} (which was already concerned with Question \ref{ques:classification}) we formulate below a partial answer to Question \ref{ques:rigidity} for connected Lie groups $G$ and real hyperbolic space $X$. 
While this is not made apparent in the statement, all the groups obtained are either of Heintze or rank-one type, in the typology of \cite{CoTesContracting} and \cite{CCMT}.

\begin{theoremintro}
\label{th:Tukia-SBE}
Let $G$ be a Lie group with finitely many connected components and $n \geqslant 2$ an integer.
The following are equivalent:
\begin{enumerate}[{\rm (\ref{th:Tukia-SBE}.1)}]
    \item 
    \label{sublinear-characterization}
    $G$ is $O(u)$-bilipschitz equivalent to $\mathbb H^n_{\mathbf R}$, for some sublinear admissible $u$.
    \item
    \label{item:log-characterization}
    $G$ is $O(\log)$-bilipschitz equivalent to $\mathbb H^n_{\mathbf R}$.
    \item
    \label{pinching-characterization}
    For every $\varepsilon>0$, $G$ has an $n$-dimensional Riemannian model with $-1 \leqslant K \leqslant -1+ \varepsilon$.
\end{enumerate}
Moreover, if $G$ is completely solvable with Lie algebra $\mathfrak g$, the former conditions are equivalent to:
\begin{enumerate}[{\rm (\ref{th:Tukia-SBE}.1)}]
    \setcounter{enumi}{3}
    \item
    \label{item:degeneration-characterization-real}
    $\mathfrak g$ degenerates to the (isomorphism class of a) maximal completely  solvable subalgebra $\mathfrak g_\infty$ of $\mathfrak o(n,1)$.
    \item 
    \label{item:explicit-characterization-tukia}
    The Lie algebra $\mathfrak g$ decomposes as $\mathfrak [\mathfrak g, \mathfrak g] \oplus \mathbf R A$, where $[\mathfrak g, \mathfrak g]$ is abelian and $\operatorname{ad}_A$ is unipotent on $[\mathfrak g, \mathfrak g]$.
\end{enumerate}
\end{theoremintro}

Here saying that $\mathfrak g$ degenerates to $\mathfrak g_\infty$ means that the Zariski closure of the orbit of $\mathfrak g$ in the variety of Lie algebra laws contains $\mathfrak g_\infty$, which occurs especially if there is a continuous $( \varphi_t)_{t \in [0,+\infty)}$ in $\mathrm{GL}(\mathfrak g)$ and a linear isomorphism $\psi : \mathfrak g \to \mathfrak g_\infty$ such that for every $X,Y \in \mathfrak g$,
\begin{equation*}
    \lim_{t \to + \infty} \varphi_t^{-1} [\varphi_t X, \varphi_t Y]_{\mathfrak g} = \psi^{-1} [\psi X, \psi Y]_{\mathfrak g_\infty}. 
\end{equation*}

Theorem \ref{th:Tukia-SBE} combines known results.
That (\ref{th:Tukia-SBE}.\ref{sublinear-characterization}) implies (\ref{th:Tukia-SBE}.\ref{pinching-characterization}) rests on \cite{CoTesContracting} and \cite{pallier2019conf}, the equivalence of the last two conditions (\ref{th:Tukia-SBE}.\ref{item:degeneration-characterization-real}) and (\ref{th:Tukia-SBE}.\ref{item:explicit-characterization-tukia}) is \cite[Theorem 6.2]{LauretDegenerations} with minor enhancement, the implication from (\ref{th:Tukia-SBE}.\ref{pinching-characterization}) to (\ref{th:Tukia-SBE}.\ref{item:explicit-characterization-tukia}) uses \cite{PansuDimConf}, while the fact that (\ref{th:Tukia-SBE}.\ref{item:explicit-characterization-tukia}) implies (\ref{th:Tukia-SBE}.\ref{item:log-characterization}) is a consequence of \cite{CornulierCones11}.
When $n=2$, Theorem \ref{th:Tukia-SBE} reduces to a weak form of \cite[Corollary 1.10(2)]{cornulier2017sublinear}. The statement is simpler when $n=2$, since $2$-dimensional homogeneous metrics have constant curvature. It holds with the mere assumption that $G$ be compactly generated locally compact, and the techniques are specific, relying essentially on \cite{GabaiConFuchsAnnals}, \cite{CassonJungreis}. 

For general connected Lie groups, the process of going from $\mathfrak g$ to a less complicated $\mathfrak g_\infty$ so that the simply connected $G$ and $G_\infty$ remain $O(u)$-bilipschitz equivalent has an alternative description given in \cite{CornulierCones11} (recalled here in Theorem \ref{th:cornulier-red}) which does not require degenerations.
Our formulation using degeneration is half-successful in this generality. While it also applies well when $\mathfrak g$ is nilpotent (in this case it is due to Pansu \cite{PanCBN}), we do not know whether $\mathfrak g_\infty$ is a degeneration of $\mathfrak g$ in general. This will be discussed in \S\ref{subsec:nilpotent}.

The appearance of the sectional curvature pinching in characterization (\ref{th:Tukia-SBE}.\ref{pinching-characterization}) might appeal to some comments.
The sphere theorem of Berger and Klingenberg implies that on a positively curved Riemannian manifold, a pinching sufficiently close to $1$ determines the homotopy type of the (finite) universal cover. Namely, the latter must be a sphere.
As demonstrated by Gromov and Thurston, there is no counterpart for this in negative curvature as one constructs sequences of closed manifolds supporting negatively curved metrics, arbitrarily pinched close to $-1$, albeit with vanishing first cohomology, hence not homotopy equivalent to any locally symmetric space of constant negative curvature \cite{MGromovThurstonPinching}.

This is not even repaired if one replaces homotopy equivalence with quasiisometry, as one constructs isometrically homogeneous manifolds with pinching $> -1/4$ or even arbitrarily close to $-1$ (characterized in \cite{EberleinHeber}, see \S\ref{subsec:pinching}), that are not quasiisometric to $\mathbb H^n_{\mathbf R}$ \cite{XieLargeScale}.
Theorem \ref{th:Tukia-SBE} implies the following as far as Lie groups are concerned.

\begin{corollaryintro}[of Theorem {\ref{th:Tukia-SBE}}]
\label{cor:sbe-corona}
If a connected Lie group $G$ has Riemannian models with pinching arbitrarily close to $-1$, then its sublinear Higson corona $\nu_L G$ is homeomorphic to that of a real hyperbolic space.
\end{corollaryintro}

(We recall the definition of the sublinear Higson corona in \S\ref{subsec:assouad-nagata}.)

Finally, we also characterize the Lie groups $O(u)$-bilipschitz equivalent to $\mathbb H^2_{\mathbf C}$. Following \cite{Cornulier_Focal}, say that the locally compact $G$ and $H$ are commable if there exists a finite sequence of homomorphisms with compact kernels and co-compact images (both directions allowed) between $G$ and $H$.
\begin{theoremintro}
\label{thm:groups-sbe-to-h2c}
Let $G$ be a Lie group with finitely many connected components. 
The following are equivalent:
\begin{enumerate}[{\rm (\ref{thm:groups-sbe-to-h2c}.1)}]
    \item 
    \label{item:G-SBE-to-H2C}
    $G$ is $O(u)$-bilipschitz equivalent to $\mathbb H^2_{\mathbf C}$
    \item 
    \label{item:G-log-SBE-to-H2C}
    $G$ is $O(\log)$-bilipschitz equivalent to $\mathbb H^2_{\mathbf C}$
    \item
    \label{item:G-SBE-to-commable}
    $G$ is {\em commable} either to the semisimple $\operatorname{SU}(2,1)$ or to the solvable $S' = H_3 \rtimes \mathbf R$, where $H_3$ is the $3$-dimensional Heisenberg group and $t \in \mathbf R$ acts by
    \[ t.\exp (x,y,z) = \exp(e^t x + te^ty, e^t y, e^{2t} z) \]
    in a basis of infinitesimal generators $X, Y, Z$ such that $[X,Y] = Z$.
\end{enumerate}
Moreover, if $G$ is completely solvable, the former conditions are equivalent to: 
\begin{enumerate}[{\rm (\ref{thm:groups-sbe-to-h2c}.1)}]
   \setcounter{enumi}{3}
   \label{item:G-SBE-to-H2C-degenerates}
    \item 
    $\mathfrak g$ degenerates to the maximal completely solvable subalgebra of $\mathfrak{u}(2,1)$
\end{enumerate}
where $\mathfrak g$ denotes the Lie algebra of $G$.
\end{theoremintro}

The restriction that $G$ be a connected Lie group makes Theorems \ref{th:Tukia-SBE} and \ref{thm:groups-sbe-to-h2c} very special compared to the QI rigidity recalled above, and we benefit from some constraints of the structure theory of Lie groups. Unlike Theorem \ref{th:Tukia-SBE}, Theorem \ref{thm:groups-sbe-to-h2c} requires some additional technical work, done in \S\ref{sec:proofE}.

\subsection{Other spaces}
We know little even about Question \ref{ques:classification} for higher rank symmetric spaces and other settings, even when quasiisometric rigidity is known to hold.
In the end of this paper, we summarize the current situation for symmetric space of higher rank and Fuchsian buildings; especially we explain why their classification is still open at the time of writing. 

\subsection{Organization of the paper}
\S\ref{sec:coars-geometry} is a general discussion on the theoretical status of SBE (especially, as compared to QI). It is not concerned with Lie groups and can be read independently. \S\ref{subsec:prelim} provides some preliminaries for \S\ref{sec:coars-geometry}. \S\ref{sec:proofB} and \S\ref{sec:proofE} establish the characterizations of Lie groups $O(u)$-bilipschitz equivalent to real, resp. complex hyperbolic space, and follow a similar scheme, so we advise to read \S\ref{sec:proofB} first.
Most of the technical input in this paper serve the proofs of Theorems \ref{th:coarse-is-sbe} and \ref{thm:groups-sbe-to-h2c} and is concentrated in \S\ref{subsec:coarse-structures} and \S\ref{subsec:pointedSphere} respectively.
SBE appears to be quite a new notion and some of the contents of this paper are rather expository in nature, including especially \S\ref{subsec:assouad-nagata} on Theorem \ref{th:geodim}, \S\ref{subsec:pinching} and \S\ref{subsec:degenerations} preparing the proof of Theorem \ref{th:Tukia-SBE}, and \S\ref{subsec:nilpotent} on general connected Lie groups.
\S\ref{subsec:hrss} and \S\ref{subsec:rafb} gather a collection of independent remarks. Finally, a certain amount of actual Lie algebra cohomology computations (for trivial and adjoint modules) are required in particular in Lemma \ref{lem:deg-to-bnC} and Example \ref{exm:L67}; we summarize these in Appendix \ref{app:cohomcomput}.

\subsection*{Convention, notation}
When $G,H, \ldots $ are simply connected Lie groups, then $\mathfrak g, \mathfrak h, \ldots$ denote their Lie algebra. 
We often consider semi-direct products of the form $N \rtimes \mathbf R$ or $\mathfrak n \oplus \mathbf R$; we then write $N \rtimes_\alpha \mathbf R$ or $\mathfrak n \rtimes_\alpha \mathbf R$ meaning that the {Lie algebra} representation $\rho : \mathbf R \to \operatorname{Der}(\mathfrak n)$ (and not the Lie group representation) is determined by $1 \mapsto \alpha$.
If $V$ is a module and $n$ a nonnegative integer, we denote by $\Lambda^n V$ its $n$-fold exterior product and by $\Lambda^n V^\ast$ the $n$-fold exterior product of its dual.
If $\mathfrak g$ is a Lie (sub)algebra, $\operatorname{Vect}(\mathfrak g)$ will denote its underlying vector (sub)space. (This is useful to avoid confusions because we may sometimes consider several Lie brackets on a given space.)
\endgroup

\subsection{Aknowledgement}
The author thanks Salim Tayou for a useful discussion, Tomohiro Fukaya for useful comments on a draft of this paper, and especially the anonymous referee for many comments and corrections.

\section{Coarse geometry and Theorems \ref{th:coarse-is-sbe} and \ref{th:geodim}}
\label{sec:coars-geometry}

This \S motivates sublinear bilipschitz equivalence (defined in \S\ref{def:sbe}) by comparing it to the more standard notions of quasiisometry and coarse equivalence. This comparison will be made through the relations that sublinear bilipschitz equivalence enjoys with asymptotic cones and certain coarse structures. The relation to asymptotic cones is the reason why they were introduced by Cornulier in the first place, in \cite{CornulierDimCone} and then more explicitly\footnote{We should warn the reader about terminology: they were called ``cone bilipschitz'' in \cite{CornulierCones11} and ``asymptotically bilipschitz'' in \cite{DrutuKapovich}.} in \cite{CornulierCones11}, \cite{cornulier2017sublinear} (See \S\ref{subsubsec:unique} for precisely why).
In the end of this section, we show that the geometric dimension of connected Lie groups is a SBE invariant.

\subsection{Preliminaries}
\label{subsec:prelim}

\begin{definition}[Coarse equivalence and quasiisometry]
\label{def:qi-and-coarse}
Let $X$ and $Y$ be two metric spaces. A map $\phi \colon X \to Y$ is a (uniform) coarse embedding if there exists two proper functions $\rho_-$ and $\rho_+: [0,+\infty) \to [0,+\infty)$ such that for every $x,x' \in X$
    \begin{equation}
        \label{eq:coarse-equiv}
        \rho_-(d_X(x,x')) \leqslant d_Y(\phi(x), \phi(x')) \leqslant \rho_+(d_X(x, x')).
    \end{equation}
    The map
    $\phi$ is a coarse equivalence if moreover, there exists a coarse embedding $\psi: Y \to X$ and a constant $R \geqslant 0$ such that for all $x \in X$, $d_X(\psi \circ \phi(x),x) \leqslant R$ and for all $y \in Y$, $d_Y(\phi \circ \psi(y),y) \leqslant R$; we call $g$ a coarse inverse.
    $\phi$ is a $(\kappa, c)$-quasiisometric embedding if $\rho_-$ and $\rho_+$ can be taken affine in \eqref{eq:coarse-equiv}, namely $\rho_\pm(r) = \kappa^{\pm 1} r \pm c$. If in addition $\phi$ a coarse equivalence, $\phi$ is called a quasiisometry and any coarse inverse $g$ is also a quasiisometry ; equivalently a quasiisometry is a quasiisometric embedding $\phi$ such that $\sup_{y \in Y} d_Y(y, \phi(X)) < + \infty$. 
    We may define a quasiisometry only on a net, that is, a closed subspace $X^{(0)} \subseteq X$ such that $\sup_{x \in X} d(x, X^{(0)}) < + \infty$. 
\end{definition}

\begin{proposition}[{See e.g. \cite[3.B.9]{CornulierHarpeMetLCGroups}}]
\label{prop:coarse-geod-qi}
If $X$ and $Y$ are two geodesic metric spaces, then any coarse equivalence $\phi \colon X \to Y$ is a quasiisometry.
\end{proposition}

\begin{proposition}
\label{prop:milnor-svarc}
Let $G$ be a compactly generated locally compact group. Then
\begin{enumerate}[{\rm (1)}]
    \item 
    \label{item:svarc-milnor}
    If $G$ acts continuously, properly cocompactly by isometries on the locally compact geodesic spaces $X$ and $Y$, then there exists a quasiisometry $\phi: X \to Y$ such that 
    
   $$\sup_{(g,x) \in G \times X} d_Y(\phi(g.x), g.\phi(x)) < +\infty.$$
    \item
    \label{item:svarc-milnor-exists}
    There exists $X$ locally compact geodesic metric space and an isometric proper co-compact continuous action by isometries of $G$ on $X$.
\end{enumerate}
\end{proposition}

\eqref{item:svarc-milnor} is a consequence of \cite[Theorem 4.C.5]{CornulierHarpeMetLCGroups}. For \eqref{item:svarc-milnor-exists}, see \cite[Proposition 2.1]{CCMT}.
In this paper we call $X$ and $Y$ as in the previous proposition {geometric models} for $G$.

\subsection{Admissible sublinear functions}
\begin{definition}
\label{dfn:admissible}
Call $u: [0,+\infty) \to (0,+\infty)$  admissible if 
$\limsup_{r \to + \infty} u(r)/r = 0$ 
and for every $A \geqslant 1$ there exists $B<+\infty$ (only depending on $A$) such that for all sequences $(r_n, s_n)$ with $r_n \to + \infty$ and $1/A < \inf s_n /r_n \leqslant  \sup s_n/r_n < A$, one has 
\begin{equation}
\label{eq:admissible}
    1/B \leqslant \liminf \frac{u(s_n)}{u(r_n)} \leqslant  \limsup \frac{u(s_n)}{u(r_n)} \leqslant B.
\end{equation}
\end{definition}
\begin{lemma}
\label{lem:admissible2adm}
Let $u:[0,+\infty) \to (0,+\infty)$ be a sublinear function.
If $u$ is nondecreasing and $\limsup u(2r)/u(r) < +\infty$, resp.\ if $u$ is nonincreasing and $\liminf u(2r)/u(r)>0$, then $u$ is admissible.
\end{lemma}

\begin{proof}
Let us consider only the case where $u$ is nondecreasing, the proof going the same way. Let $A > 1$ and $(r_n, s_n)$ be such that $r_n \to + \infty$ and $\lbrace s_n /r_n \rbrace \in [1/A, A]$. Set $\beta = \limsup u(2r)/u(r)$. Since $u(s_n)/u(r_n) \leqslant 1$ when $s_n \leqslant r_n$, one has
\begin{align*}
    \limsup \frac{u(s_n)}{u(r_n)} & = \sup \left( 1,\, \limsup_{n: s_n \geqslant r_n} \frac{u(s_n)}{u(r_n)} \right) \leqslant \beta^{\lceil \log_2 A \rceil}.
\end{align*}
This is the inequality on the right in \eqref{eq:admissible} with $B= \beta^{\lceil \log_2 A \rceil}$. The left inequality is obtained by reversing $r_n$ and $s_n$.
\end{proof}

The usefulness of Lemma \ref{lem:admissible2adm} may not be obvious. Let us give two motivations.
The first is that it ensures that the functions $u$ considered in \cite[Definition 2.4]{cornulier2017sublinear} are admissible in our sense. The second is that, while Definition \ref{dfn:admissible} allows a unified treatment for sublinear functions $u$ with $u(r) \to +\infty$ or $u(r) \to 0$ and is sufficient for our purposes in \S\ref{subsec:going-through-cones} and \S\ref{subsec:coarse-structures}, it appears that it is often easier to argue and prove the main statement of this section with monotonic functions $u$.

The above notion of admissible function resembles the much-studied class of (not necessarily sublinear) regularly varying function in real analysis, but we found no implication between the two without further assumptions.

\subsection{Going through cones}
\label{subsec:going-through-cones}

Let $(\sigma_n)$ be a sequence of positive real numbers.
For $(x_n), (x'_n) \in X^{\mathbf N}$, denote $(x_n) \sim_{\sigma_n} (x'_n)$ if $\sup d(x_n, x'_n)/\sigma_n <+\infty$ 
and $(x_n) \approx_{\sigma_n} (x'_n)$ if
\[ \limsup d(x_n, x'_n)/\sigma_n =0. \]
Let $\operatorname{Precone}(X,x_n, \sigma_n)$ denote the $\sim_{\sigma_n}$ equivalence class of $X$.

\begin{definition}[Cone and pointed cone]
\label{dfn:asymptotic-cones}
Let $X$, $(x_n)$ and $(\sigma_n)$ be as above.
Given a nonprincipal ultrafilter $\omega$ on $\mathbf N$, define $\operatorname{Cone}_\omega(X, x_n, \sigma_n)$ as follows: for any pair of sequences $(x'_n)$ and $(x''_n)$ in $\operatorname{Precone}(X,x_n, \sigma_n)$, define $\widehat d_\omega((x'_n), (x''_n))= \lim_{n \to \omega} d(x'_n, x''_n)/\sigma_n$. 
If $\widehat d_\omega((x'_n), (x''_n))$ is zero, identify $(x'_n)$ and $(x''_n)$, and for any sequence $(x'''_n)$ in $\operatorname{Precone}(X,x_n, \sigma_n)$, denote by $[x'''_n]$ the equivalence class of $(x'''_n)$.
Equip the quotient space $\operatorname{Cone}_\omega(X, x_n, \sigma_n)$ with the function $d_\omega$ by setting $d_\omega([x'_n], [x''_n]) = \widehat d_\omega((x'_n), (x''_n))$; this is a distance function (see e.g. \cite{KramerWeiss}).

Further, if $\sigma_n \to + \infty$, denote by $\operatorname{Cone}^\bullet_\omega(X, \sigma_n)$ the metric space obtained by fixing a basepoint and taking $x_n$ equal to the basepoint for all $n$ in the previous definition. This does not depend on the basepoint.
\end{definition}

\begin{remark}
When $\sigma_n \to 0$ and $(x_n)$ is a constant sequence, the space \[ \operatorname{Cone}_\omega(X,x_n, \sigma_n). \] is more commonly referred to as a metric tangent. However because our emphasis is on large-scale geometry and moving basepoints, and because the distinction would be artificial here, we denote both by the same name.
\end{remark}

Though our main interest is in homogeneous spaces, it is useful to work out some examples of asymptotic cones of nonhomogeneous spaces in order to appreciate the difference between quasiisometry and $O(u)$-bilipschitz equivalence.

\begin{examples}
\label{exm:riemannian-planes}
For $i \in \lbrace 1, \ldots, 4 \rbrace$ let $P_i$ be a Riemannian plane with metric $ds^2 = dr^2 + A_i(r)^2 d\theta^2$, where $A_1(r) = 1/r$, $A_2(r) = 1$, $A_3(r)= \log r$ and $A_4(r) = r/2$ for $r$ large enough.
See some sketches of $P_i$ on Figure \ref{fig:riemannian-planes}, and various cones on Table \ref{tab:riemannian-planes}.
\end{examples}

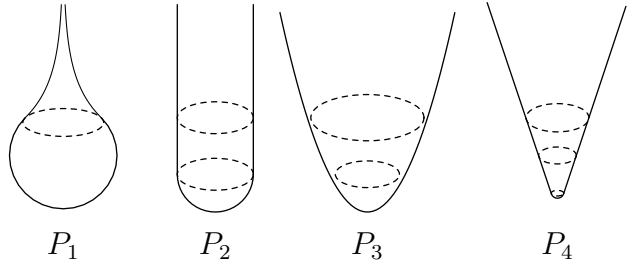
\begin{figure}
    \begin{center}
        
\begin{tikzpicture}[line cap=round,line join=round,>=triangle 45,x=0.5cm,y=0.5cm]
\clip(-5.5,-1.8) rectangle (11.5,5.5);
\draw [samples=50,rotate around={0.:(4.,-0.25)}, shift={(4,-0.5)},line width=0.5pt,domain=-2.3:2.3)] plot (\x,{(\x)^2/2/0.5});
\draw [shift={(-4.,1.)},line width=0.5pt]  plot[domain=-3.9269:0.78539,variable=\t]({1.*1.41421*cos(\t r)},{1.*1.414213*sin(\t r)});
\draw [shift={(0.,0.5)},line width=0.5pt]  plot[domain=3.14159:6.283,variable=\t]({1.*1.*cos(\t r)},{1.*1.*sin(\t r)});
\draw [line width=0.5pt] (-1.,0.5)-- (-1.,5);
\draw [line width=0.5pt] (1.,0.5)-- (1.,5);
\draw [variable=\y, domain=2:5] plot({-4+exp(2-\y)},{\y});
\draw [variable=\y, domain=2:5] plot({-4-exp(2-\y)},{\y});
\node at (-4,-1.4) {$P_1$};
\node at (0,-1.4) {$P_2$};
\node at (4,-1.4) {$P_3$};
\node at (9,-1.4) {$P_4$};
\draw [rotate around={0.:(-4.,1.875394)},line width=0.5pt,dash pattern=on 2pt off 2pt] (-4.,1.875) ellipse (1.062409 and 0.36899);
\draw [rotate around={0.:(0.,2.001)},line width=0.5pt,dash pattern=on 2pt off 2pt] (0.,2.001) ellipse (1. and 0.42269);
\draw [rotate around={0.:(0.,0.5011)},line width=0.5pt,dash pattern=on 2pt off 2pt] (0.,0.5011) ellipse (1.00000 and 0.4226);
\draw [rotate around={0.:(4.,0.501555)},line width=0.5pt,dash pattern=on 2pt off 2pt] (4.,0.5015) ellipse (0.843405 and 0.35371);
\draw [rotate around={0.:(4.,2.0022)},line width=0.5pt,dash pattern=on 2pt off 2pt] (4.,2.0022) ellipse (1.48754 and 0.6133);
\draw [rotate around={0.:(9,2.)},line width=0.5 pt,dash pattern=on 2pt off 2pt] (9,2.) ellipse (0.827 and 0.3748);
\draw [rotate around={0.:(9,1.)},line width=0.5 pt,dash pattern=on 2pt off 2pt] (9,1.) ellipse (0.4962 and 0.2248);
\draw [rotate around={0.:(9,0.)},line width=0.5 pt,dash pattern=on 2pt off 2pt] (9,0.) ellipse (0.1654 and 0.074);
\draw [shift={(9,0.03612)},line width=0.5 pt]  plot[domain=3.463343:5.9614,variable=\t]({1.*0.1695*cos(\t r)},{1.*0.169539*sin(\t r)});
\draw [line width=0.5 pt,domain=1.15:2.8391] plot(6+\x,{(-5.46852+1.9300*\x)/-0.643});
\draw [line width=0.5 pt,domain=3.1608:4.8] plot(6+\x,{(-7.972+2.5174*\x)/0.83});
\end{tikzpicture}
    \end{center}
    \caption{Sketch view of the four Riemannian planes of Example \ref{exm:riemannian-planes} with $\mathrm{U}(1)$ symmetry.}
    \label{fig:riemannian-planes}
\end{figure}

\begin{proposition}[Characterizing quasiisometries I]
\label{prop:charact-qi-1}
Let $X$ and $Y$ be geodesic metric spaces, and let $\phi: X \to Y$.
Then, $\phi$ is a quasiisometry if and only if
 for every $(\sigma_n)$ such that $\lim_n \sigma_n = +\infty$, it holds:
    \begin{align}
    \label{item:coarse-embed-precones}
    \forall (x_n) \in X^{\mathbf N}, \forall (x'_n) \in X^{\mathbf N},\, (x_n) \sim_{\sigma_n} (x'_n) & \implies \phi(x_n) \sim_{\sigma_n} \phi(x'_n) \tag{$\mathrm{I}_\sigma$} \\
         \forall (x_n) \in X^{\mathbf N}, \forall (x'_n) \in X^{\mathbf N}, \,\phi(x_n) \approx_{\sigma} \phi(x'_n)  & \implies (x_n) \approx_{\sigma} (x'_n) \tag{$\mathrm{II}_\sigma$}
         \label{eq:second-condition-precone} \\
       \forall (y_n) \in Y^{\mathbf N} \, \exists (x_n) \in X^{\mathbf N} & : \phi(x_n) \sim_{\sigma_n} y_n 
       \label{eq:coarse-surj-prec}
       \tag{$\mathrm{III}_\sigma$}
   \end{align}
   and then, given any such $\sigma_n$, for all pair $(x_n)\in X^{\mathbf N}$ and $(y_n) \in Y^{\mathbf N}$, either $\phi(\operatorname{Precone}(X,x_n, \sigma_n)) \cap \operatorname{Precone}(Y,y_n, \sigma_n)$ is empty or for every $\omega \in \beta \mathbf N \setminus \mathbf N$, $\phi$ induces a bilipschitz homeomorphism
   \begin{equation}
       \operatorname{Cone}_\omega(\phi, x_n, y_n, \sigma_n): \operatorname{Cone}_\omega(X, x_n, \sigma_n) \to \operatorname{Cone}_\omega(Y, y_n, \sigma_n)
       \tag{Cone}
       \label{eq:cone}
   \end{equation}
   whose bilipschitz constant only depends on $\phi$.
   \end{proposition}

    \begin{proof}
    Assume that $\phi$ is not a quasiisometry.
    Especially it is not a coarse equivalence, which means that there exists a sequence of positive numbers $(\rho_n)$ where $\rho_n \to + \infty$ as $n \to + \infty$, such that at least one of the following is true:
    \begin{enumerate}
        \item 
        \label{item:notcoarseembed}
        $\phi$ is not a coarse embedding: there exists an integer $M \geqslant 0$ and sequences of points $(x_n, x'_n)$ in $X$ such that
        \begin{enumerate}[a.]
            \item 
            \label{subitem:notcoarseembed1}
        either $d(x_n, x'_n) \leqslant M$ and $d(\phi(x_n), \phi(x'_n)) \geqslant \rho_n$ 
        \item 
        \label{subitem:notcoarseembed2}
        or $d(x_n, x'_n) \geqslant \rho_n$ and $d(\phi(x_n), \phi(x'_n)) \leqslant M$,
        \end{enumerate}
         or
        \item
        \label{item:notcoarsesurj}
        $\phi$ is not coarsely surjective: there exists a sequence of points $(y_n)$ in $Y^{\mathbf N}$ such that $d(y_n, \phi(X)) \geqslant \rho_n$.
    \end{enumerate}
    
    In case \ref{subitem:notcoarseembed1}, $(x_n) \sim_{\rho_n^{1/2}} (x'_n)$ whereas $(\phi(x_n)) \nsim_{\rho_n^{1/2}} (\phi(x'_n))$, contradicting \eqref{item:coarse-embed-precones} for $\sigma_n = \rho_n^{1/2}$.
    In case \ref{subitem:notcoarseembed2}, note that $\phi(x_n) \approx_{\rho_n} \phi(x'_n)$, while $x_n \sim_{\rho_n} x'_n$ does not hold, contradicting \eqref{eq:second-condition-precone} with $\sigma = \rho$. 
    If \ref{item:notcoarsesurj} holds, then \eqref{eq:coarse-surj-prec} does not hold with $\sigma_n = \rho_n^{1/2}$.

    Conversely, assume that $\phi$ is a $(\kappa, c)$-quasiisometry. Then  $x_n \sim_{\sigma_n} y_n$ means that $d_X(x_n, y_n) \leqslant C \sigma_n$ for some $C \geqslant 0$, so that $d_Y(\phi(x_n), \phi(y_n)) \leqslant \kappa C \sigma_n + c \leqslant (\kappa C + 1) \sigma_n$ for $n > \sup \lbrace m : \sigma_m \leqslant c \rbrace$. This proves \eqref{item:coarse-embed-precones}; the proof of \eqref{eq:second-condition-precone} goes the same way using the left inequality in \eqref{eq:coarse-equiv} with $\rho_-(r) = \kappa^{-1}r - c$.
    
    Finally, if $\phi$ is a quasiisometry, then for every parameters $(x_n), (y_n), (\sigma_n)$ as above with $\sigma_n \to + \infty$, $\phi \left( \operatorname{Precone}(X, x_n, \sigma_n) \right) \cap \operatorname{Precone}(Y,y_n, \sigma_n)$ is equal to
    \[
    \begin{cases}
    \emptyset & \text{if } y_n \nsim \phi(x_n) \\
    \operatorname{Precone}(Y,y_n, \sigma_n) & \text{if } y_n \sim \phi(x_n)
    \end{cases}
    \]
    and in the latter case, for every $\omega \in \beta \mathbf N \setminus \mathbf N$, $\operatorname{Cone}_\omega(\phi, x_n, \sigma_n)$ is a bilipschitz homeomorphism, with bilipschitz constant $\kappa$ independent of $\omega$.
    \end{proof}
    
    \begin{proposition}[Characterizing quasiisometries, II]
    \label{prop:charact-qi-2}
    Let $X$ and $Y$ be geodesic metric spaces and $\phi : X \to Y$.
    If for all $(x_n, y_n) \in X^{\mathbf N} \times Y^{\mathbf N}$ and $(\sigma_n)$ a sequence of positive numbers with limit $+\infty$, either \[ \phi(\operatorname{Precone}(X,x_n, \sigma_n)) \cap \operatorname{Precone}(Y,y_n, \sigma_n) = \emptyset \] or $\operatorname{Cone}_\omega(\phi)$ is well-defined and a bilipschitz homeomorphism for all $\omega$, then $\phi$ is a quasiisometry. 
    \end{proposition}
    
    \begin{proof}
    The first hypothesis implies, for every $\sigma$, the conditions \eqref{item:coarse-embed-precones} and \eqref{eq:second-condition-precone} of Proposition \ref{prop:charact-qi-1} for $\phi$ (where the injectivity of the coned map implies \eqref{eq:second-condition-precone}).
    Similarly, the second hypothesis implies, for every $\sigma$, \eqref{item:coarse-embed-precones}, \eqref{eq:second-condition-precone} and \eqref{eq:coarse-surj-prec}.
    \end{proof}
    
    The characterization given by Proposition \ref{prop:charact-qi-2} may be summarized as follows:
        a quasiisometry is a map between metric spaces which, when photographed between any pair of asymptotic cones with equal scaling factors, is either completely undefined or induces a bilipschitz homeomorphism.

    As mentionned in the introduction, $o(r)$-bilipschitz equivalences are the maps inducing bilipschitz homeomorphisms between asymptotic cones with fixed basepoints. This is less demanding than the previous characterization. We recall Cornulier's characterization below.
    
    \begin{proposition}[Cornulier]
    \label{prop:sbe-cor}
    Let $X$ and $Y$ be pointed metric spaces. Denote by $\vert \cdot \vert$ the distance to the basepoint in $X$ and in $Y$. Let $\phi : X \to Y$.
    The following are equivalent:
    \begin{enumerate}[{\rm (\ref{prop:sbe-cor}.1)}]
        \item 
        \label{item:sbe-cor-1}
        $\phi$ is $o(r)$-bilipschitz, i.e.
        There exists $\kappa \geqslant 1$ and $v: \mathbf R_{\geqslant 0} \to \mathbf R_{\geqslant 0}$ with $\lim_{r + \infty} v(r)/r = 0$ and for every $(x,x')\in X$ and $y \in Y$,
        \begin{align*}
       -v(\vert x \vert \vee \vert x' \vert) + \frac{1}{\kappa} d_X(x,x') & \leqslant d_Y(\phi(x), \phi(x')) \\
       & \leqslant \kappa d_X(x,x') + v(\vert x \vert \vee \vert x' \vert) \\ 
        d_Y(y, \phi(x)) & \leqslant v (\vert y \vert),
\end{align*}
    \item
    \label{item:sbe-cor-2}
    For every sequence $(\sigma_n)$ of positive real numbers with $\sigma_n \to + \infty$,
    there is a well-defined, bilipschitz homeomorphism
\begin{equation}
    \tag{$\operatorname{Cone}^\bullet$}
    \operatorname{Cone}_\omega^\bullet(\phi, \sigma_n): \operatorname{Cone}_\omega^\bullet(X, \sigma_n) \to \operatorname{Cone}^\bullet_\omega(Y, \sigma_n)
       \label{eq:conep}
\end{equation}
where we recall that $\operatorname{Cone}^\bullet$ denotes the asymptotic cone with observation centers fixed at basepoint according to Definition \ref{dfn:asymptotic-cones}.
    \end{enumerate}
    \end{proposition}
    
    \begin{proof}
    This results from the combination of \cite[Propositions 2.4, 2.5, 2.9, 2.12 and 2.13]{CornulierCones11}.
    There is no sequence $\sigma_n$ in Cornulier's statement, however the formulations are easily seen to be equivalent to ours.
    \end{proof}

    In this way, the groupoids of quasiisometries and $o(r)$-bilipschitz equivalences respectively are the largest groupoids over metric spaces so that the parametrized family of functors $\operatorname{Cone}$ and $\operatorname{Cone}^\bullet$ respectively are well defined to the groupoid of metric spaces with bilipschitz homeomorphism.
    Note that when characterizing quasiisometries in Proposition \ref{prop:charact-qi-1}, we only assumed that $\phi$ has to be well defined at the level of asymptotic cones, and then the bilipschitzness of every $\operatorname{Cone}(\phi)$ came for free, with a common bilipschitz constant.

   On the other hand, it is explicitely required in Proposition \ref{prop:sbe-cor} that the map be bilipschitz through asymptotic cones.
   There is indeed a strictly larger groupoid, that of isomorphisms in the category of cone-defined maps in Cornulier's terminology, whose pictures through $\operatorname{Cone}^\bullet$ only have nonzero and finite local lipschitz and expansion constant at basepoint; see \cite[\S 2.2]{CornulierCones11} for characterizations of this category. 
    
    Let us state a refinement of (\ref{prop:sbe-cor}.\ref{item:sbe-cor-1}) $\implies$ (\ref{prop:sbe-cor}.\ref{item:sbe-cor-2}) in the last Proposition.
    
    \begin{proposition}
    \label{prop:Ou-to-cone}
    Let $X$ and $Y$ be metric spaces. Let $\phi : X \to Y$ and assume that (\ref{prop:sbe-cor}.\ref{item:sbe-cor-1}) holds for some $\kappa$ and $v$, where $v$ is admissible (Definition \ref{dfn:admissible}).
    Then for every sequence $(\sigma_n)$ of positive real numbers and for every $(x_n) \in X^{\mathbf N}$ such that $\limsup v(\vert x_n \vert) / \sigma_n = 0$,
    $\phi$ induces a bilipschitz homeomorphism
    \begin{equation}
       \operatorname{Cone}_\omega(\phi, x_n, \sigma_n): \operatorname{Cone}_\omega(X, x_n, \sigma_n) \to \operatorname{Cone}_\omega(Y, \phi(x_n), \sigma_n)
       \tag{Cone}
       \label{eq:conelast}
   \end{equation}
    \end{proposition}
    
    \begin{proof}
    This conveniently follows from \cite{KramerWeiss}, by setting for any $\rho >0$, $X_n = B(x_n, \rho \sigma_n)$, $t_n = v((1+\rho) \vert x_n \vert)$ and $\phi_n = \phi_{\mid X_n}$. Since $t/\sigma$ is infinitesimal, by \cite[Lemma 1.16]{KramerWeiss} the sequence $\phi_n$ defines $\phi_\omega$ between the ultralimits of the spaces $X_n/\sigma_n$, namely, the ball of radii $\rho$ in the asymptotic cones.
    \end{proof}

    In Proposition \ref{prop:Ou-to-cone} the assumption that $v$ be admissible is necessary. Otherwise $t_n$ may not be negligible when compared to $\sigma_n$, which is necessary assumption so that the sequence $t_n/\sigma_n$ defines an infinitesimal number in the real field $\prod_\omega \mathbf R$ for every ultrafilter $\omega$.
    
    As an application, we can now distinguish the nonhomogeneous spaces from Examples \ref{exm:riemannian-planes}:
    \begin{itemize}
        \item 
        None of $P_1$, $P_2$, $P_3$ is $o(r)$-bilipschitz to $P_4$ since $\dim \operatorname{Cone}^\bullet_\omega(P_i)$ is $1$ for $i=1,2,3$ and $2$ for $i=4$.
        \item
        $P_2$ and $P_3$ are $O(\log)$-bilipschitz through the identity map in polar coordinates, but they are not $O(\log^{1-\epsilon})$-bilipschitz equivalent for any $\epsilon >0$, since $\dim \operatorname{Cone}_\omega(P_2, x_n, n) = 1$ and $\dim \operatorname{Cone}_\omega(P_3, x_n, n) = 2$ if $\vert x_n \vert = e^n$ (See Table \ref{tab:riemannian-planes}) and $\log (e^{n})^{1-\epsilon} \ll n$.
        \item
        $P_1$ and $P_2$ are quasiisometric; however they are not $O(u)$-bilipschitz equivalent for $u \to 0$.
    \end{itemize}

\begin{table}[t]
    \centering
    \begin{tabular}{|c|c|c|c|c|}
    \hline
    & \small{$\operatorname{Cone}_\omega(P_1)$}
    & \small{$\operatorname{Cone}_\omega(P_2)$}
    & \small{$\operatorname{Cone}_\omega(P_3)$}
    & \small{$\operatorname{Cone}_\omega(P_4)$} \\
    \hline
         bounded $x_n$, $\sigma_n \equiv 1$ & $(P_1, x_\omega)$ & $(P_2, x_\omega)$ & $(P_3, x_\omega)$ & $(P_4, x_\omega)$  \\
         bounded $x_n$, $\sigma_n \to + \infty${} & $\mathbf R_{\geqslant 0}$ & $\mathbf R_{\geqslant 0}$ & $\mathbf R_{\geqslant 0}$  & $(C,0)$  \\
         bounded $x_n$, $\sigma_n \to 0$ & $\mathbf E^2$ & $\mathbf E^2$ & $\mathbf E^2$ & $\mathbf E^2$  \\
         $\vert x_n \vert = n, \sigma_n = 1/n$ & $S^1 \times \mathbf R$ & $\mathbf E^2$ & $\mathbf E^2$  & $\mathbf E^2$  \\
         $\vert x_n \vert = n, \sigma_n = 1$ & $\mathbf R$ & $S^1 \times \mathbf R$ & $\mathbf E^2$  & $\mathbf E^2$  \\
         $\vert x_n \vert = n, \sigma_n = n$ & $\mathbf R_{\geqslant -1}$ & $\mathbf R_{\geqslant -1}$ & $\mathbf R_{\geqslant -1}$  & $(C,i)$  \\
         $\vert x_n \vert = e^n, \sigma_n = 1/n$ & $\mathbf R$ & $\mathbf E^2$ & $\mathbf E^2$  & $\mathbf E^2$  \\
         $\vert x_n \vert = e^n, \sigma_n = 1$ & $\mathbf R$ & $S^1 \times \mathbf R$ & $\mathbf E^2$  & $\mathbf E^2$  \\
         $\vert x_n \vert = e^n, \sigma_n = n$ & $\mathbf R$ & $\mathbf R$ & $S^1 \times \mathbf R$  & $\mathbf E^2$  \\
         $\vert x_n \vert = e^n, \sigma_n = n^2$ & $\mathbf R$ & $\mathbf R$ & $\mathbf R$  & $\mathbf E^2$  \\
         \hline
    \end{tabular}
    \vskip 10pt
    \caption{Various cones on the Riemannian planes $P_i$ from Example \ref{exm:riemannian-planes}. 
    We provide the cones as pointed metric spaces (on the second line they do not depend on $\sigma_n$ as soon as it goes to $+\infty$). Here $C$ denotes $\lbrace z \in \mathbf C : \Im z \geqslant 0 \rbrace /(x \sim -x)$ with the distance induced from the absolute value.
    }
    \label{tab:riemannian-planes}
\end{table}

\subsubsection{On cone dimension}
\label{subsubsec:unique}

We have seen that the covering dimension of (moving) cones is an efficient tool to discriminate between the Examples \ref{exm:riemannian-planes} up to quasiisometry or $O(u)$-bilipschitz equivalence.  
When $X$ is co-boundedly acted upon, however (which is one case of interest for geometric group theorists) all its asymptotic cones are isometric once the ultrafilter is fixed. 
Hence, computing $\dim \operatorname{Cone}_\omega$ for fixed $\omega$ will provide the same information with respect to QI or SBE.

Beyond geometric models of polynomially growing groups $G$, it should not be expected that different ultrafilters will yield isometric or even just homeomorphic asymptotic cones; an extensive litterature and even the notion of lacunary hyperbolic group on its own have been built over this distinction (\cite{VT}, \cite{KSTT}, \cite{OOS}).
If $G$ is a simply connected, completely solvable Lie group with a completely solvable $\mathfrak g$, nevertheless, then for every geometric model $X$, $\omega \in \beta \mathbf N \setminus \mathbf N$ and $\sigma_n$ with $\lim_{\sigma_n} = +\infty$, 
\begin{equation}
    \label{eq:cornlier-formula}
    \dim \operatorname{Cone}^\bullet_\omega(X,\sigma_n) = \dim G^{\mathrm{nil}}
    \tag{conedim}
\end{equation}
where $G^{\mathrm{nil}}$ is the largest nilpotent quotient of $G$ \cite{CornulierDimCone}.
Following Cornulier we denote this integer $\operatorname{conedim}$. This is the first, and perhaps the most natural numerical SBE invariant. 

In the special case when $G$ is nilpotent, \eqref{eq:cornlier-formula} follows from the earlier construction of Pansu, which can be formulated in terms of Gromov-Hausdorff convergence with no reference to a ultrafilter \cite{PanCBN}. Beware that this limit is not functorial, however.

When no homogeneity assumption is made, the dimension of the asymptotic cone (even with fixed basepoint) depends not only on the ultrafilter but also on the scaling sequence.
One encounters four-dimensional complete Riemannian spaces with positive Ricci curvature and $\operatorname{SU}(2)$ symmetry, for which the covering dimension of the asymptotic cones can be $2$ or $4$ depending on how one chooses the scaling factors  \cite{PerelmanCone}. These cones are genuine rescaled Gromov-Hausdorff limits, obtained without passing to a subsequence and thus do not depend on the ultrafilter.


\subsection{Coarse structures}
\label{subsec:coarse-structures}
In the 1930s, Weil abstracted the notion of a uniform structure from the topology of locally compact groups. 
Coarse structures are large-scale counterparts of uniform structures; they were introduced by Roe in the 1990s.
We recall below the definition of a coarse space.

Let $X$ be a set. The square $X \times X$ is a groupoid for the composition law $(x_0, x_1) \circ (x_1, x_2) = (x_0, x_2)$ and $(x_0, x_1)^{-1} = (x_1, x_0)$ for $x_0, x_1, x_2 \in X$. For $E, F \subseteq X \times X$, define $E \circ F = \lbrace e \circ f : e \in E, f \in F\rbrace$ and $E^{-1} = \lbrace e^{-1} : e\in E \rbrace$. 

\begin{definition}[{\cite[Definition 2.3]{RoeLectures}}]
A collection $\mathcal{E} \subseteq \mathfrak P(X \times X)$ is called a coarse structure if it contains the diagonal $\Delta_{X \times X}$, is stable by composition, inverse, taking subsets, and taking finite unions; the subsets $E \in \mathcal{ E}$ are called entourages.
\end{definition}

A coarse structure $\mathcal{E}$ is called monogenic if it is generated by a single entourage, that is if there exists $E \in \mathcal E$ such that $\mathcal E$ is smallest among all coarse structures containing $E$. Note that this notion has no analog among uniform structures.

\begin{definition}[Coarse equivalence]
\label{def:coars-equivalence}
Given two coarse spaces $(X, \mathcal{E}_X)$ and $(Y, \mathcal E_Y)$ and a map $\phi : X \to Y$, we say that $\phi$ is coarse if
\begin{enumerate}[(\ref{def:coars-equivalence}.a)]
    \item 
    \label{item:Coarse1}
    for all $B \subseteq Y$, $B \times B \in \mathcal{E}_Y \implies \phi ^{-1}(B) \times \phi^{-1}(B) \in \mathcal{E}_X$ and
    \item
    \label{item:Coarse2}
    for all $E\in \mathcal E_X$, $(\phi \times \phi) (E) \in \mathcal{E}_Y$, where $\phi \times \phi(x,y) = (\phi(x), \phi(y))$.
\end{enumerate}
A pair of coarse maps $\lbrace \phi: X \to Y$, $\psi: Y \to X \rbrace$ realizes a coarse equivalence if the graphs of $\phi \circ \psi$ and $\psi \circ \phi$ are both contained in entourages of the coarse structures.
\end{definition}

\begin{proposition}
[$O(u)$-coarse structure, $o(v)$-coarse structure]
\label{prop:Ou-coarse-structure}
Let $u: [0, +\infty) \to (0,+\infty)$ be a an admissible function, let $v$ be either an admissible function or $v(r) = r$, and let $(X, d_X)$ be a metric space.  
Given some $o \in X$, define
\begin{align}
    \mathcal{E}^{O(u)} & = \left\{ E \subseteq X \times X: \exists M,\, \limsup_{(x,x') \in E} \frac{d_X(x,x')}{u(\vert x \vert)} \leqslant M  \right\}
    \label{eq:EOu}
    \\
    \mathcal{E}^{o(v)} & = \left\{ E \subseteq X \times X: \limsup_{(x,x') \in E} \frac{d_X(x,x')}{v(\vert x \vert)} =0 \right\}
    \label{eq:Eov}
\end{align}
where $\vert x\vert = d_X(o,x)$ and $\limsup$ are taken as $(x,x')$ evades every bounded set fixed in advance (for the sup distance in $X \times X$). 
$\mathcal{E}^{O(u)}$ and $\mathcal E^{o(v)}$ define coarse structures on $X$. 
\end{proposition}

The bounded coarse structure is $\mathcal{E}_X^{O(1)}$, and the coarse equivalences between metric spaces equipped with $\mathcal{E}_X^{O(1)}$ are the coarse equivalences as defined in \eqref{eq:coarse-equiv}. Wright's $c_0$ coarse structure is $\mathcal{E}^{o(1)}$ \cite[Definition 1.1]{WrightC0scalar}.
Dranishnikov and Smith's sublinear coarse structure is $\mathcal{E}^{o(r)}$ (See \S\ref{subsec:assouad-nagata}) \cite{DranishnikovSmith}.

\begin{proof}
We need to check Roe's axioms. In view of \eqref{eq:EOu} and \eqref{eq:Eov} it is clear that $\mathcal{E}_X^{O(u)}$ and $\mathcal{E}_X^{o(v)}$ are closed under finite union and taking subsets. Possibly left nonobvious is the stability when taking inverses and composing.
\par {\em Inverses.}
Fix a basepoint $o$ and take a sequence $x_n, x'_n$ such that 
$\sup(\vert x_n\vert, \vert x'_n \vert) \to +\infty$, with $d_X(x_n, x'_n) \leqslant K u(\vert x_n \vert)$ for some $K \geqslant 0$ when $n$ is large enough, resp. $d_X(x_n, x'_n) \leqslant k_n v(\vert x_n \vert)$ where $k_n \to 0$. 
We need to prove that $d_X(x_n, x'_n) \leqslant L u(\vert x'_n \vert)$ for some $L\geqslant  0$, resp. $d_X(x_n, x'_n) \leqslant \ell_n v(\vert x'_n \vert)$ for some $L\geqslant  0$ when $n$ is large enough.

We claim that
\begin{equation}
    \label{eq:bounded-quotients}
    0 < \liminf \frac{\vert x'_n \vert}{\vert x_n \vert} \leqslant \limsup \frac{\vert x'_n \vert}{\vert x_n \vert} < +\infty.
\end{equation}
Indeed, if it were not the case there would be a sequence $R_n$ such that for arbitrarily large values of $D$, either for arbitrarily large $n$, $\vert x_n \vert \leqslant R_n \leqslant DR_n \leqslant \vert x'_n \vert$ or for arbitrarily large $n$, $\vert x'_n \vert \leqslant R_n \leqslant DR_n \leqslant \vert x_n \vert$.  In the first case, along a sub-sequence, by the triangle inequality $\vert x'_n \vert \leqslant R_n + Ku(R_n)$ (where we may replace $u$ by $v$ and $K$ by some $k_{n_0}$ if necessary) contradicting the hypothesis that $\vert x'_n \vert \geqslant DR_n$ for $n$ large enough (observe that $\vert x'_n \vert \to +\infty$ along that sub-sequence). In the second case, again by the triangle inequality one would have $R_n \geqslant \vert x_n \vert - Ku(\vert x_n \vert)$ (or $\vert x_n \vert - k_{n_0} \vert x_n \vert$ if necessary); but the right-hand side can be assumed greater than $\vert x_n \vert/2$ for $n$ large enough if $D$ is set large enough; this is a contradiction. Now from \eqref{eq:bounded-quotients} and the property that $u$, resp. $v$ is admissible, we obtain that also 
\begin{equation}
    0 < \liminf \frac{\vert u(x'_n) \vert}{\vert u(x_n) \vert} \leqslant \limsup \frac{\vert u(x'_n) \vert}{\vert u(x_n) \vert} < +\infty \notag
\end{equation}
(resp. the same with $v$ replacing $u$), which provides the requested constant $L$ (resp. $\ell_n$) as a function of $K$ (resp. of $k_n$) and $u$, resp. $v$. 
At this point it is useful to record that we can rewrite $\mathcal E$ in a more symmetric way:
\begin{align*}
    \mathcal{E}_X^{O(u)}  =  & \left\{ E \subseteq X \times X: \right. \\
    & \left. \exists r > 0,\, \sup_{(x,x') \in E \setminus B_r(o) \times B_r(o)} {d_X(x,x')}/{(u(\vert x \vert) + u(\vert x' \vert))} < +\infty  \right\}.
\end{align*}
\par {\em Composition.}
Start assuming $u$ is nondecreasing; we will explain how to adapt the proof in case it is not the case in the end (this philosophy was alluded to after Lemma \ref{lem:admissible2adm}).
For every $K, r \geqslant 0$, introduce 
\begin{equation}
\notag
    E_K^r(X,o) = \left\{ (x,x') : \inf (\vert x \vert, \vert x' \vert) \geqslant r, \; d_X(x,x') \leqslant K(u(\vert x \vert + \vert x' \vert)) \right\}.
\end{equation}
We need to prove that for every $K,L$ there are $r,s,t$ and $\eta(K,L)$ such that
\begin{equation}
\label{eq:composing}
    E_L^s \circ E_K^r \subseteq E_{\eta(K,L)}^t.
\end{equation}
Let $(x,x'') \in E_L \circ E_K$. By definition, there exists $x' \in X$ such that $d_X(x,x') \leqslant K(u(\vert x \vert) + u(\vert x' \vert))$ and $d_X(x',x'') \leqslant L(u(\vert x' \vert) + u(\vert x'' \vert))$. \\
Set a radius $R = \sup \left\{ r \geqslant 0 : u(r) > r/(2K+1) \right\}$. We claim that
\begin{equation}
    u(\vert x' \vert) \leqslant  \sup (u(3R), u(3 \vert x \vert)).
    \label{eq:to-prove-by-exhaustion}
\end{equation}
To prove \eqref{eq:to-prove-by-exhaustion} we proceed by exhausting all the case arising from the comparison of $\vert x \vert$ and $\vert x' \vert$ with $R$.

First, note that either $\vert x' \vert \leqslant R$, or $\vert x' \vert > R$ and then $u(\vert x' \vert) \leqslant \frac{\vert x' \vert}{2K+1}$. In the second case, by the triangle inequality 
    \begin{equation}
        \vert x' \vert \leqslant \vert x \vert  + Ku(\vert x \vert) + Ku(\vert x' \vert) \leqslant \vert x \vert  + Ku(\vert x \vert) + \frac{\vert x' \vert}{2},
        \notag
    \end{equation}
    so that $\vert x' \vert \leqslant 2 \vert x \vert + 2 K u(\vert x \vert)$.
So we always have
$\vert x' \vert \leqslant \sup(R, 2 \vert x \vert + 2 Ku(\vert x \vert))$.
Since $u$ has been assumed nondecreasing, 
\begin{equation}
\notag
    u(\vert x' \vert) \leqslant \sup(u(R), u(2 \vert x \vert + 2 Ku(\vert x \vert))).
\end{equation}
Now, either $\vert x \vert \leqslant R$, in which case $u(\vert x' \vert) \leqslant \sup(u(R), u(3 \vert  x \vert)$ and \eqref{eq:to-prove-by-exhaustion} holds, or $\vert x \vert > R$ and then $2Ku(\vert x \vert) \leqslant \vert x \vert$, so $u(\vert x' \vert) \leqslant K \sup (u(3R), u(3 \vert x \vert))$: \eqref{eq:to-prove-by-exhaustion} holds as well.
We can now finish the proof using the claim. By the triangle inequality, 
\begin{align*}
d_X(x, x'') & \leqslant Ku(\vert x \vert) + (K+L) u(\vert x' \vert) + L u (\vert x'' \vert)    \\
& \leqslant (K+L) \left[ u(\vert x \vert) + \sup(u(3R), u (3 \vert x \vert) + u(\vert x'' \vert) \right]
\end{align*}
so we may set $\eta(K,L) = 2(K+L) \limsup_{r \to + \infty} u(3r) /u(r)$; then for $r$ large enough and arbitrary $s$, \eqref{eq:composing} holds.

We now return to the general case when $u$ is not assumed non-decreasing.
If $\vert x' \vert \leqslant R$ then there is a uniform bound on $\vert x \vert$.
If $\vert x' \vert > R$ then by the triangle inequality,
\[ \vert x' \vert \geqslant \vert x \vert - Ku(\vert x \vert) - Ku(\vert x' \vert) \geqslant \vert x \vert - \frac{\vert x' \vert}{2} - Ku(\vert x \vert), \]
so that $\vert x' \vert \geqslant 2 \vert x \vert /3 - 2 Ku(\vert x \vert)/3$.
As soon as $\vert x \vert \geqslant R$, $\vert x' \vert \geqslant \vert x \vert /3$.
Using the assumption that $u$ is admissible, then, $u(\vert x'\vert) \leqslant Bu(\vert x \vert)$ for some $B \geqslant 1$. Using the same line of reasonning as before, this implies \eqref{eq:composing} with $\eta(K,L) = B(K+L)$. 
\end{proof}

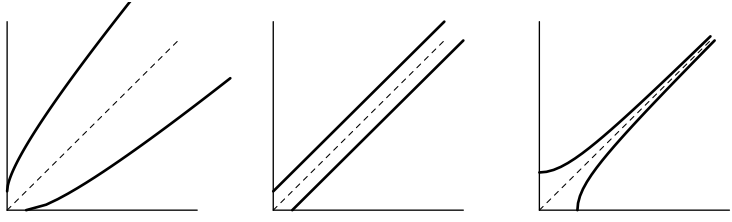
\begin{figure}
    \begin{center}
        
\begin{tikzpicture}[line cap=round,line join=round,>=angle 45,x=0.5cm,y=0.5cm]
\clip(-0.5,-0.5) rectangle (20.5,5.5);
\draw [-, line width=0.5pt] (0,0)-- (0,5);
\draw [-, line width=0.5pt] (0,0)-- (5,0);
\draw [-, line width=0.5pt] (7,0)-- (7,5);
\draw [-, line width=0.5pt] (7,0)-- (12,0);
\draw [-, line width=0.5pt] (14,0)-- (14,5);
\draw [-, line width=0.5pt] (14,0)-- (19,0);
\draw [dash pattern= on 2pt off 2pt] (0,0) -- (4.5,4.5);
\draw [dash pattern= on 2pt off 2pt] (7,0) -- (11.5,4.5);
\draw [dash pattern= on 2pt off 2pt] (14,0) -- (18.5,4.5);

\draw [line width= 1pt, variable=\x, samples = 200, domain=0:3.5] plot({\x},{\x + sqrt(\x)+0.5});
\draw [line width= 1pt, variable=\y, domain=0:3.5] plot({\y + sqrt(\y)+0.5},{\y});

\draw [line width= 1pt, variable=\x, domain=7:11.5] plot({\x},{\x-7 + 0.5});
\draw [line width= 1pt, variable=\y, domain=0:4.5] plot({7+ \y + 0.5},{\y});

\draw [line width= 1pt, variable=\x, domain=14:18.5] plot({\x},{sqrt((\x-14)*(\x-14) + 1)});
\draw [line width= 1pt, variable=\y, domain=0:4.5] plot({14+ sqrt((\y)*(\y) + 1)},{\y});

\end{tikzpicture}
    \end{center}
    \caption{Some entourages of the $O(u)$-coarse structure on the half real line $X = [0, +\infty)$, with $u(r) = \sqrt{r}, u(r) = 1$ and $u(r) = 1/r$.}
    \label{fig:my_label}
\end{figure}

\begin{proposition}
\label{prop:sbe-is-coarse}
Let $X$ and $Y$ be metric spaces. The following statements hold: 
\begin{enumerate}[(1)]
    \item 
    \label{item:Ou2coarse}
    Let $u:[0,+\infty) \to (0,+\infty)$ be an admissible function.
Let $\phi : X\to Y$ be a $O(u)$-bilipschitz equivalence. Then $\phi$ induces a coarse equivalence $(X, d_X, \mathcal{E}^{O(u)}) \to (Y, d_Y, \mathcal E^{O(u)})$.
\item
\label{item:ov2coarse}
Let $\phi : X\to Y$ be a $o(v)$-bilipschitz equivalence, where . Then $\phi$ induces a coarse equivalence $(X, d_X, \mathcal{E}^{o(r)}) \to (Y, d_Y, \mathcal E^{o(r)})$.
\end{enumerate}
\end{proposition}

\begin{proof}
Let us prove \eqref{item:Ou2coarse} first. Let $(x_n, x'_n)$ be sequences of points with $d(x_n, x'_n) \leqslant Mu(\vert x_n \vert)$ and $\vert x_n \vert \to + \infty$. Then, for $n$ large enough, $\vert x'_n \vert \leqslant 2 \vert x_n \vert$.
Hence
\begin{align*}
    d(\phi(x_n), \phi(x'_n)) & \leqslant \kappa M u(\vert x_n \vert) + cu(\vert x_n \vert \vee \vert x'_n \vert) \leqslant C(\kappa M + c) u(\vert x_n \vert)
\end{align*}
 for some $C \geqslant 1$. But also, for $n$ large enough,
 \begin{equation}
    \label{eq:phi-boundedly-proper}
     \vert \phi(x_n) \vert \geqslant \vert x_n \vert/(2\kappa).
 \end{equation}
 So there exists a constant $C'$ so that 
 $
     d(\phi(x_n), \phi(x'_n)) \leqslant C'u(\vert \phi(x_n) \vert).
 $
 On the other hand, $\phi$ has axiom (\ref{def:coars-equivalence}.\ref{item:Coarse1}) by \eqref{eq:phi-boundedly-proper}.
 This proves that $\phi$ is a coarse map.
 $\phi$ has a coarse inverse $\widetilde{\phi}$ such that $d(\widetilde{\phi} \circ \phi(x), x) \leqslant c'u(\vert x \vert)+c'$ for all and $x \in X$ and $d({\phi} \circ \widetilde{\phi}(y),y) \leqslant c'u(\vert y \vert) + c'$ for all $y\in Y$ \cite[Proposition 2.4]{cornulier2017sublinear}. So $\phi$ is a coarse equivalence from $(X, d_X, \mathcal{E}^{O(u)})$ to $(Y, d_Y, \mathcal E^{O(u)})$.
 
 Now let us turn to \eqref{item:ov2coarse}.
 Let $\phi \colon X \to Y$ be a $o(r)$-bilipschitz equivalence; this means that there exists a function $v$ and a constant $\kappa \geqslant 1$ such that $v(r) = o(r)$ and
 \[ -v(\vert x \vert \vee \vert x' \vert) + \frac{d(x,x')}{\kappa} \leqslant d(\phi(x), \phi(x')) \leqslant \kappa d(x,x') + v(\vert x \vert \vee \vert x' \vert) \]
 and $d(y, \phi(X)) \leqslant v(\vert y \vert)$ for all $x,x' \in X$ and $y \in Y$.
 Let $(x_n), (x'_n)$ be such that $\vert x_n \vert \to + \infty$ and $d(x_n, x'_n) / \vert x_n \vert \to 0$ as $n \to + \infty$. Fix $\varepsilon \in (0,1)$. For $n$ large enough, $v(2 \vert x_n \vert) \leqslant \frac{\varepsilon}{2 \kappa} \vert x_n \vert$, $d(x_n, x'_n) \leqslant \frac{\varepsilon \vert x_n \vert}{2\kappa^2} $, and $\vert \phi(x_n) \vert \geqslant \vert x_n \vert /2 \kappa$ (as in \eqref{eq:phi-boundedly-proper} in the previous case) so that
 \begin{align*}
     d(\phi(x_n), \phi(x'_n)) & \leqslant \kappa d(x_n,x'_n) + v(2\vert x_n \vert) \\
     & \leqslant \frac{\varepsilon}{2 \kappa} \vert x_n \vert + \frac{\varepsilon}{2 \kappa} \vert x_n \vert \leqslant \varepsilon \vert \phi(x_n) \vert. 
 \end{align*}
 Hence, $\phi$ is a coarse map. Again, by \cite[Proposition 2.4]{cornulier2017sublinear} there is $\widetilde \phi \colon Y \to X$ and a positive constant $c'$ such that $d(\widetilde \phi \circ \phi (x), x) \leqslant v(\vert x \vert) + c'$ and  $d(\phi \circ \widetilde \phi (y), y) \leqslant v(\vert y \vert) + c'$ for all $x \in X$ and $y \in Y$. So $\phi$ is a coarse equivalence from $(X, d_X, \mathcal{E}^{o(r)})$ to $(Y, d_Y, \mathcal E^{o(r)})$.
\end{proof}

\begin{lemma}
\label{lem:constructing-hatd}
Assume that $(X, d_X)$ is a geodesic metric space. Let $u$ be admissible and unbounded. Then 
$ E_u = \left\{ (x,x') \in X \times X: d_X(x,x') \leqslant 1 + u(\vert x \vert + \vert x' \vert) \right\} $
is a symmetric entourage generating $\mathcal{E}^{O(u)}$ on $X$. 
Define $\widehat d_X$ on $X$ such that 
\[ \widehat{d}(x,x') = \inf \left\{ n: (x,x') \in E_u^n \right\}. \]
Then, the identity map 
$ \left(X, d_X, \mathcal E^{O(u)} \right) \to \left( X, \widehat d_X, \mathcal E^{O(1)} \right) $
is a coarse equivalence.
\end{lemma}

\begin{proof}
Let us check first that $E_u$ generates $\mathcal E$. Take $E \in \mathcal{E}^{O(u)}$; then by definition
\begin{equation}
    \notag
    \sup_{(x,x') \in E} \frac{d_X(x,x')}{1+u(\vert x \vert) + u(\vert x' \vert)} = M < +\infty.
\end{equation}
For all $(x,x')$, and for every segment $\gamma : [0, d_X(x,x')] \to X$ and set $x_1 = \gamma(1+u(\vert x \vert)), x_2 = \gamma(2+u(\vert x \vert) + u(\vert x_1 \vert)), \ldots$.
Let 
\[ N_\gamma(x,x') = \inf \left\{ n : n+u(\vert x \vert) + \cdots + u(\vert x_n \vert) > d_X(x,x') \right\}. \]
We claim that $\sup_{(x,x')\in E} \inf_\gamma N < +\infty$.
Indeed, if $x$ and $x'$ are far enough there exists some constant $\mu>0$ such that $u(\vert x_k \vert) \geqslant \mu u(\vert x \vert)$ as long as $\vert x_k \vert \geqslant \vert x \vert/2$, especially as long as $k + u(\vert x \vert) + \cdots + u(\vert x_k \vert) \leqslant \vert x \vert/2$. So either $N(x,x') \leqslant \lceil M/\mu \rceil$ or $N +u(\vert x \vert) + \cdots + u(\vert x_N \vert) > \vert x \vert /2$. But in the latter case,
\begin{equation}
    \label{eq:contrad-with-N}
    M(1 +u(\vert x \vert) + u(\vert x' \vert)) \geqslant d_X(x,x') >  \frac{\vert x \vert}{2} - 1 - u(\vert x_{N} \vert)
\end{equation}
where we used the definition of $N$ on the right. To reach a contradiction, note that again by the definition of $N$, $d(x_N, x') < 1+ u(\vert x_N \vert)$, so there exists $L$ such that $d(x_N, x') \leqslant 1+ Lu(\vert x' \vert)$, reproducing the reasoning in the ``Inverse'' part of the proof of Proposition \ref{prop:Ou-coarse-structure}. Hence, there exists some constant $M'$ such that if $x'$ is far enough, $u(\vert x_N \vert) \leqslant M' u(\vert x' \vert)$. Plugging this in \eqref{eq:contrad-with-N} yields an inequality of the form $u(\vert x' \vert) + u(\vert x \vert) \geqslant \rho \vert x \vert$ for some $\rho >0$, which can only occur if $\vert x \vert$ is close to the origin.
We conclude that $E \subseteq E_u^{N_{\max}}$, where $N_{\max} = \sup_{(x,x')\in E} \inf_\gamma N$ is a finite integer.

This proves that $(X, d_X, \mathcal E^{O(u)} ) \to ( X, \widehat d_X, \mathcal E^{O(1)} )$ has the axiom (\ref{def:coars-equivalence}.\ref{item:Coarse2}) of a coarse map. In order to check (\ref{def:coars-equivalence}.\ref{item:Coarse1}) we must prove that if $B \times B$ is in $\mathcal{E}^{O(u)}$ then $B$ is bounded; fixing $x \in B$, by \eqref{eq:EOu}, for any sequence $x'_n$ that escape to infinity $x'_n$ cannot stay in any entourage of $\mathcal{E}^{O(u)}$ fixed in advance.
Conversely, if $B$ is bounded then $B \times B$ is in $\mathcal{E}^{O(u)}$, while axiom (\ref{def:coars-equivalence}.\ref{item:Coarse2}) holds for $(X, \widehat d_X, \mathcal E^{O(1)} ) \to ( X, d_X, \mathcal E^{O(u)} )$ by definition of $\widehat d$.
\end{proof}

The new distance $\widehat d_X$ may be made geodesic as well, by adding metric edges between pairs of point at distance $1$. Note however that one may loose properness in this process.

If $(X,d)$ has an isometric group action, this group action will not be an isometric group action for  $(X,\widehat d_X)$ neither. In fact the main interest of $\widehat d_X$ is theoretical, and appears in the next Proposition.

Say that a map $\phi: X \to Y$ between pointed metric spaces is radial if there exists $\kappa \geqslant 1$ and $R, R' \geqslant 0$ such that for all $x \in X$,
\begin{equation}
    \frac{1}{2 \kappa} \sup(R, \vert x \vert) \leqslant \sup(R', \vert \phi(x) \vert) \vert \leqslant 2 \kappa \sup(R, \vert x \vert).
\end{equation}
Also, call discrete geodesic between $x$ and $x'$ at distance $n$ in $X$ a finite sequence of points $x_i$ with $x=x_0$, $x_n=x'$ and $d(x_i, x_{i+1}) = 1$.

\begin{proposition}
\label{prop:fromOucoasretoSBE}
Let $X$ and $Y$ be geodesic metric spaces, and let $\phi : X \to Y$ be a $O(\log)$-coarse equivalence. Then
\begin{enumerate}[{\rm (1)}]
    \item
    \label{item:phi-radial}
    $\phi$ is radial.
    \item 
    \label{item:phi-is-SBE}
    $\phi$ is a $O(\log)$-bilipschitz equivalence.
\end{enumerate}
\end{proposition}

We need a preliminary Lemma.
\begin{lemma}
\label{lem:s-and-t}
Let $t$ and $s$ be positive real numbers.
Then for every $M > 0$, there exists $R\geqslant 1$ and $M' > 0$ such that
\[ 
\begin{cases}
 \frac{t}{\log t} & \leqslant M \frac{s}{\log s} \\
 \inf (s,t) & \geqslant R
\end{cases}
\implies t \leqslant M's
\]
\end{lemma}

\begin{proof}
We will prove first a weaker inequality and then self-improve it.
Taking logarithms on both sides we get $\log t - \log \log t \leqslant \log M +  \log s - \log \log s$, so for every $\varepsilon > 0$ one has, for $s$ and $t$ large enough, $(1-\varepsilon/2) \log t \leqslant (1+ \varepsilon/2) \log s$, and then $t \leqslant s^{1+\varepsilon}$. Now, assume by contradiction that there is a sequence $(s_n, t_n)$ with $t_n / \log t_n \leqslant M s_n / \log s_n$, but $q_n = t_n /s_n$ going to infinity. Then $t_n/\log t_n = t_n /(\log s_n + \log q_n)$; but we know that $\log q_n \leqslant \varepsilon \log s_n$; so $t_n / \log s_n \leqslant M' s_n / \log s_n$ for some $M'$, reaching the desired inequality.
\end{proof}

\begin{proof}[Proof of the Proposition \ref{prop:fromOucoasretoSBE}]
Consider the metrics $\widehat d_X$ and $\widehat d_Y$ provided by Lemma \ref{lem:constructing-hatd} on $X$ and $Y$.
Then $\phi : (X, \widehat d_X) \to (Y, \widehat d_Y)$ becomes a $O(1)$-coarse equivalence. Since $\widehat d_X$ and $\widehat d_Y$ are geodesic, $\phi$ is a $\widehat d$-quasiisometry, especially it is $\widehat d$-radial.

Now, we need to compare $\widehat d$ and $d$.
Start with \eqref{item:phi-radial}; for this we need to compare $\vert x \vert$ and $\widehat d(0,x)$ for all $x \in X$.
Let $(x_n)$ be a discrete $\widehat d$-geodesic segment from $o$ (we do not specify an endpoint yet).
We claim that $\vert x_n \vert \leqslant 2n \log n + 2n$ for $n >0$.
Let us proceed by induction on $n$. This holds for $n=1$.
Assume it holds for some $n >0$.
Then, 
\begin{align*}
    \vert x_{n+1} \vert & = \vert x_n \vert + d(x_n, x_{n+1}) \\
    & \leqslant \vert x_n \vert + 1 + \log (\vert x_n \vert) \\
    & \leqslant 2n + 2n\log n + 1 + \log 2 + \log n + \log (1 + \log n) \\
    & \leqslant 2n+2n\log n + 2 + 2 \log n \\
    & = (2n+2) + (2n+2) \log n \leqslant (2n+2) + (2n+2) \log (n+1)
\end{align*}
where we used $\log 2 < 1$ and $\log n \leqslant n-1$.
Using this inequality, we deduce
\begin{align}
    \widehat{d}_X(o,x) \geqslant \inf \left\{ n : 2n(1+\log n) \geqslant
    \vert x \vert \right\} \geqslant \frac{\vert x \vert}{1+3 \log \vert x \vert}
    \label{eq:lower-estimate-on-hat-distance-to-origin}
\end{align}

Conversely, repeating a construction made in the proof Lemma \ref{lem:constructing-hatd}, consider a geodesic segment $\gamma : [0, \vert x \vert] \to X$, and a sequence \[ x_0 = o, \, x_1 = \gamma(2), x_2 = \gamma(1+\log \vert x_1 \vert), \ldots x_{i+1} = \gamma(\vert x_i \vert + \log \vert x_i \vert) \] 
 and define $N$ such that $x_N$ is the farthest element from $o$ before reaching $x$; in this way, $\widehat d_X(o,x) \leqslant N+1$. By induction on $n$, we can prove that $\vert x_n \vert \geqslant n \log n$ for all $n$. So
 \begin{equation}
     \widehat d_X(o,x) \leqslant 1+ \inf \left\{ n : n \log n \geqslant
    \vert x \vert \right\} \leqslant 1 + \frac{\vert x \vert}{1+ \log \vert x \vert}. 
    \label{eq:upper-estimate-on-hat-distance-to-origin}
 \end{equation}
We are now ready to prove \eqref{item:phi-radial}. We know that $\phi$ is $(\widehat d_X, \widehat d_Y)$-radial ; so there exists $\kappa_0$ such that
\begin{equation}
 \frac{\vert \phi(x) \vert}{1+ 3 \log \vert \phi(x) \vert} \leqslant 
    \widehat{d}_Y(o, \phi(x)) \leqslant 2 \kappa_0 \left( 1+ \frac{\vert x \vert}{1+\log \vert x \vert} \right)
\end{equation}
Combining both inequality, $\vert \phi(x) \vert$ and $\vert x \vert$ satisfy the hypotheses of $t$ and $s$ in Lemma \ref{lem:s-and-t}. We conclude from the Lemma that $\phi$ is radial.

The proof of \eqref{item:phi-is-SBE} will now rely on \eqref{item:phi-radial} together with an estimate akin to \eqref{eq:lower-estimate-on-hat-distance-to-origin} and \eqref{eq:upper-estimate-on-hat-distance-to-origin}, but where we replace $o$ with $x' \in X$.
Let $x, x' \in X$; assume $2 \leqslant \vert x \vert \leqslant \vert x' \vert$, and let $\gamma$ be a geodesic segment from $x$ to $x'$. Define $x_0 = x$, $x_{i+1} = \gamma (d(x_0, x_i) + 1 + \log \vert x_i \vert)$ as long as it makes sense (let $n$ be the largest one, so that $x_n$ is the closest to $x'$ among all $x_i$'s).
By the triangle inequality, for all $i$ such that $0 \leqslant i \leqslant n$, 
\begin{equation}
    \notag
    \vert x_i \vert \leqslant \vert x \vert + d(x, x_i) \leqslant \vert x \vert + d(x,x') \leqslant 2 \vert x' \vert + \vert x \vert \leqslant 3 \vert x' \vert.
    \label{eq:xi-leq-3x}
\end{equation}
From this inequality, we deduce that 
\begin{equation}
    \notag
    \widehat d_X(x,x') \geqslant \frac{d(x,x')}{ 2\log (3\vert x' \vert)} \geqslant \frac{d(x,x')}{4 \log \vert x' \vert}.
\end{equation}
Conversely, if $\inf_t \vert \gamma(t) \vert \leqslant \vert x' \vert/2$, then $d(x,x') \geqslant \vert x' \vert /2$. So
    \begin{align*}
        \widehat d_X(x,x') & \leqslant \widehat d_X(x,o) + \widehat d_X(o,x')  \leqslant 2 +  \frac{2 \vert x' \vert}{1 + \log \vert x' \vert} 
        \leqslant 2 +  \frac{4 d(x,x')}{1 + \log \vert x' \vert}.
    \end{align*}
    Otherwise, $\inf_t \vert \gamma(t) \vert > \vert x' \vert /2$, and then
   $\widehat d_X(x,x') \leqslant \frac{d(x,x')}{\log(\vert x' \vert /2)}$.
Combining the previous inequalities, we get that for every pair $x,x'$ with $\sup(\vert x \vert, \vert x' \vert)$ large enough,
\begin{equation}
    \label{eq:compare-d-dhat}
    \frac{1}{\lambda_X} \frac{d_X(x,x')}{\log (\sup(\vert x \vert, \vert x' \vert)}
    \leqslant \widehat d_X(x,x') 
    \leqslant \lambda_X \frac{d_X(x,x')}{\log (\sup(\vert x \vert, \vert x' \vert)}
\end{equation}
for some $\lambda_X >1$. A similar inequality holds for pairs of points in $Y$, with a multiplicative factor $\lambda_Y$.
We are ready to finish the proof. Assume that $\phi$ is a $(\kappa_0, c_0)$ quasiisometry with respect to $\widehat d_X$ and $\widehat d_Y$.
Then
\begin{equation}
\notag
     - c_0 + \frac{1}{\kappa_0} \widehat d_X(x,x') \leqslant \widehat d(\phi(x), \phi(x')) \leqslant \kappa_0 \widehat d_X(x,x') + c_0
\end{equation}
for all $x, x'$. ing $\lambda = \sup(\lambda_X, \lambda_Y)$ and using \eqref{eq:compare-d-dhat} and its counterpart in $Y$,
\begin{equation}
\notag
    - c_1 + \frac{1}{\lambda^2 \kappa_0} \frac{d_X(x,x')}{\log \sup(\vert x\vert, \vert x' \vert)} \leqslant \frac{d_Y(\phi(x), \phi(x'))}{\log(\sup(\vert \phi(x) \vert,  \vert \phi(x'))) \vert} \leqslant \lambda^2 \kappa_0 \frac{d_X(x,x')}{\log \sup(\vert x\vert, \vert x' \vert)} + c_1
\end{equation}
for some $c_1 \geqslant 0$.
Using that $\phi$ is radial, we know that $\vert \phi(x) \vert$ and $\vert \phi(x') \vert$ are within linear control from $\vert x \vert$ and $\vert x' \vert$. So we may rewrite the previous estimate as
\begin{equation}
\notag
    - c_2 + \frac{1}{\kappa_1} \frac{d_X(x,x')}{\log \sup(\vert x\vert, \vert x' \vert)} \leqslant \frac{d_Y(\phi(x), \phi(x'))}{\log(\sup(\vert x \vert,  \vert x' \vert)} \leqslant  \kappa_1 \frac{d_X(x,x')}{\log \sup(\vert x\vert, \vert x' \vert)} + c_2
\end{equation}
where $\kappa_1 \geqslant 1$ and $c_2 \geqslant 0$. Multiplying by $\log \sup(\vert x \vert,  \vert x' \vert)$ on both sides yields the required \eqref{eq:sbe-1}.
\end{proof}

\begin{remark}
The assumption $u =\log$ made in Proposition \ref{prop:fromOucoasretoSBE} is possibly too strong. 
On the other hand, it is not true that every coarse equivalence between $o(r)$-coarse structure is a $o(r)$-bilipschitz equivalence: consider $\phi : \mathbf R^n \to \mathbf R^n$ such that $\phi(x) = \Vert x \Vert x$. A notable distinction between $\mathcal{E}^{O(\log)}$ and $\mathcal{E}^{o(r)}$ is that the former is monogenic whereas the latter is not.
Also, observe that Lemma \ref{lem:s-and-t} breaks down for $u(t)= t^e$, $e >0$. 
\end{remark}

\subsection{Invariance of the geometric dimension for connected Lie groups}
\label{subsec:assouad-nagata}

\begin{definition}[sublinear Higson function]
Let $X$ be a proper metric space.
Define the ${}^\star$-algebra $C_{h_L}(X)$ of sublinear Higson functions on $X$ as 
\[ \left\{ f \in C_b(X, \mathbf C) \;: \,  \forall E \in \mathcal E^{o(r)}, \lim_{r \to + \infty} \sup_{(x,x') \in E, \inf (\vert x \vert, \vert x' \vert) \geqslant r} \vert df (x,x') \vert = 0 \right\} \]
where $f \in C_b$ means that $f$ is continuous, $\sup \vert f \vert < + \infty$ and $df(x,x') = f(x) - f(x')$.
\end{definition}

\begin{figure}

\begin{tikzcd}
\nu X \arrow[d, two heads] & \text{Higson corona} \\
\nu_L X \arrow[d, two heads] & \text{sublinear Higson corona} \\
\partial_\infty X & \text{Gromov boundary}
\end{tikzcd}

\caption{Coronae and Gromov boundary for hyperbolic $X$.}

\label{fig:coronae}
\end{figure}
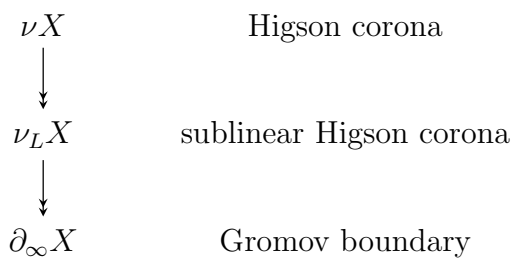

\begin{remark}[Compare Fukaya \cite{Fukaya}, 3.1]
$f$ is Higson sublinear if and only if there exists $C_f < + \infty$ such that 
for all $x, x'$ in $X$ and $R >0$ large enough, if $\inf(\vert x \vert, \vert x' \vert) \geqslant R$ and $d_X(x,x') \leqslant R/2$, then $\vert f(x) - f(x') \vert \leqslant \frac{C_f}{R}$. 
\end{remark}

The closure $\overline{C_{h_L}(X)}$ is a unital $\mathrm C^\star$-algebra; the sublinear Higson corona $\nu_L X$ of $X$ is the spectrum of  $\overline{C_{h_L}(X)}$ modded out by the ideal of functions vanishing at infinity \cite[Definition 2.37]{RoeLectures}.

\begin{remark}[See Figure \ref{fig:coronae}]
If $X$ is a proper, geodesic, Gromov-hyperbolic space with basepoint $o$, say that $f : X \to \mathbf C$ is a Gromov function if it is continuous, bounded, and for every $\varepsilon > 0$ there exists $K > 0$ such that $(x\mid x')_o > K \implies \vert f(x) - f(x')\vert < \varepsilon$.
The Gromov functions on $X$ are Higson sublinear, and the Higson sublinear functions are Higson functions in the classical sense. It follows that the sublinear Higson corona sits in between the Higson corona $\nu X$ and the Gromov boundary $\partial_\infty X$ seen in the topological category.
\end{remark}

The following is a generalization of \cite[Proposition 2.1]{DranishnikovSmith}.

\begin{proposition}
\label{prop:SBE-to-corona}
Let $X$ and $Y$ be metric spaces. Let $\nu_L X$ and $\nu_L Y$ be their sublinear Higson coronae.
Then, any $o(r)$-bilispchitz equivalence $f : X \to Y$ induces a homeomorphism
$\nu_L f : \nu_L X \to \nu_L Y$.
\end{proposition}

\begin{proof}
By Proposition \ref{prop:sbe-is-coarse}, a $o(r)$-bilipschitz equivalence $X \to Y$ represents a coarse equivalence $(X, d_X, \mathcal{E}^{o(r)}) \to (Y, d_Y, \mathcal{E}^{o(r)})$, and then induces a homeomorphism between the sublinear Higson coronae \cite[Corollary 2.42]{RoeLectures}. 
\end{proof}

\begin{theorem}[{\cite[Theorem 3.10 and Corollary 3.11]{DranishnikovSmith}}; see also {\cite{DydakLipExt}}]
\label{th:DS}
Let $X$ be a proper connected metric space. Assume that $\operatorname{Isom}(X)$ is co-compact on $X$, and that $\operatorname{asdim}_{\mathrm{AN}}(X) < + \infty$.
Then 
\begin{equation}
    \dim \nu_L X = \operatorname{asdim}_{\mathrm{AN}}(X).
\end{equation}
\end{theorem}

\begin{theorem}[{\cite[Theorem 7.9]{HigesPeng}}]
\label{th:HP}
Let $G$ be a connected Lie group, and let $X$ be any geometric model of $G$. Then 
\begin{equation}
    \operatorname{asdim}_{\mathrm{AN}}(X) = \dim G - \dim K.
\end{equation}
where $K$ is any maximal compact subgroup of $G$.
\end{theorem}

Theorem \ref{th:geodim} from the introduction now follows by combining Proposition \ref{prop:SBE-to-corona} with Theorems \ref{th:DS} and \ref{th:HP}.

To the best of the author's knowledge, the only connected Lie group for which some description of the sublinear Higson corona is currently available is $\mathbf R^n$: Fukaya proved that $\nu_L \mathbf R^n \simeq S^{n-1} \times \nu_L \mathbf R$ \cite{Fukaya}. These spaces are ``big'' and not metrizable, so it seems not easy to extract fine topological invariants from them as one would do for, say, the Gromov boundary.

\begin{ques}
Let $X$ be a proper metric space. Is the Čech cohomology group $\check{H}^1(\nu_L X, \mathbf Z)$ finitely generated?
\end{ques}

The answer is known to be negative for the Higson coronae associated to bounded coarse structures \cite{keesling}; nevertheless Fukaya proves that $\nu_L \phi$ is homotopic to the identity whenever $\phi \in \operatorname{GL}(n,\mathbf R)$ has positive determinant.

\section{Real hyperbolic spaces and Theorem \ref{th:Tukia-SBE}}
\label{sec:proofB}

In this section we prove Theorem \ref{th:Tukia-SBE} on Lie groups $O(u)$-bilipschitz equivalent to real hyperbolic spaces. \S\ref{subsec:pinching} gathers preliminary results on pinching and conformal dimension, and \S\ref{subsec:degenerations} sets the terminology of degenerations and deformations. The equivalences of Theorem \ref{th:Tukia-SBE} are proved in \S\ref{subsec:proofA}.

\subsection{Heintze groups, conformal dimension and pinching}
\label{subsec:pinching}
In 1955, Jacobson proved that all real Lie algebras who possess a derivation with no purely imaginary eigenvalue are nilpotent \cite{JacobsonCNLA}. Later Heintze characterized the semidirect products of nilpotent Lie algebra by derivations whose spectrum has positive real part, as the Lie algebras of Lie groups that possess at least one negatively curved left-invariant metric (note that these are centerless) \cite{Heintze}.  Most importantly, Heintze showed that the negatively curved metrics on these groups exhaust all the isometrically homogeneous negatively curved manifolds, shedding light on the earlier result of Kobayashi that these spaces had to be simply connected \cite{KobayashiHMN}.

\begin{definition}[{\cite{CoTesContracting}}]
\label{def:heintze-type}
Let $G$ be a Lie group with finitely many components. Then $G$ is of Heintze type if there exists a simply connected nilpotent $N$, a derivation $\alpha \in \operatorname{Der}(\mathfrak n)$ with $\inf \left\{  \Re \lambda : \lambda \in \operatorname{Sp}(\alpha) \right\} > 0$ and a compact group $K$ with a representation $\rho: K \to \operatorname{Aut}(N)$ such that 
\begin{equation}
    \label{eq:heintze}
    G = (K \times \mathbf R) \ltimes N, 
\end{equation}
where $(k,t).n = \rho(k)(n)e^{\alpha t} n$ (the actions of $K$ and $\mathbf R$ do commute).
A Heintze group is a group of Heintze type with $K=1$. 
\end{definition}

By normalized Jordan form of a derivation $\alpha$ as in Definition \ref{def:heintze-type}, we mean the Jordan form of the unique positive multiple $[\alpha]$ of $\alpha$ such that 
\begin{equation}
    \label{eq:normalHeintze}
    \inf \left\{  \Re \lambda : \lambda \in \operatorname{Spec}([\alpha]) \right\} = 1.
\end{equation}
Note that $N \rtimes_\alpha \mathbf R \simeq N \rtimes_{[\alpha]} \mathbf R$ (Compare Example \ref{exm:affine}.)
The following useful fact is proved in E. Sequeira's thesis using a highest weight argument \cite[Proposition 5.2.2]{Sequeirathesis}\footnote{\cite{Sequeirathesis} has the assumption that $G_\alpha$ and $G_\beta$ are purely real, but the general proof goes along the same lines.}.

\begin{proposition}
Let $N$ be a simply connected nilpotent Lie group.
If the Heintze groups $G_\alpha = N \rtimes_\alpha \mathbf R$ and $G_\beta = N \rtimes_\beta \mathbf R$ are isomorphic, then $\alpha$ and $\beta$ have the same normalized Jordan form. 
\end{proposition}

\begin{definition}[after {\cite[Section 4]{EberleinHeber}}]
\label{def:amalgam}
Given two Heintze groups $G = N \rtimes_\alpha \mathbf R$ and $G' = N' \rtimes_{\alpha'} \mathbf R$ and $\lambda >0$, we write $G \; \sharp \; (G')^\lambda = (N \times N') \rtimes \mathbf R$ where $t.n = (e^{\alpha t}, e^{\lambda \alpha' t})$ with the convention that both $\alpha$ and $\alpha'$ are normalized as in \eqref{eq:normalHeintze}, and call this group Heintze amalgam of $G$ and $G'$. Denote the Lie algebra of $\operatorname{Lie}(G \; \sharp \; (G')^\lambda )$ by $\mathfrak g \; \sharp \; \lambda \mathfrak g'$.
\end{definition}

A Heintze group is purely real if it is completely solvable, i.e. if $\operatorname{Sp}(\alpha) \subseteq \mathbf R$; every group of Heintze type has a Riemannian model in common with a purely real Heintze group, that we call its shadow (See \cite{AlekHMN} and \S\ref{subsec:nilpotent}).
If $G$, $N$, $\alpha$ are as in Definition \ref{def:heintze-type} with $K=1$ and if $\mathfrak n = \operatorname{Liespan}(\ker([\alpha] - 1))$, then we say that $G$, resp. $\mathfrak g$ is a Carnot-type Heintze group, resp. algebra. In this case  isomorphism type of $G$ does not depend on $\alpha$, so we abbreviate $G = N \rtimes_{\mathrm{Carnot}} \mathbf R$ \cite[Proposition 3.5]{CrnulierSystolicGrowth}. Carnot-type Heintze groups are purely real.

\begin{example}
\label{exm:bnK}
Let $\mathbf K$ be a division algebra over $\mathbf R$ and $n$ a positive integer, $n=2$ if $\mathbf K = \mathbf {Ca}$. $\mathfrak b(n,\mathbf K)$ is the solvable Lie algebra over the vector space $V = \mathbf K^{n-1} \oplus \Im \mathbf K \oplus \mathbf R$ (where $\Im \mathbf K = 0$ if $\mathbf K = \mathbf R$) with Lie bracket
\begin{equation*}
    \left[ (z_i, \tau, s), (z'_i, \tau', s')  \right] =\left[ sz_i'- s'z_i, 2s\tau' - 2s'\tau + \sum_{i=1}^{n-1} \Im(z_i \overline{z'_i}), 0 \right].
\end{equation*}
$\mathfrak b(n, \mathbf K)$ for $\mathbf K = \mathbf R, \mathbf C, \mathbf H$ is the maximal completely solvable subalgebra of $\mathfrak o(n,1)$, $\mathfrak u(n,1)$, $\mathfrak {sp}(n,1)$ respectively.
\end{example}

The Heintze groups with Lie algebra $\mathfrak b(n, \mathbf K)$ are exactly those who carry (rank one) symmetric metrics \cite{Heintze} (for $\mathbf K= \mathbf R$, all the left-invariant metrics are symmetric, see e.g. \cite{LauretDegenerations}).

The topological dimension $\operatorname{Topdim} \partial_\infty$ and conformal dimension $\operatorname{Cdim} \partial_\infty$ are quasisometry invariant of Gromov-hyperbolic locally compact compactly generated groups (\cite{MTconfdim}, \cite{CCMT}). For a group of Heintze type $G = (K \times \mathbf R) \ltimes_\alpha N$,
\begin{align}
    \label{Topdim-Heintze}
    \operatorname{Topdim} \partial_\infty G & = \dim G - \dim K - 1 = \operatorname{geodim} G - 1; \\
    \label{Cdim-Heintze}
    \operatorname{Cdim} \partial_\infty G & = \operatorname{Tr} [\alpha] 
\end{align}
where $\operatorname{Tr}$ denotes the trace.
Though not explicitly stated there, the following is a direct consequence of \cite[Section 5]{PansuDimConf}. 

\begin{theorem}[After Pansu]
\label{th:pansuconf}
Let $(M,g)$ be a complete, simply connected Riemannian manifold of dimension $n \geqslant 2$. Let $b \geqslant 1$. Assume that $M$ is $-1/b^2$-pinched, i.e. (up to normalization of $g$) $- b^2 \leqslant K_g \leqslant -1$. Then 
\begin{equation}
    \label{eq:Cdim-pinching}
    \operatorname{Cdim} \partial_\infty M \leqslant (n-1)b.
\end{equation}
\end{theorem}

\begin{proof}
It follows from the lower bound on sectional curvature that $\operatorname{Ric} \geqslant (n-1)b^2g$. Then, by the Bishop-Gromov inequality 
\begin{equation*}
    \operatorname{vol}(B(x,r)) \leqslant \operatorname{cst.} \int_0^r \sinh^{{n-1}}(b t) dt,
\end{equation*}
so that the volume-theoretic entropy $h = \limsup_{r \to + \infty} r^{-1} \log \operatorname{vol}(B(x,r))$ is bounded above by $(n-1)b$.
Pansu proves $\operatorname{Cdim} \partial_\infty S \leqslant h$ \cite[Lemme 5.2]{PansuDimConf}.
Combining these inequalities yields the desired \eqref{eq:Cdim-pinching}.
\end{proof}

\begin{corollary}
Let $G$ be a group of Heintze type; then every Riemannian model of $G$ has a pinching of at least 
\begin{equation}
    \label{eq:pansu-bound}
    - \left( \frac{\operatorname{geodim} G -1}{\operatorname{Tr}[\alpha]} \right)^2.
\end{equation}
\end{corollary}

The bound \eqref{eq:pansu-bound} is not optimal. Building on a theorem of Belegradek and Kapovitch and curvature computations, Healy determined the exact optimal pinching (which is attained) when $G$ is Carnot-type and $N$ has a lattice (equivalently, when $\mathfrak n$ has a $\mathbf Q$-form) and found
an optimal pinching of $-1/s^2$, where $s$ is the nilpotency step of $N$ \cite[Theorem 4.3]{healy2021pinched}. Note that for Carnot type groups, $s$ is the spectral radius of $[\alpha]$ so $\operatorname{Tr}[\alpha] \leqslant s(\operatorname{Topdim} \partial_\infty G) =  s(\operatorname{geodim} G -1)$.

\begin{corollary}
\label{cor:on-pinching}
Let $G$ be a group of Heintze type. Assume that $G$ has Riemannian models with pinching arbitrarily close to $-1$. Then $\alpha$ has all its eigenvalues with the same real part, and $N$ is abelian.
\end{corollary}

\begin{proof}
Order the eigenvalues of $\alpha$ as $\sigma_1 \leqslant \cdots \leqslant \sigma_r$. In view of the formula \eqref{Cdim-Heintze} and the assumption on the pinching of $G$, Pansu's theorem forces the equality to occur in 
\[ \sigma_1 \dim \mathfrak n \leqslant  \sum_{\lambda} \Re \lambda = \operatorname{Tr}(\alpha). \]
So one may set $\sigma = \sigma_1 = \cdots = \sigma_r$, where $\sigma$ is a positive real number.
Denoting by $\mathfrak n^\lambda$ the generalized eigenspace of $\alpha$ with  eigenvalue $\lambda$,  observe  that $[\mathfrak n^\lambda, \mathfrak n^\mu] \subseteq \mathfrak n^{\lambda + \mu}$ for any complex numbers $\lambda$ and $\mu$. 
Since $\oplus_{\tau \in \mathbf R} \mathfrak n^{\sigma+i\tau} = \mathfrak n$, one has $[\mathfrak n, \mathfrak n] \subseteq \oplus_{\tau \in \mathbf R} \mathfrak n^{2\sigma+i\tau} =\{ 0 \}$, and $N$ is abelian.
\end{proof}

\begin{remark}
The conclusion that $N$ is abelian remains if a single left-invariant metric on $S$ is assumed to be strictly more than quarter-pinched,  a theorem by Eberlein and Heber, who also characterized the Heintze groups with a quarter-pinched Riemannian metric \cite{EberleinHeber}.
\end{remark}

We note that the converse of Corollary \ref{cor:on-pinching} also holds.

\begin{proposition}
\label{prop:Riemannian-computation}
Let $S=\mathbf R^{n-1}\rtimes_\alpha \mathbf R$, where $\operatorname{sp}(\alpha) \subseteq \lbrace 1 + i \tau : \tau \in \mathbf R \rbrace$.
Then, $S$ has left invariant Riemannian metrics with pinching arbitrarily close to $-1$.
Moreover, if $K$ is a compact group of automorphisms of $S$, then one can assume that those metrics are all $K$-invariant.
\end{proposition}

\begin{proof}
Let $\varepsilon >0$ be a parameter.
We consider $(e_1, \ldots ,e_{n-1})$, a basis of $\mathbf R^{n-1}$ in which $\alpha$ appears in real Jordan normal form in a definite order that we proceed to describe now. Group the generalized eigenspaces as follows:
first the generalized eigenspaces corresponding to Jordan blocks of dimension strictly more than two with a non-real eigenvalue, then the generalized eigenspaces corresponding to Jordan blocks of dimension strictly more than one with a real eigenvalue, then the remaining eigenspaces.
There are nonnegative integers $m$ and $p$ such that in the basis
\begin{align} 
\mathcal F_\varepsilon  = & (e_1, e_2, \varepsilon e_3, \varepsilon e_4, \ldots, \varepsilon^{m-1} e_{2m-1}, \varepsilon^{m-1}e_{2m}, e_{2m+1}, \label{eq:variation-basis} \\
&  \varepsilon e_{2m+2} \ldots, \varepsilon^{p-1} e_{2m+p}, e_{2m+p+1},\ldots,e_{n-1}), \notag
\end{align}
$\alpha$ has a block upper triangular form with blocks of the form
    \[
     J'_{2d}(1+i \tau) = \begin{bmatrix}
    A_\tau & \varepsilon I & & \\
     & \ddots & & \varepsilon I \\
     & & & A_\tau
    \end{bmatrix} \text{ where } \; A_\tau = \begin{pmatrix} 1 & \tau \\ -\tau & 1
    \end{pmatrix}\]
    and
    \[ 
    J_d(1) =  \begin{bmatrix}
   1 & \varepsilon & & \\
     & \ddots & & \varepsilon \\
     & & & 1
    \end{bmatrix}
   \]
   where $d \geqslant 1$ denotes the size of the block (the blocks with $d=1$ being in the end).
Consider the left invariant metric $ \langle \cdot , \cdot \rangle_\varepsilon$ such that $\mathcal F_\varepsilon$ is orthonormal and $T \perp [\mathfrak s, \mathfrak s]$, $\langle T, T \rangle = 1$ for some $T$ such that $\alpha = \operatorname{ad}(T)$.
Decompose $\operatorname{ad}(T) = D_\varepsilon + S_\varepsilon$, where $D_\varepsilon$ is symmetric and $S_\varepsilon$ is skew-symmetric in $\mathcal F_\varepsilon$. To express the Riemann curvature tensor, following Heintze, Eberlein and Heber it is convenient to introduce\footnote{They are denoted $D_0, S_0$ in \cite{Heintze} and $D_0, S_0, N_0$ in \cite{EberleinHeber}.} $N_\varepsilon = D_\varepsilon^2 +[D_\varepsilon,S_\varepsilon]$. For all $X$, $Y$, $Z$ in $\mathfrak s$,
\begin{align*}
    R_{X,Y} Z = & - \langle D_\varepsilon \underline Y, Z \rangle D_\varepsilon \underline X + \langle D_\varepsilon \underline X, Z \rangle D_\varepsilon \underline Y \\
    & -  \left\langle \underline Z, \langle X,T \rangle N_\varepsilon \underline Y - \langle Y, T \rangle N_\varepsilon \underline X \right\rangle T \\
    &     +  \langle Z, T \rangle (\langle X, T \rangle N_\varepsilon \underline Y - \langle Y,T \rangle N_\varepsilon \underline X ) ,
\end{align*}
where $\underline X$, $\underline Y$ and $\underline Z$ are the orthogonal projections of $X,Y,Z$ to $[\mathfrak s, \mathfrak s]$.
(This is differently expressed as in, but still in agreement with, \cite{EberleinHeber} who performed a more general computation where $[\mathfrak s, \mathfrak s]$ is not assumed abelian and provided $R_{X,Y}Z$ for $X,Y,Z \in [\mathfrak s, \mathfrak s]$ and the sectional curvature of all planes.)
Any $2$-plane $\pi$ in $\mathfrak s$ can be generated by $u, v \in \mathfrak s$ such that $v\in [\mathfrak s, \mathfrak s]$, so that $v = \underline v$. 
Observe that as $\varepsilon \to 0$, $D_\varepsilon \to I$ and $N_\varepsilon \to I$ so that, denoting by $\mathrm{sec}^\varepsilon$ the sectional curvature with respect to $\langle \cdot, \cdot \rangle_{\varepsilon}$, 
\begin{align*}
    \mathrm{sec}^\varepsilon(\pi) &= \frac{\langle R^\varepsilon (u,v)v,u \rangle}{\langle u, u  \rangle \langle v,v \rangle - \langle u,v \rangle^2} \\
    & = \frac{\langle -D_\varepsilon \underline u, \underline u \rangle \langle D_\varepsilon  v,  v \rangle  + \langle D_\varepsilon \underline u, v \rangle^2 
    - \langle u, T \rangle^2 \langle  v, N_\varepsilon  v \rangle
    }{\langle u, u  \rangle \langle v,v \rangle - \langle u,v \rangle^2} \\
    & \longrightarrow_{\varepsilon \to 0} 
    \frac{- \langle \underline u, \underline u \rangle \langle  v,  v \rangle + \langle u, v \rangle^2 - \langle u,T \rangle^2 \langle v, v \rangle}{\langle u, u  \rangle \langle v,v \rangle - \langle u,v \rangle^2} = -1,
\end{align*}
using that $\langle u, u \rangle = \langle \underline u, \underline u \rangle + \langle u, T \rangle^2$ and $\langle \underline u, \underline v\rangle = \langle u, v \rangle$.
Finally, the pointwise convergence of a rational function on a Grassmanian implies its uniform convergence, so $\sup \operatorname{sec}^\varepsilon - \inf \operatorname{sec}^\varepsilon$ goes to zero and $\sup \mathrm{sec}^\varepsilon / \inf \mathrm{sec}^\varepsilon$ goes to $1$ as $\varepsilon \to 0$.

Let us now prove the ``moreover'' part. 
Start assuming for simplicity that $K$ is connected. Every block $\mathcal B$ of $\alpha$ of type $J_{d}(s)$ or $J'_{2d}(s)$ for $d>1$ and $s \in \mathbf C$ determines a linear subspace of $\mathbf R^{n-1}$ of the form $\operatorname{span}(e_k,\ldots,e_{k+d})$ or $\operatorname{span}(e_k,\ldots,e_{k+2d})$ together with a non-trivial flag of subspaces $\{ \mathcal B_i^{\triangleleft}\}_{0 \leqslant i \leqslant d-1}$ (in increasing order for inclusion) stabilized by $\alpha$. 
Let $\{ \varphi^t \}_{t \in \mathbf R}$ be a one-parameter subgroup of $K_0$; once restricted  to $\mathcal B^{\triangleleft}_{d-1}$,
$\{ \varphi^t \}$ being a connected group of automorphisms of $\mathfrak s$, must stabilize the flag, hence (remembering that $K_0$ is compact) it must act trivially on $\mathcal B_{d-1}^\triangleleft$ if $\mathcal B$ is of type $J_d$ or within a diagonal torus if $\mathcal B$ is of type $J'_d$. Consequently, there are integers $r$ and $\ell$ such that
\begin{equation}
\label{eq:decompo-of-K}
 K < \mathbf T^\ell     \times  K_{0} \times K_{\tau_1} \times \cdots \times K_{\tau_r}
\end{equation}
where $K_{0}$ is a compact group stablilizing the direct sum of blocks of type $J_1(1)$, $K_{\tau_i}$ stabilizes the direct sum of blocks of type $J_2(\tau_i)$ for all $1 \leqslant i \leqslant r$, and the remaining torus $\mathbf T^\ell$ stabilizes the higher sized blocks of type $J'$.
Then letting $\mu$ be the normalized Haar measure on $K$, replace $\langle \cdot , \cdot \rangle_{\varepsilon}$ with $\langle X , Y \rangle^K_\varepsilon = \int_{K} \langle \varphi X, \varphi Y \rangle_\varepsilon d\mu(\varphi)$. If $\mathcal F'_1 = (e'_1, \ldots, e'_{n-1})$ denotes an orthonormal basis for $\langle \cdot, \cdot \rangle_1$ that respects the ordered block decomposition of $\alpha$ (we know there is such a basis thanks to \eqref{eq:decompo-of-K}), then $\mathcal F'_\varepsilon$ obtained from $\mathcal F'_1$ by rescaling the vectors as in \eqref{eq:variation-basis} will be orthonormal for $\langle \cdot, \cdot \rangle^K_{\varepsilon}$, and one can now apply the previous argument estimating the sectional curvature verbatim.

Finally, $K$ may not be connected, and in this last case, one needs to change slightly the rescaling procedure of the basis to account for the fact that $K$ can now exchange the higher sized blocks. 
One should reorganize the powers of $\varepsilon$ so that higher sized blocks of the same type are scaled by the same powers of $\varepsilon$. Specifically, with the notation as above, $\mathcal B_i^\triangleleft$ must be spanned by the vectors $\varepsilon^i e_k$ in the new basis. In doing so, we preserve the matrices $D_\varepsilon$ and $N_\varepsilon$ as they were before, hence the bounds on the sectional curvature.
\end{proof}

\begin{remark}

Using Eberlein and Heber's amalgams (Definition \ref{def:amalgam}) and curvature estimates would simplify the proof of the first part of Proposition \ref{prop:Riemannian-computation} (yet not drastically so) by reducing it to the case where $\alpha$ has a single Jordan block as Jordan normal form. See also Remark \ref{rem:dejavu}.
\end{remark}

\begin{ques}
\label{ques:minimize-pinching}
Let $G = N \rtimes  (K \times R)$ be a group of Heintze type.
Is it true that among all negatively curved Riemannian models of $G$, an optimal pinching is attained if and only if $\alpha$ is diagonalizable over $\mathbf C$? 
\end{ques}

Note that the (Ahlfors-regular) conformal dimension of $\partial_\infty [\mathbf R^{n-1} \rtimes_{\alpha} \mathbf R]$ is attained if and only if $\alpha$ is diagonalizable over $\mathbf C$ \cite{BonkKleinerCdim}.

\subsection{Degenerations and deformations}
\label{subsec:degenerations}

We provide more information here than is strictly needed for Theorem \ref{th:Tukia-SBE}.
That will be useful to us in the discussion in \S\ref{subsec:nilpotent}.

\subsubsection{Setting}
Let $\mathcal L_n(\mathbf R) \subseteq (\Lambda^2 \mathbf R^n)^\ast \otimes \mathbf R^n$ be the subset of Lie algebra laws on $\mathbf R^n$. 
Note that $\mu \in \Lambda^2 (\mathbf R^n)^\ast \otimes \mathbf R^n$ is in $\mathcal L_n(\mathbf R)$ if and only if the Jacobi identity holds in $\mu$, that is, if and only if
\begin{equation}
    \label{eq:mu-circ-nu}
    \mu^2(X_1 \wedge X_2 \wedge X_3) = \sum_{\sigma} \mu \left( \mu(X_{\sigma(1)} \wedge X_{\sigma(2)}) \wedge X_{\sigma(3)} \right) = 0
\end{equation}
for every $X_1, X_2, X_3 \in \mathbf R^n$, the sum being taken over the three positive permutations $\sigma$ over $\lbrace 1, 2, 3 \rbrace$.
$\mathcal L_n(\mathbf R)$ has two topologies: the Zariski topology, and the topology it inherits as a subspace of $\Lambda^2 (\mathbf R^n)^\ast \otimes \mathbf R^n$ with the operator norm, that we will call the metric topology.
It follows from Engel's theorem that the nilpotent laws form a Zariski closed subset $\mathcal N_n(\mathbf R)$.

Let $\lambda \in \mathcal L_n(\mathbf R)$. 
$\mathbf R$, resp. $\lambda$, is a $\lambda$-module for the trivial, resp. the adjoint representation of $\lambda$. Following Chevalley and Eilenberg \cite[Theorem 10.1]{ChevalleyEilenberg} there are differential complexes $K_\lambda$ and $K'_\lambda$ on $\Lambda^\bullet(\mathbf R^n)^\ast$ and $\Lambda^\bullet(\mathbf R^n)^\ast \otimes \mathbf R^n$ with the following exterior derivatives $d_\lambda$, resp. $d_\lambda'$ on degree $q$-forms, resp.\ on $\lambda$-valued degree $q$-forms $\omega$:
\begin{align}
    d_\lambda \omega (x_1, \ldots, x_{q+1}) & = \sum_{k < \ell} (-1)^{k+\ell} \omega( \lambda(x_k, x_{\ell}), x_1, \ldots , \widehat {x_k}, \ldots, \widehat {x_\ell}, \ldots , x_{q+1}) \label{eq:chevalley-eilenberg-trivial} \\
    d'_\lambda \omega (x_1, \ldots, x_{q+1}) & = \sum_{k < \ell} (-1)^{k+\ell} \omega( \lambda(x_k, x_{\ell}), x_1, \ldots , \widehat {x_k}, \ldots, \widehat {x_\ell}, \ldots , x_{q+1}) \notag \\ 
    & \quad + \sum_{k} (-1)^{k+1} \lambda(x_k, \omega(x_1, \ldots, \widehat{x_k}, \ldots, x_{q+1})).
    \label{eq:chevalley-eilenberg-adjoint}
\end{align}

The group $\operatorname{GL}(n,\mathbf R)$ acts on $\mathcal{L}_n(\mathbf R)$ by restricting its natural action on $\Lambda^2 (\mathbf R^n)^\ast \otimes \mathbf R^n$.
We denote the orbit of $\lambda$ by $O(\lambda)$ or $O_\mathfrak g$ if $\mathfrak g$ is a Lie algebra isomorphic to $\lambda$; it is a smooth submanifold of $\Lambda^2 (\mathbf R^n)^\ast \otimes \mathbf R^n$ of dimension $n^2 - \dim \operatorname{Der}(\mathfrak g)$, embedded in $\mathcal L_n(\mathbf R)$. 
Moreover, $T_\lambda {O_{\mathfrak g}} = B^2(\lambda, \lambda)$, as is most conveniently seen by differentiating the action of $\mathrm{GL}(n, \mathbf R)$ at $\lambda$: for every $\eta \in \mathfrak{gl}(\mathbf R^n)$,
\begin{align}
    \label{eq:linearization}
    e^\eta \lambda( e^{-\eta} X, e^{-\eta} Y) - \lambda(X \wedge Y) = d'_\lambda \eta (X \wedge Y) + O(\Vert \eta \Vert^2).
\end{align}

\begin{example}
\label{exm:affine}
Let $\mathfrak g = \mathfrak {aff}$ be the $2$-dimensional affine Lie algebra with basis $\{ X,T \}$ such that $[T,X] = X$ and dual basis $\{ dx, dt \}$. Then $X \otimes dx \wedge dt \in B^2(\mathfrak g, \mathfrak g)$; in the language of \S \ref{subsec:pinching}, $\mathbf R \rtimes_{1+\varepsilon} \simeq \mathbf R \rtimes_1 \mathbf R \simeq \mathfrak g$. 
\end{example}

\begin{definition}
Let $\mathfrak g$ and $\mathfrak h$ be Lie algebras of dimension $n$ over $\mathbf R$. 
We say that $\mathfrak g$ degenerates to $\mathfrak h$, denoted $\mathfrak g \to_{\mathrm {deg}} \mathfrak h$, if $\overline {O_\mathfrak h} \subsetneq \overline {O_{\mathfrak g}}$ where the closure is taken for the Zariski topology.
\end{definition}

Note that it is equivalent to require a single $\mu \in O_{\mathfrak h}$ such that $\mu \in \overline{O_{\mathfrak g}}$.
Since the metric topology is finer than the Zariski topology, a sufficient condition to have $\mathfrak g \to_{\mathrm {deg}} \mathfrak h$ is that there is a sequence $\lambda_0, \ldots, \lambda_r$ such that
\begin{equation}
\label{eq:sequence-of-metric-degenerations}
\begin{cases}
    \lambda_0 \in O_{\mathfrak g}, \; \lambda_r \in O_{\mathfrak h} & \\
    \forall X \in \Lambda^2 (\mathbf R^n),\,
    \underset{t \to + \infty}{\lim} (\varphi_{t,i} . \lambda_i)(X) = \lambda_{i+1} (X) & i=0,\ldots, r-1.
\end{cases}
\end{equation}
where $\varphi_t \in \mathrm{GL}(n, \mathbf R)$ is continuous with respect to $t$.

When $r=1$, \eqref{eq:sequence-of-metric-degenerations} amounts to $\mu \in \overline{O(\lambda)}^{\, \mathrm{ met}}$ and is called a contraction (especially, by the physicists).
The author does not know whether the existence of a sequence of contractions as in \eqref{eq:sequence-of-metric-degenerations} is a necessary condition for $\mathfrak g \to_{\mathrm {deg}} \mathfrak h$ to hold. 

\begin{example}[Nilpotent Lie algebras]
\label{ex:nilpotent}
Let $\mathfrak n$ be a nilpotent Lie algebra. 
Let $\mathfrak n = \oplus_i V_i$ be a linear splitting such that $V_i \oplus C^{i+1} \mathfrak n = C^i \mathfrak n$ for all $i$. For $t > 0$, let $(\varphi_t)$ be the one parameter subgroup of $\operatorname{GL}(\mathfrak n)$ such that
\begin{align}
    \varphi_t(X) = t^i X, & \qquad X \in V_i.
\end{align}
Then, the $V_i$ becomes a Lie algebra grading on $\varphi_t.\mathfrak n$ in the limit when $t \to + \infty$: $\mathfrak n$ degenerates metrically to the graded Lie algebra $\operatornamewithlimits{gr} (\mathfrak n)$ associated to the central filtration of $\mathfrak n$, supporting the asymptotic cone of the simply connected $N$ by \cite{PanCBN}. In particular, $\mathfrak n \to_{\mathrm{deg}} \operatorname{gr}(\mathfrak n)$.
(This description of the law in $\operatorname{gr}(\mathfrak n)$ as a limit is the one given in \cite[\S 2.1]{CantFur}, who prove a generalization of \cite{PanCBN}.)

\end{example}

For $\lambda \in \Lambda^2 (\mathbf R^n)^\ast \otimes_{\mathbf R} (\mathbf R[[1/t]])^n$, we denote $(\lambda,t) \mapsto \lambda(t)$ provided that $t$ is in the convergence domain of every coefficient of $\lambda$, and $\lambda[1/t^d]$ the monomial of degree $d$. (The choice of $\mathbf R[[1/t]]$ over $\mathbf R[[t]]$ is just a peculiarity for our convenience.) We also denote $\lambda(\infty)$ the constant term of $\lambda$. If $\lambda(t) \in \mathcal L_n(\mathbf R)$ for all $t \geqslant 1$, $\lambda$ is called a formal deformation.

Differentiating \eqref{eq:mu-circ-nu} to express that $\lambda$ is a formal deformation with $\lambda(\infty) = \mu$ yields an infinite system of equations, the first of which after \eqref{eq:mu-circ-nu} being
\begin{align}
    \label{eq:deformation-equation}
    d'_\mu \lambda[1/t] = 0,
\end{align}
that is, $\lambda[1/t] \in Z^2(\mu, \mu)$. 

\begin{definition}
Let $\mathfrak g$ be a Lie algebra over $\mathbf R$. Let $\mu \in \mathcal{L}_n(\mathbf R)$ represent $\mathfrak g$, and let $\omega \in H^2(\mathfrak g, \mathfrak g)$ be nonzero. 
We say that the formal deformation $\lambda$ integrates the infinitesimal deformation $\omega$ at $\mu$ if $\lambda(\infty) = \mu$, $\lambda$ is convergent on $\mathbf C \setminus \lbrace 0 \rbrace$ and $\lambda[1/t] \in Z^2(\mu, \mu)$ represents $\omega$. 
We say that $\omega$ is integrable, resp. linearly expandable (as the authors in \cite{AncocheaCampoamor} do) if a formal deformation $\lambda$ integrates $\omega$, resp. if $\lambda$ is a formal deformation of $\omega$ and $\lambda = \lambda(\infty) + \lambda_1/t$ for some $\lambda_1 \in \Lambda^2 \mathbf R^n \otimes \mathbf R^n$.
\end{definition}

In the last Definition, we insisted more on the cohomology class than on the particular cocycle $\lambda[1/t]$ for the following reason. 
Two formal deformations $\lambda, \lambda'$ of $\mu$ are called equivalent if $\lambda(t) = \varphi(t). \lambda'(t)$ for some $\varphi \in \operatorname{GL}(\mathbf R[[t]])$ with $\varphi(\infty) = 1$. If $\lambda$ and $\lambda'$ are equivalent then $\lambda[1/t] - \lambda'[1/t] \in B^2(\mu, \mu)$; this is a better version of \eqref{eq:linearization}, see e.g. Proposition just before \S 2.5 in \cite{AncocheaCampoamor}. 
In view of \eqref{eq:linearization}, \eqref{eq:deformation-equation} and this, $H^2(\mathfrak g, \mathfrak g)$ encodes the degree to which $\mathfrak g$ can be deformed; one should nevertheless beware that infinitesimal deformations are not always integrable (See Remark \ref{remark:not-always-integrable}). 

\subsubsection{Degenerations to $\mathfrak b(n, \mathbf R)$}
Let $\mathfrak{b}(n, \mathbf R)$ denote the maximal completely solvable subalgebra of $\mathfrak {o}(n,1)$, namely $\mathfrak{b}(n, \mathbf R) = \mathbf R^{n-1} \rtimes_1 \mathbf R$, where the adjoint action of $1 \in \mathbf R$ on $\mathbf R^{n-1}$ is by the identity.
The situation of $\mathfrak{b}(n, \mathbf R)$ with respect to degenerations and deformations is favorable:

\begin{theorem}[After Lauret]
\label{th:Lauret}
Let $\mathfrak g$ be a completely solvable Lie algebra and $n \geqslant 2$ an integer. The following are equivalent:
\begin{enumerate}[{\rm (\ref{th:Lauret}.1)}]
    \item 
    \label{item:metric-deg}
    $\mathfrak g$ contracts to $\mathfrak b(n, \mathbf R)$.
    \item 
    \label{item:deg}
    $\mathfrak g \to_{\mathrm{deg}} \mathfrak b(n, \mathbf R)$.
    \item
    \label{item:explicit-met-deg}
    $\mathfrak g$ decomposes as $\mathbf R^{n-1} \rtimes_{\nu} \mathbf R$ where $\nu$ is unipotent.
\end{enumerate}
Moreover, under the former conditions there exists $\omega \in H^2(\mathfrak{b}(n, \mathbf R), \mathfrak{b}(n, \mathbf R))$ linearly expandable into a formal deformation $\lambda$ such that $\lambda(1) \in \mathcal{O}_{\mathfrak g}$ and $\lambda(\infty) \in O_{\mathfrak b(n, \mathbf R)}$.
\end{theorem}

Lauret proved (\ref{th:Lauret}.\ref{item:metric-deg}) $\iff$ (\ref{th:Lauret}.\ref{item:explicit-met-deg}) \cite[Theorem 6.2]{LauretDegenerations} with no a priori assumption on $\mathfrak g$. The core of the proof below uses the same idea. (\cite{LauretDegenerations} additionaly used bounds on pinching and \cite{EberleinHeber} that give constraints a priori on $\mathfrak g$).

We need a Lemma which is well-known, however we could only find proofs for the metric topology in the literature.

\begin{lemma}
\label{lem:lsc-zariski}
Let $n$ be a positive integer and $0 \leqslant i \leqslant n$.
Then, the following are upper semi-continuous with respect to the Zariski topology on $\mathcal L_n (\mathbf R)$:
\begin{enumerate}[{\rm (a)}]
    \item 
    \label{item:betti-lsc}
    The Betti number $b_p (\lambda) = \dim H^p(\lambda, \mathbf R)$, for all $p \geqslant 0$.
    \item
    \label{item:adjoint-lsc}
    The dimension of the outer derivations $H^1(\lambda, \lambda) = \operatorname{Der}(\lambda) / \operatorname{InnDer} (\lambda)$.
    \item
    \label{item:dim-center-lsc}
    The dimension of the center $\dim Z(\lambda)$. 
\end{enumerate}
\end{lemma}

\begin{proof}
Note that $Z(\lambda) = H^0(\lambda, \lambda)$, so to prove \eqref{item:betti-lsc}, \eqref{item:adjoint-lsc} and \eqref{item:dim-center-lsc} it is actually sufficient to prove that $\lambda \mapsto b_p(\lambda)$ and $\lambda \mapsto \dim H^{p}(\lambda, \lambda)$ are upper semicontinuous on $\mathcal L_n(\mathbf R)$.
We will prove this by a change of basis argument.
Denote by $x_{ij}^k$ the coordinate functions on $\Lambda^2 (\mathbf R^n)^\ast \otimes \mathbf R^n$, and 
let $\mathcal I$ be the ideal of $\mathbf R[x_{ij}^k]$ generated by the relation \eqref{eq:mu-circ-nu}.
Let $A = \mathbf R[x_{ij}^k] /\mathcal I$.
Then $A$ is a Noetherian ring by Hilbert's basis theorem, and $\mathcal L_n(\mathbf R)$ with the Zariski topology is a closed subspace of $\operatorname{Spec}(A)$ with the Zariski topology; all the points in $\mathcal L_n(\mathbf R)$ are maximal ideals.
Consider the graded $A$-modules
\begin{align*}
    K = \Lambda^\bullet (A^n)^\ast \\
    K' = \Lambda^\bullet (A^n)^\ast \otimes_A {A}^n.
\end{align*}
(Here, $(A^n)^\ast$ denotes $\operatorname{Hom}(A^n,A)$.)
For every pair $y_1,y_2$ in $\mathbf R[x_{ij}^k]^n$, there is a polynomial $z$ such that $\lambda(y_1(\lambda),y_2(\lambda)) = z(\lambda)$ for every $\lambda$ in $\Lambda^2 (\mathbf R^n)^\ast \otimes \mathbf R^n$. The class of $z$ modulo $\mathcal I$ only depends on the classes of $y_1$ and $y_2$ modulo $\mathcal I$. Hence, there is a well defined application $ A^n \times A^n \to A^n$, $A$-linear in both arguments, that we denote by the bracket. In this way $[\cdot, \cdot]$ defines an element of $\Lambda^2(A^n)^\ast \otimes A^n$.
The differentials on $K$ and $K'$ are defined as in \eqref{eq:chevalley-eilenberg-trivial} and \eqref{eq:chevalley-eilenberg-adjoint} defining the differentials on $K_\lambda$ and $K'_\lambda$ respectively.

In this way $K_\lambda = K \otimes_A A/\lambda$ and $K'_\lambda = K' \otimes_A A/\lambda$, where $A/\lambda$ is the residual field of $A$ at the maximal ideal $\lambda$. 
$K$ and $K'$ are flat $A$-modules, because being flat is preserved by taking exterior and tensor products over the base ring \cite[Proposition 2.3]{LazardPlat}.
We may now conclude by applying the following \cite[Théorème 7.6.9(i)]{EGAIII2}: if $A$ is Noetherian and $K$ is a differential complex of finitely generated flat modules, then for every $p \geqslant 0$, the function $y \mapsto \dim H^p(K \otimes_A k(y))$ is upper semi-continuous on $\operatorname{Spec}(A)$, where $k(y)$ denotes the residual field at $y$.
In particular, it is upper semi-continuous on the closed subspace $\mathcal L_n(\mathbf R)$.

\end{proof}

\begin{proof}[Proof of Theorem \ref{th:Lauret}]
\eqref{item:metric-deg} $\implies$ \eqref{item:deg} is clear.

Assume \eqref{item:deg}. 
By Lemma \ref{lem:lsc-zariski}, $b_1(\mathfrak g) \leqslant 1$. If it is zero, then $\mathfrak g$ is perfect, especially it is not solvable; hence
$b_1 = 1$, and $\mathfrak g$ splits as a semidirect product 
\begin{equation}
    [\mathfrak g, \mathfrak g] \oplus \mathbf R A
\end{equation}
where the restriction of $\operatorname{ad}_A$ to $[\mathfrak g, \mathfrak g]$ is nonsingular in view of the fact that $Z(\mathfrak g) = 0$, again by Lemma \ref{lem:lsc-zariski}. Choosing an adequate representative $\lambda_0$ in $O_{\mathfrak g}$ and an adequate basis we may as well assume that $[\lambda_0, \lambda_0] = \mathbf R^{n-1}$ and $A = (0^{n-1},1)$.

The coefficients of the characteristic polynomial $P_{\mu, X}$ of $\operatorname{ad}_X: Y \mapsto \mu(X,Y)$ are polynomial functions on $\mathcal L_n(\mathbf R)$, and for every $\lambda_1 \in O(\lambda)$ the spectrum of $P_{\lambda_1, X}$ is either a nonzero multiple of $\operatorname{Sp} (P_{\lambda_0, A})$, or $0$ with multiplicity $n$, the latter case occurring $X \in [\lambda_1, \lambda_1]$.
So, for $\mu \in \overline{O(\lambda_0)}$ this holds as well. But for $\mu \in O_{\mathfrak b(n, \mathbf R)}$, this spectrum is always concentrated at one point.
So $\operatorname{ad}_A$ cannot have two distinct eigenvalues, and then $[\mathfrak g, \mathfrak g]$ is abelian, which proves \eqref{item:explicit-met-deg}.

Assume \eqref{item:explicit-met-deg}. 
Then $\nu-1$ is nilpotent; let $X_1, \ldots , X_{n-1}$ be a basis of $[\mathfrak g, \mathfrak g]$ in which it appears in lower-triangular Jordan form, $\nu-1 = \sum_i \delta_i X_{i}^\ast \otimes X_{i+1}$ where $\delta_i \in \lbrace 0,1 \rbrace$. One computes that $d(A^\ast \wedge X_i^ \ast \otimes X_{i+1}) = 0$ (Lemma \ref{lem:differentials-of-2-cochains-R}; beware that $S$ replaces $A$ there) and that no nonzero linear combination of those is a coboundary (Lemma \ref{lem:differentials-of-1-cochains-R}).
Setting $\mu$ the law of $\mathfrak b(n, \mathbf R)$ in the basis $(X_1, \ldots X_{n-1}, A)$ and
$\omega = A^\ast \wedge \sum_i \delta_i X_i^\ast \otimes X_{i+1} $ we find that $\lambda_0 = \mu + \omega$.
Then $\mu$ is the degeneration of $\lambda$ through $(\varphi_t)$, where $\varphi_t A = A$ and $\varphi_t X_i = t^{-i} X_i$ for all $t$.
\end{proof}

\begin{remark}
\label{rem:dejavu}
A contraction to (a deformation of) $\mathfrak b(n, \mathbf R)$ was already  used in the proof of Proposition \ref{prop:Riemannian-computation}; in accordance with \cite{LauretDegenerations}, contractions can be considered as limit points in the space of left-invariant Riemannian metrics over a given group.
\end{remark}

\begin{remark}
\label{remark:not-always-integrable}
We can additionally check that $H^3(\mathfrak b, \mathfrak b)=0$ when $\mathfrak b = \mathfrak b(2, \mathbf R)$, though it is unnecessary. 
This vanishing ensures that the deformation system can be solved and every infinitesimal deformation of $\mathfrak b$ is integrable into a formal deformation {\cite[p.98]{NijenhuisRichardsonStruc}}.
For nilpotent Lie algebras $\mathfrak n$ that will be discussed more in detail in \S\ref{subsec:nilpotent}; on the other hand, one must beware that $H^3(\mathfrak n, \mathfrak n)$ is large, for instance $\dim H^3(\mathfrak n, \mathfrak n) \geqslant 8$ for all the $6$-dimensional nilpotent $\mathfrak n$ \cite[Table 11]{Magnin08}.
\end{remark}

\subsection{Groups $O(u)$-bilipschitz equivalent to $\mathbb H^n_{\mathbf R}$}
\label{subsec:proofA}

We prove here Theorem \ref{th:Tukia-SBE}.
Let us first recall some terminology from \cite{CoTesContracting} and \cite{CCMT}.

\begin{definition}
Let $G$ be a Lie group with finitely many components.
$G$ is of rank-one type if it has a maximal normal compact subgroup $W$ such that $G/W$ is isomorphic to a simple Lie group $G_{\mathbf R}$ of real rank one, with $Z(G_{\mathbf R}) = 1$.
\end{definition}

Let us proceed to prove the following chains of implications:
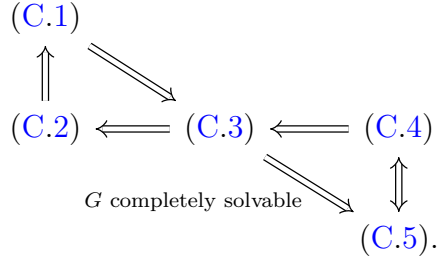
\begin{center}
\begin{tikzcd}
(\ref{th:Tukia-SBE}.\ref{sublinear-characterization}) \arrow[dr, Rightarrow] 
& 
& \\
(\ref{th:Tukia-SBE}.\ref{item:log-characterization}) \arrow[u, Rightarrow] 
& (\ref{th:Tukia-SBE}.\ref{pinching-characterization}) \arrow[dr, Rightarrow, swap, "G \text{ completely solvable}"] \arrow[l, Rightarrow]
& (\ref{th:Tukia-SBE}.\ref{item:degeneration-characterization-real}) \arrow[l, Rightarrow, swap] \\
& 
& (\ref{th:Tukia-SBE}.\ref{item:explicit-characterization-tukia}). \arrow[u, Leftrightarrow, swap]
\end{tikzcd}
\end{center}

\begin{description}
\item[{\bf {\rm (\ref{th:Tukia-SBE}.\ref{sublinear-characterization})} implies {\rm (\ref{th:Tukia-SBE}.\ref{pinching-characterization})}:}]
\label{1to3}

Let $G$ be a Lie group with finitely many connected components. Assume that $G$ is $O(u)$-sublinear bilipschitz equivalent to $\mathbb H_{\mathbf R}^n$ for some $n$. 
Then all asymptotic cones of $G$ being $\mathbf R$-trees, $G$ is Gromov-hyperbolic. 
By Cornulier and Tessera's theorem \cite{CoTesContracting}, $G$ is either of Heintze or rank-one Lie type.
First assume that $G$ is of Heintze type, write $G =  (K \times \mathbf R) \ltimes N$ and call $H$ the co-compact normal subgroup $\mathbf R \ltimes N$ so that $G/K$ is simply transitively acted upon by $H$.
By \cite{pallier2019conf}, $\operatorname{Cdim}_{O(u)} \partial_\infty H = \operatorname{Cdim}_{O(u)} \partial_\infty \mathbb H_{\mathbf R}^n = n-1$.
By \cite{cornulier2017sublinear}, $ \operatorname{Topdim} \partial_\infty H = n-1$. So $H$ is metabelian and every eigenvalue of $\alpha$ has real part $1$. By Proposition \ref{prop:Riemannian-computation}, (\ref{th:Tukia-SBE}.\ref{pinching-characterization}) holds, while by \cite[Theorem 1.2]{CornulierCones11}, ({\rm \ref{th:Tukia-SBE}.\ref{item:log-characterization}}) holds.
If $G$ is of rank-one type, then it acts properly co-compactly by isometries on a rank one symmetric space, which can only be $\mathbb H^n_{\mathbf R}$ in view of the equality of conformal dimension and topological dimension of the boundary; especially, ({\rm \ref{th:Tukia-SBE}.\ref{item:log-characterization}})  and (\ref{th:Tukia-SBE}.\ref{pinching-characterization}) hold as well.

\item[
{\bf {\rm (\ref{th:Tukia-SBE}.\ref{pinching-characterization})} implies {\rm (\ref{th:Tukia-SBE}.\ref{item:log-characterization})}:}]
\label{3to2}
Since it acts geometrically on Gromov-hyperbolic spaces, $G$ is Gromov-hyperbolic. Again by \cite{CoTesContracting}, it is of Heintze type or rank-one type. If it is rank-one type, then it is quasiisometric to a rank one symmetric space $X$; by Pansu's Theorem \ref{th:pansuconf}, $\operatorname{Cdim}(\partial_\infty G) = \operatorname{Topdim}(\partial_\infty G)$, so $X = \mathbb H_{\mathbf R}^n$.
If it is Heintze-type, then it is quasiisometric to a purely real Heintze group of the form $N \rtimes_{\alpha} \mathbf R$. Arguing as in the proof of Corollary \ref{cor:on-pinching}, every eigenvalue of the $[\alpha]$ is equal to $1$. By \cite{CornulierCones11}, {\rm (\ref{th:Tukia-SBE}.\ref{item:log-characterization})} holds.

\item[{\bf {\rm (\ref{th:Tukia-SBE}.\ref{item:log-characterization})} implies {\rm (\ref{th:Tukia-SBE}.\ref{sublinear-characterization})}:}]

\label{2to1}
$u= \log$ is an admissible function.

\item[{\bf If $G$ is completely solvable then {\rm (\ref{th:Tukia-SBE}.\ref{pinching-characterization})} implies {\rm (\ref{th:Tukia-SBE}.\ref{item:explicit-characterization-tukia})}:}]
\label{3to5}
By \cite{CoTesContracting}, it is of Heintze type.
and by Corollary \ref{cor:on-pinching}, $N$ is abelian and all the eigenvalues of $\alpha$ have real part $1$. 

\item[{\bf {\rm (\ref{th:Tukia-SBE}.\ref{item:degeneration-characterization-real})} and {\rm (\ref{th:Tukia-SBE}.\ref{item:explicit-characterization-tukia})} are equivalent:}]
\label{4iff5}
This is our version of Lauret's theorem, Theorem \ref{th:Lauret}.

\item[{\bf {\rm (\ref{th:Tukia-SBE}.\ref{item:degeneration-characterization-real})} implies {\rm (\ref{th:Tukia-SBE}.\ref{pinching-characterization})}}]
\label{4to3}
This is a special case of Proposition \ref{prop:Riemannian-computation} where all the eigenvalues of $\operatorname{ad}_A$ are real.

\end{description}

\subsection{Proof of Corollary \ref{cor:sbe-corona}}

Corollary \ref{cor:sbe-corona} follows by applying {(\ref{th:Tukia-SBE}.\ref{pinching-characterization})} $\implies$ {\rm (\ref{th:Tukia-SBE}.\ref{sublinear-characterization})} together with Proposition \ref{prop:sbe-is-coarse}.

\section{Proof of Theorem \ref{thm:groups-sbe-to-h2c}}
\label{sec:proofE}

\subsection{Pointed sphere}

\label{subsec:pointedSphere}
We will prove the implication (\ref{thm:groups-sbe-to-h2c}.\ref{item:G-SBE-to-H2C}) $\implies$ (\ref{thm:groups-sbe-to-h2c}.\ref{item:G-SBE-to-commable}) in Theorem \ref{thm:groups-sbe-to-h2c} by establishing a baby case of a variant of Cornulier's pointed sphere conjecture \cite[Conjecture 19.104]{CornulierQIHLC}. Precisely we establish a special case of the conjecture in the setting of sublinear bilipschitz equivalences rather than quasiisometries for which it is usually formulated. We denote by $\operatorname{SBE}^{O(u)}(X)$ the group of self $O(u)$-bilipschitz equivalences of the metric space $X$ (modulo the relation of $O(u)$-closeness).
Let us first recall that sublinear bilipschitz equivalences induce homeomorphisms of the compact boundary sphere $\partial_\infty X$ when $X$ is Gromov-hyperbolic \cite{cornulier2017sublinear}.

\begin{lemma}
\label{lem:baby-sphere}
Let $u$ be an admissible function.
Let $S$ be a purely real Heintze group such that $[S,S]$ is abelian, and let $\Omega$ be the unique closed orbit of  $\operatorname{SBE}^{O(u)}(S)$ acting by homeomorphisms on $ \partial_\infty S$. The following are equivalent:
\begin{enumerate}[{\rm (1)}]
    \item 
    \label{item:metab-two-distinct-egv}
    $\alpha$ has at least two distinct eigenvalues
    \item
    \label{item:focal-group}
    $\Omega$ is reduced to a single point.
\end{enumerate}
\end{lemma}

\begin{proof}
The reasoning is inspired by \cite[6.9 Corollaire]{PansuDimConf}.
Let $\omega$ be the endpoint of a section of the group $\mathbf R = S/[S,S]$ in $S$, so that $\partial_\infty S \setminus \lbrace \omega \rbrace$ is simply transitively acted upon by $[S,S]$.
Assume \eqref{item:metab-two-distinct-egv} and let $\mathcal F$ be the foliation on $\partial_\infty S \setminus \lbrace \omega \rbrace$ determined by the cosets of $\ker(\alpha - \lambda)$, where $\lambda$ is the minimal eigenvalue of $\alpha$ (since $[S,S]$ is abelian, we may identify it with its Lie algebra). 
Then by \cite[Lemma 3.9]{pallier2019conf}, for all sublinear bilispchitz equivalence $f:S \to S$, the boundary map $\partial_\infty f$ preserves $\mathcal{F}$. Now let $F$ be any leaf of $\mathcal F$. Then, $\lbrace \omega \rbrace$ can be written as $\overline{ F} \setminus F$ or $\overline{(\partial_\infty f) F} \setminus (\partial_\infty f) F$, so that $\partial_\infty f \omega = \omega$.
Conversely, if $\alpha$ only has a single eigenvalue, then $S$ is sublinearly bilipschitz equivalent to real hyperbolic space. Since $\operatorname{Isom}(\mathbb H^n_{\mathbf R})$ is transitive on $\partial_\infty \mathbb H^n_{\mathbf R}$, $\operatorname{SBE}^{O(u)}(S)$ is transitive on $\partial_\infty S$.
\end{proof}

\begin{proposition}
\label{prop:SBE-toh2c}
Let $u$ be an admissible function.
Let $S$ be a Heintze group. Assume that $S$ is $O(u)$-bilipschitz equivalent to $\mathbb H_{\mathbf C}^2$. Then the shadow of
$S$ is isomorphic to $\mathbf {Heis} \rtimes_{\alpha} \mathbf R$ where 
$\mathbf {Heis}$ is the three-dimensional Heisenberg group and
\[ \alpha = \begin{pmatrix} 1 & 0 & 0 \\ 0 & 1 & 0 \\ 0 & 0 & 2 \end{pmatrix} \qquad \text{or} \qquad \alpha = \begin{pmatrix} 1 & 1 & 0 \\ 0 & 1 & 0 \\ 0 & 0 & 2 \end{pmatrix} \]
in a basis $(X, Y, Z)$ of $\mathfrak {heis}$ such that $[X,Y] = Z$.
\end{proposition}

\begin{proof}
Let $S_0 = N \rtimes_\alpha \mathbf R$ be a semidirect product decomposition of the shadow $S_0$ of $S$, where $N$ is three-dimensional and $\alpha$ is normalized so that its lowest eigenvalue is $1$.
The group $S$ has been assumed $O(u)$-bilipschitz equivalent to $\mathbb H_{\mathbf C}^2$; there are two ways to prove that $S$ and $S_0$ have geometric dimension $3$. 

The first is to observe that the Gromov boundary is a topological SBE invariant. Hence $\operatorname{Topdim} \partial_\infty S = \operatorname{Topdim} \partial_\infty \mathbb H_{\mathbf C}^3$, and this is also the dimension of $N$ as a Lie group.
The second (less direct) is to apply Theorem \ref{th:geodim}, $\dim N = \operatorname{asdim}_{\operatorname{AN}} \mathbb H^2_{\mathbf C} - \operatorname{conedim} \mathbb H^2_{\mathbf C} = 3$. 
So $N$ is isomorphic either to $\mathbf R^3$ or to the $3$-dimensional Heisenberg group.
    In the first case, since $\operatorname{Tr}(\alpha) = \operatorname{Cdim}_{O(u)}(S) = 4 > 3$, $\alpha$ has at least two distinct eigenvalues, and by Lemma \ref{lem:baby-sphere}, the unique closed orbit of $\operatorname{SBE}^{O(u)}(S_0)$ acting on $\partial_\infty S_0$ has only one element (namely, $\omega$ from the proof of \ref{lem:baby-sphere}). This contradicts the fact that $\operatorname{SBE}^{O(u)}(\mathbb H^2_{\mathbf C})$ is transitive on $\partial_\infty \mathbb H^2_{\mathbf C}$, so this cannot be.
    Consequently, $N$ is isomorphic to the the three-dimensional Heisenberg group. Let $1, \lambda, \mu$ be the eigenvalues of $\alpha$, where $\mu$ corresponds to the eigenvector generating the center of $\mathbf{Heis}$, and $1\leqslant \lambda \leqslant \mu$. Necessarily, $1+\lambda= \mu$ and $1+\lambda + \mu = 4$, so $2 + 2 \lambda = 4$, and then $\lambda = 1$.
We deduce from there that $\alpha$ can only be one of the two derivations in the conclusion.
\end{proof}

\begin{table}[]
    \begin{tabular}{|cll|}
        \hline
         Nilradical & $\operatorname{Jordan}(\alpha)$ & $\mathbb H^n_{\mathbf K}$ \\
         \hline
         \hline
         $\mathbf R^2$ & $\operatorname{diag}(1, \lambda)$ &  \\
         \hline
         $\mathbf R^2$ & $\operatorname{diag}(1,1)$
         & $\mathbb{H}_{\mathbf R}^3$ \\
         $\mathbf R^2$ & $J_2(\lambda)$
         & \\
         \hline
         $\mathbf R^3$ & $\operatorname{diag}(1,\lambda,\lambda)$
         & \\
         $\mathbf R^3$ & $\operatorname{diag}(1,J_2(\lambda))$
         & \\
         \hline
         $\mathbf R^3$ & $\operatorname{diag}(1,1,\lambda)$
         & \\
         $\mathbf R^3$ & $\operatorname{diag}(J_2(1),\lambda)$
         & \\
         \hline
    \end{tabular}
    \begin{tabular}{|cll|}
        \hline
         Nilradical & $\operatorname{Jordan}(\alpha)$ & $\mathbb H^n_{\mathbf K}$ \\
         \hline
         \hline
         $\mathbf R^3$ &
         $\operatorname{diag}(1,1,1)$
         & $\mathbb{H}_{\mathbf R}^4$ \\
         $\mathbf R^3$ &
         $\operatorname{diag}(1,J_2(1))$
         & \\
         $\mathbf R^3$ &
         $J_3(1)$
         & \\
         \hline
         $\mathbf R^3$ & $\operatorname{diag}(1, \lambda, \mu)$ &  \\
         \hdashline
         $\mathbf {Heis}_3$ & $\operatorname{diag}(1,\lambda, 1+\lambda)$ & \\
         \hline
         $\mathbf {Heis}_3$ & $\operatorname{diag}(1,1, 2)$ & $\mathbb{H}_{\mathbf C}^2$ \\
         $\mathbf {Heis}_3$ & $\operatorname{diag}(J_2(1), 2)$ & \\
         \hline
    \end{tabular}
    \vskip 10pt
    \caption{Purely real Heintze groups of dimension $3$ or $4$, with parameters $1 <\lambda < \mu$. The plain horizontal lines denote the separations between $O(\log)$-bilipschitz equivalence classes that can be deduced from \cite{pallier2019conf} and Theorem \ref{thm:groups-sbe-to-h2c}. The dash line remains unknown when $\mu = 1+\lambda$.
    The isomorphism type of $N \rtimes_{\alpha} \mathbf R$ is generally not determined by $N$ and $\operatorname{Jordan}(\alpha)$ alone; see the $6$-dimensional example after Theorem 1.3 in \cite{CarrascoSequeira}.
    }
    \label{tab:SBE-Heintze-4}
\end{table}

\begin{proof}[Proof of {\rm (\ref{thm:groups-sbe-to-h2c}.\ref{item:G-SBE-to-H2C}) $\implies$ (\ref{thm:groups-sbe-to-h2c}.\ref{item:G-SBE-to-commable})}]
Let $G$ be as in the statement of Theorem \ref{thm:groups-sbe-to-h2c}, namely $G$ is a connected Lie group sublinear bilipschitz equivalent to $\mathbb H^2_{\mathbf C}$. Then $G$ is commable to a completely solvable group $G_0$ \cite[Lemma 6.7]{CornulierDimCone}. Since $G_0$ is Gromov-hyperbolic, by \cite{CoTesContracting} it is a purely real Heintze group \cite{CoTesContracting}. We may then apply Proposition 
\ref{prop:SBE-toh2c} to $G_0$. In the first case where $\alpha$ is diagonalisable, $G_0$ (hence $G$) will be commable to $\mathrm{SU}(2,1)$, in the second case it will be commable to $S'$.
\end{proof}

Let us mention an application to the quasiisometry classification of Heintze groups.
The result below also follows from \cite[Theorem A]{KiviojaThesis} which appeared during the writing of this paper.

\begin{corollary}
\label{cor:qi}
The groups $S'$ and 
$S'' = \mathbf R^3 \rtimes_\alpha \mathbf R$ where 
\[ \alpha = \operatorname{diag}(J_2(1), 2) = \begin{pmatrix} 1 & 1 & 0 \\ 0 & 1 & 0 \\ 0 & 0 & 2 \end{pmatrix} \]
are not quasiisometric.
\end{corollary}

Indeed, if $S$ and $S'$ were quasiisometric, they would be $O(\log)$-bilipschitz equivalent. But $S'$ is $O(\log)$-bilipschitz equivalent to $\mathbb H^2_{\mathbf C}$, whereas $S''$ is not.

See Table \ref{tab:SBE-Heintze-4} for the Heintze groups of dimension at most $4$ and the current knowledge on their $O(\log)$-bilispchitz classification (their quasiisometry classification is known and reduces to isomorphism, see \cite[Theorem C]{KiviojaThesis}).

\begin{remark}
\label{rem:difficulty-extendE}
We can start the same reasoning with $X = \mathbb H_{\mathbf C}^n$, $n>2$. By conformal dimension, any purely real Heintze group $S$ that is $O(u)$-bilipschitz equivalent to $X$ has $[S,S]$ isomorphic to $\mathbf {Heis}^{2k+1} \times \mathbf R^{2(n-k)}$ for some $k \in \lbrace 0, \ldots k-1 \rbrace$, where $\mathbf {Heis}^{2k-1}$ denotes the $2k-1$-dimensional Heisenberg group for $k \geqslant 2$ and $H^1 = \mathbf R$. 
Otherwise said, using the amalgam notation (Definition \ref{def:amalgam})
\[ \mathfrak s = \mathfrak b(k, \mathbf C)\;  \sharp \;  \mathfrak b(2(n-k)+1, \mathbf R), \]
where we recall that $\mathfrak b(k, \mathbf C)$ is the maximal completely solvable subalgebra of $\mathfrak u(k, 1)$. 
But we are only able to prove the pointed sphere conjecture for $S$ when $k=1$: for $k\geqslant 2$ the invariant foliation in $\partial_\infty S$ provided by \cite[Lemma 3.9]{pallier2019conf} becomes a single leaf.
The same reasoning also falls short to characterize the triangulable groups $S$ that are $O(u)$-bilipschitz equivalent to $X = \mathbb{H}_{\mathbf H}^2$, for it leaves the possibility that the Lie algebra of their shadow is
\begin{align*}
    \mathfrak s_0  \in & \left\{ \mathfrak b(5, \mathbf R)  \;  \sharp \;  2 \mathfrak b(4, \mathbf R), \mathfrak b(2, \mathbf C) \;  \sharp \;  \mathfrak b(2, \mathbf R) \;  \sharp \;  2   \mathfrak b(3, \mathbf R), \mathfrak b(3, \mathbf C)\;  \sharp \;  2 \mathfrak b(2, \mathbf R) , \right. \\
    & \left.    
    \mathfrak n_{6} \rtimes_{\operatorname{Carnot}} \mathbf R \; \sharp \; 2 \mathfrak  b(2, \mathbf R) ,
    \mathfrak n_7 \rtimes_{\operatorname{Carnot}} \mathbf R, \mathfrak b(4, \mathbf R) \; \sharp \; \mathfrak l_{4,3} \rtimes_{\mathrm{Carnot}} \mathbf R, \mathfrak b(2, \mathbf H) \right\} 
\end{align*}
where $\mathfrak l_{4,3}$ denotes the $4$-dimensional filiform algebra, $\mathfrak n_6$ is among $\mathfrak l_{6,8}$, $\mathfrak l_{6,22}(-1)$ and $\mathfrak l_{6,22}(0)$ (See \cite{deGraafclass} for structure constants), $\mathfrak n_7$ is one among the real forms of the $4$ complex nilpotent algebras denoted $\mathfrak g_{7,3.12}$ ($2$ real forms), $\mathfrak g_{7,3.24}$, $\mathfrak g_{7,4.1}$ ($2$ real forms) or $\mathfrak g_{7,4.2}$ in \cite{magnin2007adjoint}. Using \cite{pallier2019conf} one can only deduce the pointed sphere conjecture (Lemma \ref{lem:baby-sphere}) for the first $6$ out of these $14$ Lie algebras, while it is expected that it holds for all but the last one.
\end{remark}

\subsection{Degenerations to $\mathfrak b(2, \mathbf C)$}

We prove here a variant of Lauret's theorem \ref{th:Lauret}.

\begin{lemma}
\label{lem:deg-to-bnC}
Let $\mathfrak g$ be a completely solvable Lie algebra of dimension $4$. 
The following are equivalent:
\begin{enumerate}[{\rm (\ref{lem:deg-to-bnC}.1)}]
    \item 
    \label{item:deg-to-bnC-1}
    $\mathfrak g$ contracts to $\mathfrak b(2, \mathbf C)$
    \item
    \label{item:deg-to-bnC-2}
    $\mathfrak g \longrightarrow_{\mathrm{deg}} \mathfrak b(2, \mathbf C)$
    \item 
    \label{item:deg-to-bnC-3}
    $\mathfrak g$ decomposes as $\mathfrak [\mathfrak g, \mathfrak g] \oplus \mathbf R A$, where $[\mathfrak g, \mathfrak g] = \mathfrak{heis}$ and $\operatorname{ad}_A$ is unipotent on $[\mathfrak g, \mathfrak g]/D^3 \mathfrak g$.
\end{enumerate}
\end{lemma}

\begin{proof}
As in the proof of Theorem \ref{th:Lauret}, the core of the proof is that (\ref{lem:deg-to-bnC}.\ref{item:deg-to-bnC-2}) implies (\ref{lem:deg-to-bnC}.\ref{item:deg-to-bnC-3}), so let us focus on this part. Assume that $\mathfrak g \longrightarrow_{\mathrm{deg}} \mathfrak b(2, \mathbf C)$. Then $b_1(\mathfrak g) = 1$ by Lemma \ref{lem:lsc-zariski}. The ideal $\mathfrak n = [\mathfrak g, \mathfrak g]$ is nilpotent by Lie's theorem, and $\mathfrak g = \mathfrak n \rtimes_\beta \mathbf R$ for some nonsingular $\beta \in \operatorname{ad}(\mathfrak n)$. Without loss of generality we can assume that $\operatorname{Sp}(\beta) = \lbrace 1, 2 \rbrace$, and that $2$ has multiplicity $1$. So the nilpotency class of $\mathfrak n$ is at most $2$, and $\operatorname{codim} [\mathfrak n, \mathfrak n] = 1$; thus $\mathfrak n$ is either $\mathbf R^3$ or $\mathfrak {heis}$, and $\mathfrak g$ is among the four algebras
\[ \mathfrak b(2, \mathbf C), \mathfrak b(3, \mathbf R) \; \sharp \; 2 \mathfrak b(2, \mathbf R) , \mathfrak s',  \mathfrak s'', \]
where we recall that $\mathfrak s' = \mathfrak{heis}_{\alpha} \rtimes \mathbf R$ with $\alpha = \operatorname{diag}(J_2(1), 2)$ in the basis $(X,Y,Z)$ and $\mathfrak s''$ is the Lie algebra of $S''$ defined in Corollary \ref{cor:qi}. Observe that
\begin{equation*}
    \dim H^1(\mathfrak g, \mathfrak g) = 
    \begin{cases}
    2 & \mathfrak g = \mathfrak b(2, \mathbf C) \; \text{by Proposition} \; {\rm \ref{prop:first-adjoint-C}} \\
    4 & \mathfrak g = \mathfrak s'' \; \text{by Proposition} \; {\rm \ref{prop:first-adjoint-S''}}
    \end{cases}
\end{equation*}
Note that $\mathfrak s''$ degenerates to $\mathfrak b(3, \mathbf R) \; \sharp \; 2 \mathfrak b(2, \mathbf R)$.
Hence by Lemma \ref{lem:lsc-zariski}, $\dim H^1(\mathfrak g, \mathfrak g) \geqslant 4$ for $\mathfrak g = \mathfrak b(3, \mathbf R) \; \sharp \; 2 \mathfrak b(2, \mathbf R)$, which, again by Lemma \ref{lem:lsc-zariski}, forbids a degeneration of the latter algebra to $\mathfrak b(2, \mathbf C)$.
This establishes (\ref{lem:deg-to-bnC}.\ref{item:deg-to-bnC-2}) $\implies$ (\ref{lem:deg-to-bnC}.\ref{item:deg-to-bnC-3}).
Finally let us prove that $\mathfrak s' \longrightarrow_{\mathrm{deg}} \mathfrak b(2, \mathbf C)$. Take
\[ 
\begin{array}{cccc}
     \varphi_t X = X & \varphi_t Y = e^{-t} Y  &
     \varphi_t Z = e^{-t} Z & \varphi_t A = A. 
\end{array}
\]
Then $\mathfrak s'$ contracts\footnote{This was recorded by Burde and Steinhoff in their list of degenerations between $4$-dimensional complex Lie algebras: $\mathfrak s' \otimes \mathbf C$ is $\mathfrak g(1/64, 5/16)$ in \cite{BurdeSteinhoff} and $\mathfrak s'  \otimes \mathbf C \longrightarrow_{\mathrm{deg}}  \mathfrak b(4, \mathbf R) \otimes \mathbf C$ is the case $\gamma = 2$ in Table IV p. 736 op cit.} to $\mathfrak b(2, \mathbf C)$ through $(\varphi_t)$.
This establishes (\ref{lem:deg-to-bnC}.\ref{item:deg-to-bnC-3}) $\implies$ (\ref{lem:deg-to-bnC}.\ref{item:deg-to-bnC-1}).
\end{proof}

%
The author expects that Lemma \ref{lem:deg-to-bnC} should hold replacing $\mathfrak b(2, \mathbf C)$ with $\mathfrak b(n, \mathbf C)$ and $\mathfrak{heis}$ with $\mathfrak{heis}^{2n-1}$ in (\ref{lem:deg-to-bnC}.\ref{item:deg-to-bnC-3}), though generalizing  Proposition {\rm \ref{prop:first-adjoint-S''}} to higher dimensional algebras comprises some computational hurdles. 
The greatest theoretical difficulty in generalizing Theorem \ref{thm:groups-sbe-to-h2c} (if it holds) from $\mathbb H_{\mathbf C}^2$ to $\mathbb H_{\mathbf C}^n$ with $n>2$ seems to lie on the analytical side, cf. Remark \ref{rem:difficulty-extendE} above.

\section{Some remarks on spaces other than $\mathbb H^n_{\mathbf R}$ and $\mathbb H^n_{\mathbf C}$}

\subsection{Connected Lie groups}
\label{subsec:nilpotent}

In the attemps to relate the large-scale geometry of pairs of connected Lie groups, several sufficient criteria have been found (e.g. for quasiisometry in \cite{BreuillardLarge}, \cite{CornulierCones11}, for sharing simply transitive Riemannian models in \cite{cowling2021homogeneous}, and for $O(u)$-bilipschitz equivalence in \cite{CornulierCones11}).
These criteria consist for a large part\footnote{Additional subtelty comes from the ``medium-scale'' topology of the groups when it is non trivial.} in going back to the Lie algebra and simplifying its structure. 
These criteria can sometimes be formulated using deformations and degenerations of Lie algebras.
\begin{itemize}
    \item
    Pansu's theorem on asymptotic cones: those are degenerations.
    \item
    Cornulier's theorem on asymptotic cones: when the exponential radical is abelian, those are degenerations. (This is the case for the Heintze groups considered in Section \ref{sec:proofB}.)
    \item 
    Twistings (or normal modifications) introduced by \cite{AlekHMN} and \cite{GordonWilson} and studied in relation to large-scale geometry in \cite{cowling2021homogeneous}: those are deformations.
\end{itemize}

\subsubsection{Cornulier's Theorem}
A reference for the facts used in this section can be found in \cite[Chapter VII]{BbkiCartanAlg}.
Let $\mathfrak g$ be a completely solvable Lie algebra. Let $\mathfrak h$ be a Cartan subalgebra (maximal nilpotent self-normalizing in $\mathfrak g$), and let $\mathfrak r = \liminf_i C^i \mathfrak g$ be the limit of the descending central series of $\mathfrak g$. 
Decompose the adjoint representation of $\mathfrak h$ in $\mathfrak r$ into primary components \cite{BbkiCartanAlg},
\[ \mathfrak r = \bigoplus_{\omega \in \operatorname{Hom}(\mathfrak h, \mathbf R)} \mathfrak r^\omega =  \bigoplus_{\omega \in \operatorname{Hom}(\mathfrak h, \mathbf R)} \limsup_{i \to + \infty} \ker (\alpha - \omega)^i \]
where $\alpha$ is the structural morphism $\mathfrak h \to \operatorname{Der}(\mathfrak r)$.
Note that since $\mathfrak h$ is nilpotent, its ideal $\mathfrak w = \mathfrak h \cap \mathfrak r$ lies within $\mathfrak r^0$.
So the semisimple part $\delta$ of $\alpha$ factors through $\pi: \mathfrak h \to \mathfrak h / \mathfrak w$, and the resulting $\mathfrak h/\mathfrak w$-module decomposes as
\begin{equation}
    \label{eq:decomposition-of-exprad}
    \mathfrak r = \mathfrak r^0 \oplus \bigoplus_{\underline \omega \in \operatorname{Hom}(\mathfrak h/\mathfrak w, \mathbf R), \, \underline \omega \neq 0} \mathfrak r^{\underline \omega} =
\mathfrak r^0 \oplus \bigoplus_{\underline \omega \neq 0} \ker(\underline \delta - \underline \omega) 
\end{equation}
where $\delta = \underline \delta \circ \pi$ and $ \omega = \underline \omega \circ \pi$.
There is a Lie algebra homomorphism $\underline \delta_\infty : \operatorname{gr}(\mathfrak h /\mathfrak w) \to \operatorname{Der}(\mathfrak r)$ and the following diagram: 

\[ \begin{tikzcd}
\mathfrak h \arrow[r, swap,"\pi"] \arrow[bend left=30, rrrd, "\delta"] & \mathfrak h / \mathfrak w \arrow[rd] \arrow[bend left=0, rrd, "\underline \delta"]  & & \\
& & \mathfrak h/ [\mathfrak h, \mathfrak h] \arrow[r] & \operatorname{Der}(\mathfrak r). \\
\operatorname{gr} (\mathfrak h) \arrow[r] & \operatorname{gr}(\mathfrak h / \mathfrak w) \arrow[ur] \arrow[urr, swap, "\underline \delta_\infty"]
\end{tikzcd} \]

\begin{theorem}[Cornulier {\cite{CornulierCones11}}]
\label{th:cornulier-red}
Let $\mathfrak g$ be a completely solvable Lie algebra.
With notation as above, define
$\mathfrak g_1 = \mathfrak r \rtimes_{\underline \delta} (\mathfrak h /\mathfrak w)$ and $\mathfrak g_\infty$ as $\mathfrak r \rtimes_{\underline \delta_\infty} \operatorname{gr}(\mathfrak h/\mathfrak w)$.
Let $G$, $G_1$, $G_\infty$ be simply connected with Lie algebras $\mathfrak g$, $\mathfrak g_1$, $\mathfrak g_\infty$ respectively.
Then
\begin{enumerate}[{\rm (a)}]
    \item 
    \label{item:cornulier-def}
    $G$ and $G_1$ are $O(\log)$-bilipschitz equivalent.
    \item
    \label{item:pansu-goodman-def}
    If $C^{s+1} \mathfrak h = 0$, then $G_1$ and $G_\infty$ are $O(r^{1-1/s})$-bilispchitz equivalent.
\end{enumerate}
\end{theorem}

\begin{proposition}
\label{prop:cornulier-and-deg}
Let $\mathfrak g$ be a completely solvable Lie algebra.
Assume that $\mathfrak r = \liminf C^i \mathfrak g$ is abelian.
Let $\mathfrak g_1$, $\mathfrak g_\infty$ be as in Theorem \ref{th:cornulier-red}.
Then 
\begin{equation}
    \label{eq:cornulier-degenerations}
    \mathfrak g \longrightarrow_{\mathrm{deg}} \mathfrak g_1 \longrightarrow_{\mathrm{deg}} \mathfrak g_\infty.
\end{equation}
\end{proposition}
We already encountered examples of this:
\begin{itemize}
    \item 
    When $\mathfrak g$ is nilpotent, the right degeneration in \eqref{eq:cornulier-degenerations} is Example \ref{ex:nilpotent}. Note that $\mathfrak r = 0$ in this case.
    \item 
    When $\mathfrak g = [\mathfrak g, \mathfrak g] \oplus \mathbf RA$ and $\operatorname{ad}_A$ is unipotent, the left degeneration in \eqref{eq:cornulier-degenerations} is the contraction occuring in Theorem \ref{th:Lauret} \eqref{item:deg}. $\mathfrak r$ is abelian and has codimension $1$ in this case.
\end{itemize}

\begin{proof}
Start with the decomposition \eqref{eq:decomposition-of-exprad}.
Decompose further $\mathfrak r$ into $\mathfrak r^0$ and a direct sum of subspaces $U_i$ such that
\begin{equation}
    \bigoplus_{j \geqslant i} U_i = \bigoplus_{\underline \omega \neq 0} \ker (\alpha - \omega)^j.
\end{equation}
Since $\mathfrak h$ is nilpotent, we have that $\mathfrak w = \mathfrak r \cap \mathfrak h \subseteq \mathfrak r^0$.
Decompose $\operatorname{Vect}(\mathfrak g)$ into a direct sum
\begin{align}
    \mathfrak g & = \bigoplus_{i \geqslant 1} U_i \oplus \mathfrak r^0 \oplus \mathcal H 
\end{align}
where $\mathcal H$ is a linear subspace of $\operatorname{Vect}(\mathfrak g)$ representing $\mathfrak h/\mathfrak w$.
Denote by $\mu$, resp. $\mu_1$, resp. $\mu_\infty$ the brackets of the three laws on $\operatorname{Vect}(\mathfrak g)$.
For $t > 0$, set 
\[
\varphi_t(u) = 
\begin{cases}
t^{i} u & u \in U_i \\
u & u \in \mathfrak r^0 \oplus \mathcal H \\
\end{cases}
\]
Then for all $h \in \mathcal H$ and $u \in U_i \cap \mathfrak r^\omega$,
\begin{align*}
    \varphi_t . \mu(h,u)  = \varphi_t^{-1} \mu( h, t^iu) 
    & =\varphi_t^{-1} t^i (\omega(h) u + v) \quad \text{where} \quad v \in U_{i-1}  \\
    & = \omega(h) u + t^{-1-i} t^i v \\
    & = \omega(h) u + O(t^{-1}),
\end{align*}
so $\mu$ contracts to $\mu_1$ through $\varphi$.
\end{proof}

\begin{remark}
\label{rem:degeneration-perturbs-brackets}
We do not know whether Proposition \ref{prop:cornulier-and-deg} holds in general. This is because the contraction we used in the proof perturbs in general the brackets in $\mathfrak r$. We know no obstruction of the kind expressed in Lemma \ref{lem:lsc-zariski} for a degeneration from $\mathfrak g$ to $\mathfrak g_1$.
\end{remark}

A question we would like to raise, in view of Remark \ref{rem:degeneration-perturbs-brackets} in particular, is whether the group $R = \exp(\mathfrak r)$ is a large-scale invariant (if the completely solvable $G$ and $G'$ are $O(u)$-equivalent, does it hold that $\liminf C^i G \simeq \liminf C^i G'$?). This appears quite difficult to determine in general, because this subgroup is exponentially distorted and gets totally disconnected in the asymptotic cones \cite{AsInv}. Nevertheless, it holds by Cornulier's formula \eqref{eq:cornlier-formula} and Theorem \ref{th:geodim} that the dimension loss 
\begin{equation}
    \label{eq:dydak-higes-solvable}
    \dim R = \operatorname{geodim}(G) - \operatorname{conedim}(G)
\end{equation}
is indeed a $o(r)$-bilipschitz invariant. 
When $G$ is of Heintze type, the $o(r)$-bilipschitz invariance of \eqref{eq:dydak-higes-solvable} is materialized into the Gromov boundary; note also that the quasiisometry class of $R$ is a quasiisometry invariant of $G$ \cite[Theorem A]{KiviojaThesis}; but we have no asymptotic invariant in general.
We also note that the nonnegativity $\operatorname{asdim}_{\mathrm{AN}}(X) - \operatorname{conedim} X \geqslant 0$ holds more generally, a result of Dydak and Higes \cite{DydakHiges}.

\subsubsection{Shadows and deformations}

Let $\mathfrak g_0$ be a completely solvable algebra. 
We call torus an abelian algebra of semisimple derivations of $\mathfrak g_0$.
A torus $\mathfrak t$ is compactly embedded if every $T \in \mathfrak t$ has purely imaginary spectrum.
Maximal tori are conjugated.
 
\begin{definition}[Special case of {\cite[2.2]{GordonWilson}}]
Let $\mathfrak t$ be a maximal compactly embedded torus.
A modification\footnote{Modification is a more general notion, we only consider modifications of completely solvable Lie algebras for our purposes in the present paper.} of $\mathfrak g_0$ is a Lie subalgebra of $\mathfrak g_0 \rtimes \mathfrak t$ that is transverse to $\mathfrak t$.
We call $\mathfrak g_0$ the shadow of $\mathfrak g$.
\end{definition}

The modification $\mathfrak g$ is the graph of a linear map $\tau : \mathfrak g_0 \to \mathfrak t$, called the modification map: for $X \in \mathfrak g_0$, $\tau(X)$ is the only $T \in \mathfrak t$ such that $X+T \in \mathfrak g$. Note that $\mathfrak t$ being abelian, $[\mathfrak g, \mathfrak g] \subset \mathfrak g_0$. 

\begin{definition}
Let $\mathfrak g$, $\mathfrak g_0$ ant $\tau$ be as above. 
We say that $\mathfrak g$ is a twisting (and $\tau$ a twisting map) if in addition $[\mathfrak g, \tau(\mathfrak g_0)] \subseteq \mathfrak g$.
\end{definition}

If $\mathfrak g_0$ is nilpotent, all its modifications are twistings \cite{GordonWilson}.

Early works on modifications (\cite{AlekHMN}, \cite{GordonWilson}) were concerned by the problem of finding adequate data for the classification of solvmanifolds.
Modification have attracted the attention more recently because if $\mathfrak g$ is a modification of $\mathfrak g_0$, then $G$, $G_0$ and $G_0 \rtimes T$ (where $T$ is the compact torus of $\operatorname{Aut}(G_0)$ with Lie algebra $\mathfrak t$) share a common Riemannian model, especially they are quasiisometric (\cite{CornulierDimCone}, \cite{cowling2021homogeneous}).

\begin{proposition}
Let $\mathfrak g_0$, $\mathfrak l$, $\mathfrak g$ and $\tau$ be as above. Assume that $\mathfrak g$ is a twisting.
Define $\omega_\tau (X \wedge Y) = [\tau(X), Y] + [X, \tau(Y)]$ for $X, Y \in \mathfrak g_0$. Then, 
\begin{enumerate}[{\rm (1)}]
    \item 
    \label{item:omega-phi-cocycle}
    $\omega_\tau \in Z^2(\mathfrak g_0, \mathfrak g_0)$, where $\mathfrak g_0$ acts in $\mathfrak g_0$ through the adjoint representation.
    \item
    \label{item:omega-phi-deformation}
    $[\omega_\tau]$ is a linearly expandable infinitesimal deformation of $\mathfrak g_0$. The associated formal deformation goes through $O_{\mathfrak g}$.
\end{enumerate}
\end{proposition}

\begin{proof}
\eqref{item:omega-phi-cocycle}
By definition, 
\begin{align*}
    d \omega_\tau (X \wedge Y \wedge Z) & = [X, [\tau (Y), Z] + [Y, \tau (Z)]] - [Y, [\tau(X), Z] + [X, \tau(Z)]] \\
    & \quad + [Z, [\tau(X), Y] + [X, \tau(Y)]] - [ \tau[X,Y], Z ] -  [[X,Y], \tau(Z)] \\
    & \quad + [ \tau[X,Z], Y ] + [[X,Z], \tau(Y)]- [ \tau[Y,Z], X ] - [[Y,Z], \tau(X)] \\
    & =
    [\tau [X,Z], Y] - [\tau[Y,Z], X] - [\tau[X,Y],Z]
\end{align*}
where we used the Jacobi identity in $\mathfrak g_0 \rtimes \mathfrak l$ three times.
If $\mathfrak g$ is a twisting then $\tau$ is a homomorphism \cite{GordonWilson}, and since $\mathfrak t$ is abelian, the remaining terms all vanish.

\eqref{item:omega-phi-deformation}
On the vector space $\operatorname{Vect}(\mathfrak g_0)$, let us denote by $\mu_0$ the law of $\mathfrak g_0$ and put $\lambda = \mu_0 + \omega_\tau /t$. Let us check that $\lambda(1) \in O_\mathfrak g$: for every $X,Y \in \mathfrak g_0$,
\begin{align*}
    \lambda(1)(X\wedge Y) & = \mu_0(X \wedge Y) + \omega_\tau (X \wedge Y) 
    \\
    & = \pi_0 \left( [X,Y] + [X, \tau(Y)] + [Y, \tau(X)] + [\tau(X), \tau(Y)] \right) \\
\end{align*}
where $\pi_0$ denotes the projection onto $\mathfrak g_0$ parallel to $\mathfrak t$ (Remember that $\mathfrak t$ is abelian, hence $[ \tau(X), \tau(Y)] =0$). Thus the law $\lambda(1)$ is that of the Lie subalgebra $\mathfrak g$ in $\mathfrak g_0 \rtimes \mathfrak t$.
\end{proof}

Beware that it is not true that a twisting $\mathfrak g$ degenerates to its shadow $\mathfrak g_0$. Here is a counterexample.

\begin{example}[Solvable example]
\label{ex:twisting-non-deg}
Let $\mathfrak g_0 = \mathfrak b(3, \mathbf R)$, with basis $(X_1, X_2, T)$ and brackets
\begin{equation}
    [X_1, X_2] = 0, \; [T, X_1] = X_1,\;  [T,X_2] = X_2.
\end{equation}
Let $(dx_1, dx_2, dt)$ be the dual basis.
Then $H^2(\mathfrak g_0, \mathfrak g_0)$ is $3$-dimensional, and contains the linearly independent classes $\omega_1 = [dt \wedge dx_1 \otimes X_2]$ and  $\omega_2 = [dt \wedge dx_2 \otimes X_1]$.
$\omega_1$ and $\omega_2$ are linearly expandable into degenerations, 
but $\omega_1 - \omega_2$ is linearly expandable into a family of twistings that are not degenerations. See Appendix \ref{sec:adjoint-computations} for a more general computation.
\end{example}

If $\mathfrak h$ is a graded Lie algebra and $\mu \in O_{\mathfrak h}$, the groups $H^2(\mathfrak \mu, \mu)$ are naturally graded.
This is the case, for instance, if $\mathfrak h$ is a Carnot-graded group.

\begin{figure}
    \centering
\begin{tikzpicture}[line cap=round,line join=round,>=angle 45,x=0.6cm,y=0.6cm]
\clip(-11.55,0) rectangle (13.36,16.33);
\draw [shift={(-5.80,-32.77)}] plot[domain=1.22:1.63,variable=\t]({1*34.84*cos(\t r)},{1*34.84*sin(\t r)});
\draw [shift={(-3.82,-24.77)}] plot[domain=1.22:1.63,variable=\t]({34.84*cos(\t r)},{1*34.84*sin(\t r)});
\draw [shift={(20.66,-0.91)}] plot[domain=2.75:3.04,variable=\t]({1*28.81*cos(\t r)},{28.81*sin(\t r)});
\draw [shift={(34.67,-2.92)}] plot[domain=2.75:3.04,variable=\t]({1*28.81*cos(\t r)},{28.81*sin(\t r)});
\draw [->] (0,6) -- (-2,14);
\draw (0,6)-- (3.4,9.31);
\draw (0,6)-- (6.35,2.43);
\draw (0,6)-- (7.28,6.03);
\draw (-6.18,14.96) node[anchor=north west] {$\mathrm{weight} <0$};
\draw (2.47,10.8) node[anchor=north west] {$ \mathrm{weight}\; 1 $};
\draw (6.39,2.57) node[anchor=north west] {\parbox{7.71 cm}{$ \mathrm{weight}\; \\  2 $}};
\draw (-0.93,5.91) node[anchor=north west] {$ \mu $};
\draw (-0.46,9.28) node[anchor=north west] {$ \mu +k \xi_1 $};
\draw (1.73,4.07) node[anchor=north west] {$ \mu + k \xi_2  $};
\draw (2.81,6.04) node[anchor=north west] {$ \mu + k(\xi_1 + \xi_2) $};
\draw (-4.57,12.01) node[anchor=north west] {$ \mu + \omega/t $};
\draw (-5.76,4.08)-- (4.43,14.01);
\draw (2.07,12.01) node[anchor=north west] {$ \mu + \omega + k \xi_1 $};
\draw [->] (0,6) -- (3.4,9.31);
\draw [->] (0,6) -- (6.35,2.43);
\draw (-5.39,3.03) node[anchor=north west] {$ \mathcal N_6(\mathbf R) $};
\draw (-0.61,15.01) node[anchor=north west] {$ \mathcal L_6(\mathbf R) $};
\draw [shift={(-5.4,1.68)},line width=1.6pt]  plot[domain=0.23:1.44,variable=\t]({1*6.92*cos(\t r)+0*6.92*sin(\t r)},{0*6.92*cos(\t r)+1*6.92*sin(\t r)});
\draw (-6.36,9.35) node[anchor=north west] {$ O(\mu) $};
\draw [shift={(0.08,12.55)},dash pattern=on 3pt off 3pt]  plot[domain=4.7:5.33,variable=\t]({1*6.55*cos(\t r)+0*6.55*sin(\t r)},{0*6.55*cos(\t r)+1*6.55*sin(\t r)});
\draw [shift={(-0.08,-1.06)},dash pattern=on 3pt off 3pt]  plot[domain=0.84:1.56,variable=\t]({1*7.06*cos(\t r)+0*7.06*sin(\t r)},{0*7.06*cos(\t r)+1*7.06*sin(\t r)});
\draw (-3.48,9.86) node[anchor=north west] {$ O(\mathfrak l_{6,6}) $};
\draw (2.74,2.06) node[anchor=north west] {$ O(\mathfrak l_{6,11}) $};
\draw (4.15,7.34) node[anchor=north west] {$ O(\mathfrak l_{6,12}) $};
\begin{scriptsize}
\fill [color=black] (-5.82,-32.77) circle (1.5pt);
\fill [color=black] (0,6) circle (1.5pt);
\fill [color=black] (2.14,8.09) circle (1.5pt);
\fill [color=black] (3.97,3.77) circle (1.5pt);
\fill [color=black] (4.27,6.02) circle (1.5pt);
\fill [color=black] (-1.32,11.27) circle (1.5pt);
\fill [color=black] (1.39,11.04) circle (1.5pt);
\end{scriptsize}
\end{tikzpicture}
    \caption{Sketch of $\mathcal N_6(\mathbf R) \subseteq \mathcal L_6(\mathbf R)$ around the Carnot Lie algebra from Example \ref{exm:L67}, and some deformations.}
    \label{fig:N6L6}
\end{figure}
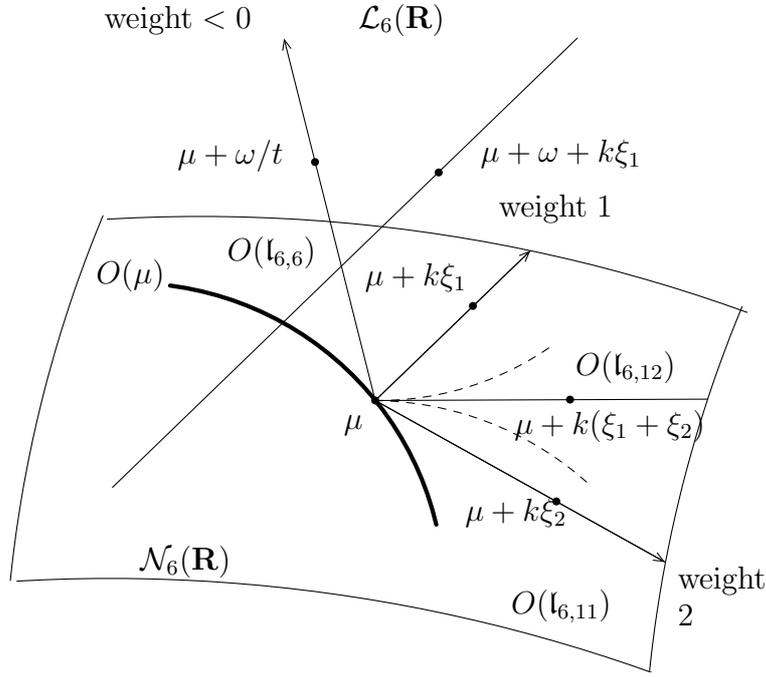

\begin{example}[A nilpotent example]
\label{exm:L67}
Let $G_0$ be the simply connected $6$-dimensional Lie group having Lie algebra with basis $X_1, \ldots , X_6$ and the nonzero brackets
\begin{equation*}
    [X_1, X_2] = X_3, \, [X_1, X_3] = X_4, \, [X_1, X_4] = X_5.
\end{equation*}
(This algebra is denoted $\mathfrak l_{6,7}$ in \cite{deGraafclass}.)
Note that $X_6$ generates an abelian direct factor.
$\mathfrak g_0$ is a Carnot-graded algebra under the grading
\begin{equation*}
    \left\langle X_1, X_2, X_6 \right\rangle \oplus \langle X_3 \rangle \oplus \langle X_4 \rangle \oplus \langle X_5 \rangle.
\end{equation*}
Let $(X^1, \ldots , X^6)$ be the dual basis, and denote by $\mu$ the law. Consider the following cochains:
\[ 
\begin{array}{llll}
    \omega = X_2^{16} + X_1^{62}; & \xi_1 = X_5^{23}; & 
    \xi_2 = X_5^{26}; & \xi_3 = X_4^{26} + X_5^{36}. 
\end{array}
\]
where we abbreviate $X_k \otimes X^i \wedge X^j$ into $X_k^{ij}$.
These are cocycles but not coboundaries in $H^2(\mu, \mu)$
(See the computations in Appendix \ref{subsec:nilpotent-adjoint}).
The cohomology classes of $\omega$, $\xi_1$, $\xi_2$ and $\xi_3$ have weight $-1$, $1$, $2$ and $1$ respectively under the grading.
The classes of $\xi_1$, $\xi_2$, $\xi_3$ and $\xi_1 + \xi_2$ linearly expand into formal deformations; the corresponding laws are $\mathfrak l_{6,6}$, $\mathfrak{l}_{6,12}$, $\mathfrak{l}_{6,13}$ and $\mathfrak{l}_{6,11}$ respectively in \cite{deGraafclass}. All these are also degenerations, entering into the description of Example \ref{ex:nilpotent} (adding cocycles with positive weights to the Lie algebra law does not change the lower central filtration).
For every $k \in \mathbf R$, the cohomology class of $k \xi_1 + \omega$ is also integrable into a formal deformation, going through families of twistings of $\mathfrak g_0 = \mathfrak l_{6,7}$ if $k=0$, or through families of twistings of $\mathfrak l_{6,6}$ if $k \neq 0$.

One can check that all the simply connected solvable Lie groups that are $O(u)$-bilipschitz equivalent to $G_0$ appear as deformations of $G_0$ of the form described above. 
Let us give a few words on this. If $H$ is such a group, then by \cite{PansuCCqi} and  \cite{BreuillardLarge}, the shadow $\mathfrak h_0$ of its Lie algebra is isomorphic to $\mathfrak l_{6,i}$ with $i \in \lbrace 6,7, 11, 12, 13 \rbrace$.  The three last algebras are irreducible (they have no direct factor), and maximal tori for those are computed in \cite{magnin2007adjoint}; in this way we check that only $\mathfrak l_{6,6}$ and $\mathfrak l_{6,7}$ possess derivations with purely imaginary spectra. Thus either $\mathfrak h$ is $\mathfrak l_{6,11}$, $\mathfrak l_{6,12}$, $\mathfrak l_{6,13}$ or a twisting of $\mathfrak l_{6,7}$ or $\mathfrak l_{6,6}$. Note that the quasiisometry classes in this family are not completely known, though it is expected that they are given by the isomorphism type of the shadow \cite[Conjecture 19.114]{CornulierQIHLC}.
The real cohomology rings $H^\ast(\mathfrak l_{6,6}, \mathbf R)$ and $H^\ast(\mathfrak l_{6,7}, \mathbf R)$ are isomorphic \cite[\S 19.6.6]{CornulierQIHLC}.
However, $b_2(\mathfrak l_{6,13})$ is $4$ while it is $5$ for all the others; and we can check that the rank of $H^2(\mathfrak l_{6,i}, \mathbf R) \odot H^2(\mathfrak l_{6,i}, \mathbf R) \to H^4(\mathfrak l_{6,i}, \mathbf R)$ given by the cup product is $2$ for $i=6$ while it is $3$ for $i=11$ and $i=12$ (see \ref{cohom-ring} for some details). 
So none of the Lie groups $L_{6,11}$, $L_{6,12}$, $L_{6,13}$ are quasiisometric to $G_0$ by \cite{SauerHom}, or the recent \cite[Corollary C]{GotfredKyed}.
\end{example}

It seems natural to expect Carnot graded algebra to have more twistings than their nilpotent deformations (on the other extreme, observe that characteristically nilpotent Lie algebras, which lie ``deep down'' in $\mathcal N_n(\mathbf R)$, have no twistings). This is indeed the situation on the previous example. Thus we ask:

\begin{ques}
\label{ques:deformations}
Let $\mathfrak h$ be a solvable Lie algebra over $\mathbf R$, $H$ its associated simply connected Lie group, and let $G_0$ be a simply connected Carnot-group. 
Assume that $H$ and $G_0$ are $O(u)$-bilipschitz.
Is there a formal deformation of $\mathfrak g_0$ going through $\mathfrak h$?
\end{ques}

The author does not know whether the dimension of compactly embedded maximal tori is upper semicontinuous on $\mathcal N_n(\mathbf R)$ (which would hint towards a positive answer to Question \ref{ques:deformations}). These tori embed linearly in $H^1(\lambda, \lambda)$ whose dimension we have seen to be upper semicontinuous in Lemma \ref{lem:lsc-zariski}. However the codimension of the tori may be high (see Table \ref{tab:max-tori}).

\begin{table}[t]
    \centering
    \begin{tabular}{|c|c|l|l|}
    \hline
    $\mathfrak g$ & $\dim \mathfrak t_{\max}^{\mathrm{c}}$ & $\dim \mathfrak t_{\max}$   & $\dim H^1(\mathfrak g, \mathfrak g)$ \\
    \hline
    $\mathfrak l_{6,7}$ & 1 & 3 & 9  \cite[$\mathfrak g_{5,5}  \times \mathbf C$]{Magnin08} \\
    $\mathfrak l_{6,6}$ & 1 & 2 & 8 \cite[$\mathfrak g_{6,5}  \times \mathbf C$]{Magnin08} \\ 
    $\mathfrak l_{6,12}$ & 0 & 2 \cite[4.2.5]{magnin2007adjoint} & 7 \cite[$\mathfrak g_{6,11}$]{Magnin08} \\ 
    $\mathfrak l_{6,11}$ & 0 & 1 \cite[4.1.1]{magnin2007adjoint} & 6 \cite[$\mathfrak g_{6,12}$]{Magnin08} \\ 
    $\mathfrak l_{6,13}$ & 0 & 2 \cite[4.2.6]{magnin2007adjoint} & 5 \cite[$\mathfrak g_{6,13}$]{Magnin08} \\ 
    \hline
    \end{tabular}
    \vskip 10pt
    \caption{Dimensions of maximal tori, compactly embedded maximal tori, and outer derivation spaces for the nilpotent Lie algebras of Example \ref{exm:L67}.}
    \label{tab:max-tori}
\end{table}

\subsection{Higher-rank symmetric spaces}
\label{subsec:hrss}
The real rank of a symmetric space $X$ is $o(r)$-bilipschitz invariant, as it is the covering dimension of asymptotic cone \cite{CornulierDimCone} or, more in line with \cite[Corollary 6.11]{KleinerLeebQI}, the minimal degree above which all relative homology group of subspaces in $\operatorname{Cone}^\bullet_\omega X$ vanish. This can be refined: the restricted root system is invariant.

\begin{proposition}[After Kleiner and Leeb]
Let $\phi : X \to Y$ be a sublinear bilipschitz equivalence between irreducible symmetric spaces $X$ of rank $\geqslant 2$. 
Then, the restricted root systems associated with $X$ and $Y$ are isomorphic.
\end{proposition}

\begin{proof}
The spherical Tits building at infinity in $\operatorname{Cone}_\omega(X)$ has the same appartments as the Tits boundary of $X$ \cite[Theorem 5.2.1]{KleinerLeebQI}. 
\end{proof}

We note that the rank $p$ irreducible symmetric spaces of noncompact type \[ \operatorname{SU}(p,2q)/S(\mathbf {U}_p \times \mathbf {U}_{2q}) \text{ and } \operatorname{Sp}(p,q)/\operatorname{Sp}(p) \times \operatorname{Sp}(q) \] have same restricted root system $BC_p$ and same asymptotic Assouad-Nagata dimension $4pq$ \cite[Table V p. 518]{HelgaDiffSym}.
Thus, we could not distinguish them with our techniques, and Question \ref{ques:classification} remains open so far for them.

The author is grateful to P. Pansu and G. Rousseau for bringing these pairs to his attention.

\subsection{Right-angled Fuchsian buildings of uniform thickness}
\label{subsec:rafb}
Given $(p,q)$ such that $p \geqslant 5$ and $q \geqslant 2$, the finitely presented group 
\[ \Gamma_{p,q} = \langle s_1, \ldots s_p \mid [s_i, s_{i+1}], s_i^q \rangle. \]
has a model $I_{p,q}$ which is a $\operatorname{CAT}(-1)$ cellular complex generalizing the cellular action of the hyperbolic Coxeter group $\Gamma_{p,2}$ on $\mathbb H^2_{\mathbf R}$ tesselated by right-angled $p$-gons and, following \cite{BourdonFuchsI},
\begin{equation}
\label{eq:confdim-buildings}
\operatorname{Cdim} \partial_\infty I_{p,q} = \frac{\log \tau(p,q)}{\log \tau(p,2)} = 1 + \frac{\log (q-1)}{\operatorname{argch}\left( \frac{p-2}{2} \right)}.
\end{equation} 
The conformal dimension of $I_{p,q}$ is not rational unless $q=2$.
It is proven in \cite{pallier2019conf} that $\operatorname{Cdim}_{O(u)} \partial_\infty I_{p,q} = \operatorname{Cdim} \partial_\infty I_{p,q}$, so that it is a $O(u)$-bilipschitz invariant.
Using Poincaré profiles, Hume, Mackay and Tessera proved that there can be no coarse embedding $I_{p,q} \to I_{p', q'}$ when $\operatorname{Cdim} \partial_\infty I_{p,q} > \operatorname{Cdim} \partial_\infty I_{p',q'}$ \cite[Theorem 13.2]{HmTPoinc}.
We found the equality case in \eqref{eq:confdim-buildings} to be related to the following conjecture.

\begin{conjecture}[Four exponential conjecture, {\cite[p.11]{LangT}}]
\label{conj:four-exponentials}
Let $\beta_1, \beta_2$ be complex numbers, linearly independent over $\mathbf Q$, and let $z_1, z_2$ be complex numbers, also linearly independent over $\mathbf Q$.
Then, at least one of the numbers $e^{\beta_i z_j}$ is transcendental.
\end{conjecture}

The analogous statement with two triples $\beta_1, \beta_2, \beta_3, z_1, z_2, z_3$ is known as the six exponentials theorem \cite{LangT}.
The unconditional form of the following Proposition is stated as a conjecture in \cite{TysonThesis} and \cite{MTconfdim}. (We indicate with an asterisk that our statement is conditional.)

\begin{cproposition}
\label{prop:tyson's conjecture}
Assume that Conjecture \ref{conj:four-exponentials} holds, and let $(p,q,p',q')$ be integers such that $p,p' \geqslant 5$ and $q, q' \geqslant 3$. Then the boundaries of the buildings $I_{p,q}$ and $I_{p', q'}$ have equal conformal dimension if and only if there exists positive integers $M,N$ such that
\begin{align}
    \left( q-1 \right)^N & = (q'-1)^M 
    \label{eq:tyson-first-CN} \\
T_N \left( \frac{p-2}{2} \right) & = T_M \left( \frac{p'-2}{2} \right)
\label{eq:tyson-second-CN}  
\end{align}
where $T_k$ is the Tchebychev polynomial of the first kind and degree $k$.
\end{cproposition}

\begin{proof}
Negating the conclusion amounts asserting that there exists an irrational number $z$ and a quadruple $(p,q,p',q')$ such that $z \log (q'-1) = \log (q-1)$ and $z \operatorname{argch}((p-2)/2) = \operatorname{argch}((p'-2)/2)$.
Define $\beta_1 = \log (q'-1)$,
\begin{align*}
    \beta_2 & = \operatorname{argch}((p'-2)/2) = \log \left( \frac{p'-2}{2} + \sqrt{p'(p'/4-1)} \right),
\end{align*}
$z_1=1$ and $z_2 = z$.
Then, $\beta_2/\beta_1$ is not rational. But $e^{\beta_1}$, $e^{z \beta_1} = q-1$, $e^\beta_2$ and $e^{z\beta_2} = \frac{p-2}{2} + \sqrt{p(p/4-1)}$ are all algebraic.
\end{proof}

Note that the $I_{p,q}$ are quasiisometrically rigid for $q \geqslant 3$ \cite{XieBuildings}; especially they are classified up to quasiisometry by the pair $(p,q)$ for $q \geqslant 3$. 

\begin{cproposition}
Assume that Conjecture \ref{conj:four-exponentials} holds. 
If there exists a $O(u)$-bilipschitz equivalence $\phi: I_{p,q} \to I_{p',q'}$ then \eqref{eq:tyson-first-CN} and \eqref{eq:tyson-second-CN} hold for some $M,N \geqslant 1$.
\end{cproposition}

\begin{proof}
This directly follows from Proposition \ref{prop:tyson's conjecture} and \cite{pallier2019conf}.
\end{proof}

Let us finish with some questions.
Though we consider \eqref{eq:tyson-first-CN} and \eqref{eq:tyson-second-CN} perhaps not sufficient for $O(u)$-equivalence between $I_{p,q}$ and $I_{p', q'}$, we could not distinguish them up to this relation. In a slightly different direction, one can ask:

\begin{ques}
Assume that $p,q,p', q'$ are as in \eqref{eq:tyson-first-CN} and \eqref{eq:tyson-second-CN}. Are the groups $\Gamma_{p,q}$ and $\Gamma_{p',q'}$ (non)-measure equivalent? If yes, are they $L^p$-measure equivalent for some $p < +\infty$?
\end{ques}

(We recall that a measure equivalence between the finitely generated $\Gamma$ and $\Lambda$ is given by a couple of free, commuting, measure preserving actions of $\Gamma$ and $\Lambda$ on a Lebesgue space $(\Omega,m)$ with Borel fundamental domains of finite measure $X$ and $Y$, such that the associated cocycles $c: G \times X \to H$ and $H \times Y \to G$ have $c(g, \cdot) \in L^p$ for all $g$.)

Closer to the problems of this paper, the author also believe the following question to remain currently open, and of some interest in view of \cite{HmTPoinc}.

\begin{ques}
Assume that $p,q,p', q'$ are as in \eqref{eq:tyson-first-CN} and \eqref{eq:tyson-second-CN}, and $p \neq p'$. Is there a coarse embedding $I_{p,q} \to I_{p',q'}$?
\end{ques}

\begin{appendix}


\section{Methods used for the cohomology computations}

The cohomology groups used in this paper are obtained by direct methods (i.e. by somewhat explicit computations of derivative, cocycles and coboundaries). 
We summarize them below.

\label{app:cohomcomput}

\subsection{Solvable Lie algebras}
\label{sec:adjoint-computations}

Let $\mathfrak b(n, \mathbf K)$ be defined as in Example \ref{exm:bnK}, with coordinates $(z_\alpha, \tau,s)$.
Decompose $z_\alpha  = x_\alpha + iy_\alpha$; for $1 \leqslant \alpha_1 < \cdots < \alpha_s \leqslant n-1$ and $1 \leqslant \beta \leqslant n-1$, denote

\begin{align*}
     X^{\alpha_1, \ldots ,\alpha_s}_{\beta} & = dx_{\alpha_1} \wedge \cdots \wedge dx_{\alpha_s} \otimes \frac{\partial}{\partial{x_{\beta}}}
     \\
     Y^{\alpha_1, \ldots ,\alpha_s}_{\beta} & = dy_{\alpha_1} \wedge \cdots \wedge dy_{\alpha_s} \otimes \frac{\partial}{\partial{y_{\beta}}} 
     \\
     T & = \partial_\tau, \; T^\ast = d \tau, \; S = \partial_s, \; S^\ast = d s.
\end{align*} 
We apply the summation convention where we simplify $\sum_\mu X^{\alpha \mu}_\alpha$ into $X^{\alpha \mu}_\alpha$ in any equality between tensors whenever $\mu$ is unbound in the RHS.

The Lie algebra grading $\mathfrak s_0 = \langle S \rangle$, $\mathfrak s_1 = \langle X_\alpha, Y_\alpha \rangle$ and $\mathfrak s_2 =\langle T \rangle$ extends to a grading of the mixed exterior/tensor product so that,
say, 
$X^{\alpha_1, \ldots \alpha_s}_{\beta} \wedge T$
has weight $ 1 -s + 2$.
The differentials have degree $0$, hence the cohomology groups are graded accordingly. Finally, $\mathfrak b(n , \mathbf C)$ has a preferred complex structure, $JX_\alpha = Y_\alpha$, $JY\alpha = -X_\alpha$ and $JS= T$. This is because $\mathbf H^n_{\mathbf C}$ is Hermitian. $J$ is not an automorphism; nevertheless, 
\begin{equation*}
    \widetilde{J}(Z) = 
    \begin{cases}
    J(Z) & Z \in \mathfrak s_1 \\
    Z & Z \in \mathfrak s_0 \oplus \mathfrak s_2
    \end{cases}
\end{equation*}
is an automorphism, and we will use it in order to simplify the computations.

\subsubsection{Results}

\begin{proposition}
\label{prop:first-adjoint-R}
$H^1(\mathfrak b(n, \mathbf R), \mathfrak b(n, \mathbf R)) = \bigoplus_{(\alpha, \beta) \neq (n-1, n-1)} \langle X_\alpha^\beta \rangle$.
\end{proposition}

\begin{proposition}
\label{prop:second-adjoint-R}
$H^2(\mathfrak b(n, \mathbf R), \mathfrak b(n, \mathbf R)) = \bigoplus_{(\alpha, \beta) \neq (n-1, n-1)} \langle [X_\alpha \otimes S^\ast \wedge X^\beta] _{\alpha \neq \beta} \rangle$.
\end{proposition}

Since the computation of $H^1(\mathfrak b(n, \mathbf C), \mathfrak b(n,\mathbf C))$ proves useful for Proposition \ref{prop:first-adjoint-R} and is not significantly harder than the case $n=2$ used in Section \ref{sec:proofE}, we provide the result for all $n$ and the weight decomposition below.

\begin{proposition}
\label{prop:first-adjoint-C}
$\dim H^1 (\mathfrak b(n, \mathbf C), \mathfrak b(n, \mathbf C)) = (n-1)^2 + 1$ and
\[    H^1 (\mathfrak b(n, \mathbf C), \mathfrak b(n, \mathbf C))
     = \operatorname{span} 
    \begin{cases}
    [T \otimes S^\ast] & \mathrm{weight} -2 \\
    [X_\alpha^\beta - Y_\beta^\alpha]_{1 \leqslant \alpha, \beta \leqslant n-1}, \alpha \neq \beta  &  \mathrm{weight}\, 0 \\
    [X_\alpha \otimes Y^{\alpha}]_{1 \leqslant \alpha \leqslant n-1}  & \mathrm{weight} \, 0. \\
    \end{cases} \]
\end{proposition}

\begin{proposition}
\label{prop:first-adjoint-S''}
Let $\mathfrak s''$ be the four-dimensional Lie algebra $\mathbf R^3 \rtimes_{\alpha} \mathbf R$, where $\alpha = \operatorname{diag}(J_2(1), 2)$. Then $\dim H^1(\mathfrak s'', \mathfrak s'') = 4$.
\end{proposition}

\subsubsection{Method}
In order to gain space for Propositions \ref{prop:first-adjoint-R} to \ref{prop:first-adjoint-C} we gather the computation for $\mathbf R$ and $\mathbf C$ and then extract the case of $\mathbf K = \mathbf R$. We abbreviate the derivative of the complex $C^\bullet(\mathfrak b(n, \mathbf K), \mathfrak b(n, \mathbf K))$ into $d'_{\mathbf K}$.

\begin{lemma}
\label{lem:differentials-of-1-cochains-C}
For all $\alpha, \beta$ such that $1 \leqslant \alpha, \beta \leqslant n-1$,
\[
    \begin{array}{rlrl}
    d'_{\mathbf C} X_\alpha^\beta & = X^\beta \wedge Y^\alpha \otimes T;
    &
    d'_{\mathbf C}Y_\alpha^\beta  & =  X^\alpha \wedge Y^\beta \otimes T;
    \\
    d'_{\mathbf C} (X_\alpha \otimes  Y^\beta) & = - 2 Y^{\alpha \beta} \otimes T; &
    d'_{\mathbf C} (Y_\alpha \otimes  X^\beta)  &= 2 X^{\alpha \beta} \otimes T; \\
    d'_{\mathbf C}(X_\alpha \otimes S^\ast) & = - Y^\alpha \wedge S^\ast \otimes T; &
    d'_{\mathbf C}(Y_\alpha \otimes S^\ast) & = X^\alpha \wedge S^\ast \otimes T. \\
     d'_{\mathbf C}(T \otimes X^\alpha) & = X^\alpha \wedge S^\ast \otimes T; &
     d'_{\mathbf C}(T \otimes Y^\alpha) & = Y^\alpha \wedge S^\ast \otimes T \\ 
     d'_{\mathbf C}(T \otimes S^\ast) & = 0 &
     d'_{\mathbf C}(S\otimes T^\ast) & = - 2 S^\ast \wedge T \otimes S.
     \end{array}
     \]
     \begin{align*}
      d'_{\mathbf C}(S \otimes S^\ast) & = - X^\mu \wedge S^\ast \otimes X_\mu - Y^\mu \wedge S^\ast \otimes  Y_\mu - 2T \wedge S^\ast \otimes T. \\
      d'_{\mathbf C}(T \otimes T) & = - X^\mu \wedge  Y^\mu \otimes T.
     \\
    d'_{\mathbf C}(S \otimes X^\alpha) & =  X_{\ell}^{\alpha \ell} - X^\alpha \wedge Y^\ell \otimes Y_\ell + X^\alpha \wedge S^\ast \otimes S - 2 X^\alpha \wedge T^\ast \otimes T; 
    \\
     d'_{\mathbf C}(S \otimes Y^\alpha) & =  Y_{\ell}^{\alpha \ell} - Y^\alpha \wedge X^\ell \otimes X_\ell + Y^\alpha \wedge S^\ast \otimes S - 2 Y^\alpha \wedge T^\ast \otimes T; 
     \\
     d'_{\mathbf C}(X_\alpha \otimes T^\ast)  & =  ( X^\mu \wedge Y^\mu - S^\ast \wedge T) \otimes X_\alpha + Y^\alpha \wedge T^\ast \otimes T 
     \\
     d'_{\mathbf C}(Y_\alpha \otimes T^\ast)  & =  ( - Y^\mu \wedge X^\mu - S^\ast \wedge T) \otimes Y_\alpha - X^\alpha \wedge T^\ast \otimes T 
     \end{align*}
\end{lemma}

\begin{proof}
The whole computation being of little interest, let us explain in detail only how one computes $d_{\mathbf C} X_\alpha^\beta$, $d_{\mathbf C} Y_\alpha^\beta$ and $d(X_\alpha \otimes S^\ast)$ as a sample of the techniques employed. Applying \eqref{eq:chevalley-eilenberg-adjoint},
\begin{align*}
    d'_{\mathbf C}X_{\alpha}^\beta (X_{\mu \ell}) & = - X_\alpha^\beta [X_\mu, X_\ell] + [X_\mu, X_\alpha^\beta X _\ell] - [X_{\ell}, X_\alpha^\beta X_\mu]  \\
    & = [X_\mu, \delta_{\beta\ell} X_\alpha] - [X_\ell, \delta_{\beta\mu} X_\alpha] =0; \\
    d'_{\mathbf C}X_\alpha^\beta (Y_{\mu \ell}) & = - X_\alpha^\beta [Y_\mu, Y_\ell] + [Y_\mu, X_\alpha^\beta Y _\ell] - [Y_{\ell}, X_\alpha^\beta Y_\mu] =  0;  \\
    d'_{\mathbf C} X_\alpha^\beta (X_\mu \wedge Y_\ell) & = - X_\alpha^\beta [X_\mu, Y_\ell] + [X_\mu , X_\alpha^\beta Y_\ell] - [Y_\ell, X_\alpha^\beta X_\mu]  \\
    & = - \delta_{\mu \ell} X_\alpha^\beta T - \delta_{\beta\mu} [Y_\ell, X_\alpha] = \delta_{\alpha \ell} \delta_{\beta \mu} T; \\
    d'_{\mathbf C} X_\alpha^\beta (X_\mu \wedge S) & = - X_\alpha^\beta [X_\mu, S] + [X_\mu, X_\alpha^\beta S] - [S, X_\alpha^\beta X_\mu] \\ & = X_\alpha^\beta X_\mu - [S, \delta_{\beta\mu}X_\alpha] = \delta_{\beta\mu} (X_\alpha - X_\alpha) = 0, \\
    d'_{\mathbf C}X_\alpha^\beta (X_\mu \wedge T) 
    &= - X_\alpha^\beta [X_\mu, T] + [X_\mu, X_\alpha^\beta T] - [T, X_\alpha^\beta X_\mu] = - [T, \delta_{\beta \mu} X_\alpha]  = 0. \\
    d'_{\mathbf C}X_\alpha^\beta (Y_\mu \wedge S) & = - X_\alpha^\beta [Y_\mu, S] + [Y_\mu, X_\alpha^\beta S] - [S, X_\alpha^\beta Y_\mu] = X_\alpha^\beta Y_\mu =0. \\
    d'_{\mathbf C}X_\alpha^\beta(Y_\mu \wedge T) &= - X_\alpha^\beta [Y_\mu, T] + [Y_\mu, X_\alpha^\beta T] - [T, X_\alpha^\beta Y_\mu] = 0; \\
    d'_{\mathbf C} X_\alpha^\beta (S \wedge T) & = - X_\alpha^\beta [S, T] + [S, X_\alpha^\beta T ] - [T, X_\alpha^\beta S]= - 2 X_\alpha^\beta T = 0,
\end{align*}
which yields the expression of $d_{\mathbf C} X_\alpha^\beta$.
Applying $\widetilde J$ produces $d_{\mathbf C} Y_\alpha^\beta$:
\begin{equation*}
    d'_{\mathbf C} Y_\alpha^\beta = d_{\mathbf C} \widetilde{J}X_\alpha^\beta = \widetilde J X^\beta \wedge \widetilde{J} X^\alpha \otimes T = Y^\beta \wedge (-X^\alpha) \otimes T.
\end{equation*}

Now for $d'_{\mathbf C}(X_\alpha \otimes S^\ast) (X_\mu \wedge S)$, using the observation that $\ker S^\ast = \mathfrak s_1 \oplus s_2 =[\mathfrak s, \mathfrak s]$, we can reduce the number of terms needed for the computations of $d(X_\alpha \otimes S^\ast)$: this will evaluate to zero for any bivector where $S$ is not a factor. The remaining terms are:
\begin{align*}
    d'_{\mathbf C}(X_\alpha \otimes S^\ast) (X_\mu \wedge S) & = [X_\mu, X_\alpha \otimes S^\ast S] = 0 ;\\
    d'_{\mathbf C}(X_\alpha \otimes S^\ast) (Y_\mu \wedge S) & = [Y_\mu, X_\alpha \otimes S^\ast S] = [Y_\mu, X_\alpha]=  - \delta_{\alpha \mu} T; \\
    d'_{\mathbf C}(X_\alpha \otimes S^\ast) (S \wedge T ) & = - [T, X_\alpha \otimes S^\ast S] + [S, X_\alpha \otimes S^\ast T] = 0.
\end{align*}

\end{proof}


\begin{lemma}
\label{lem:differentials-of-1-cochains-R}
For all $\alpha, \beta$ such that $1 \leqslant \alpha, \beta \leqslant n-1$,
\begin{align*}
    d'_{\mathbf R} X_\alpha^\beta & = 0 \\
    d'_{\mathbf R}(S \otimes X^\alpha) & = X^{\alpha \ell}_\ell + X^\alpha \wedge S^\ast \otimes S \\
    d'_{\mathbf R}(X_\alpha \otimes S^\ast) & = 0 \\
    d'_{\mathbf R}(S \otimes S^\ast) & = - X^\mu \wedge S^\ast \otimes X^\mu. 
    \end{align*}
\end{lemma}

\begin{proof}
Discard the terms with $Y,S,T$ in the results of Lemma \ref{lem:differentials-of-1-cochains-C}.
\end{proof}

\begin{lemma}[Differentials of $2$-cochains]
\label{lem:differentials-of-2-cochains-R}
\begin{align*}
    d'_{\mathbf R}(X_\alpha^{\beta \gamma}) & = - 2 X_\alpha \otimes X^{\beta \gamma} \wedge S^\ast \\
    d'_{\mathbf R}(S \otimes X^{\alpha \beta}) & = - 2 S \otimes X^{\alpha \beta} \wedge S^\ast \\
    d'_{\mathbf R}(X_\alpha \otimes S^\ast \wedge X^{\beta}) & = 0 \\
    d'_{\mathbf R}(S \otimes S^\ast \wedge X^{\alpha}) & = 2 X_\mu \otimes X^{\mu \alpha} \wedge S^\ast.
\end{align*}
\end{lemma}

\begin{proof}
Let us concentrate on $d'_\mathbf R(X_\alpha^{\beta \gamma})$. First recall that if $\mathfrak g$ is a Lie algebra and $\omega$ is a $\mathfrak g$-valued $2$-form, then for every $U, V, W \in \mathfrak g^3$,
\begin{align*}
    d\gamma(U\wedge V \wedge W) & = [U, \gamma(V \wedge W)] + [V, \gamma(W \wedge U)] + [W, \gamma(U \wedge V)] \\
    & \quad - \gamma([U,V] \wedge W) - \gamma([W,U] \wedge V) - \gamma([V,W] \wedge U).
\end{align*}
Applying this, one checks readily that $d'_\mathbf R(X_\alpha^{\beta \gamma}) (X_{\mu\nu \ell})= 0$ while
\begin{align*}
    d'_\mathbf R(X_\alpha^{\beta \gamma}) (X_{\mu\nu} \wedge S)
    & = 0 - X_\alpha^{\beta\gamma} \left( [X_\mu, X_\nu] \wedge S + [S, X_\mu] \wedge X_\nu + [X_\nu, S] \wedge X_\mu \right) \\
    & = \left( - \delta_{\beta\mu} \delta_{\gamma \nu} +\delta_{\beta\nu} \delta_{\gamma\mu} \right) X_\alpha. 
\end{align*}
\end{proof}

\begin{remark}
The Lie algebra-valued forms have a wedge product. However we did not find a clear computational advantage in using formulae for the derivative of $2$-forms using this wedge product. 
\end{remark}

\begin{proof}[Proof of Proposition \ref{prop:first-adjoint-R}]
By Lemma \ref{lem:differentials-of-1-cochains-R}, 
\[Z^1(\mathfrak b(n,\mathbf R),  \mathfrak b(n,\mathbf R)) = \operatorname{span} \left\{ X_\alpha^\beta, X_\alpha \otimes S^\ast \right\}_{1 \leqslant \alpha, \beta \leqslant n-1}, \]
while $d'_{\mathbf R}X_\alpha = X_\alpha \otimes S^\ast$ and $d'_{\mathbf R}S = - X_\mu^\mu$.
\end{proof}

\begin{proof}[Proof of Proposition \ref{prop:second-adjoint-R}]
By Lemma \ref{lem:differentials-of-1-cochains-R}, 
\begin{equation}
\notag
    B^2(\mathfrak b(n, \mathbf R), \mathfrak b(n, \mathbf R)) = \operatorname{span} \left\{ X^{\alpha \ell}_\ell + X^\alpha \wedge S^\ast \otimes S, X_\mu \otimes X^\mu \wedge S^\ast \right\}_{\alpha = 1, \ldots , n-1}
\end{equation}
while by Lemma \ref{lem:differentials-of-2-cochains-R},
\begin{equation}
    \notag
    Z^2 (\mathfrak b(n, \mathbf R), \mathfrak b(n, \mathbf R)) = \operatorname{span} \left\{ 
    X_\mu^{\mu \alpha} + S \otimes S^\ast \wedge X^\alpha,
    X_\alpha \otimes X^\beta \wedge S^\ast \right\}_{\alpha = 1, \ldots, n-1}. 
\end{equation}
\end{proof}

\begin{proof}[Proof of Proposition \ref{prop:first-adjoint-C}]
The $1$-coboundaries are computed as

\begin{align*}
    d'_{\mathbf C}X_\alpha & = X_\alpha \otimes S^\ast - T \otimes Y_\alpha \\
    d'_{\mathbf C}Y_\alpha & = Y_\alpha \otimes S^\ast + T \otimes X_\alpha \\
    d'_{\mathbf C}S & = - X_\mu^\mu - Y_\mu^\mu - 2 T \otimes T^\ast \\
    d'_{\mathbf C}T & = 2T \otimes S^\ast.
\end{align*}
The right-hand side of equations in Lemma \ref{lem:differentials-of-1-cochains-C} provide the $1$-cocycles.
\end{proof}

\begin{proof}[Proof of Proposition \ref{prop:first-adjoint-S''}]
We recall that 
$\mathfrak s''$ is the Lie algebra over $X_1, \ldots ,X_4$ with 
$\mathfrak s'' = \langle X_2, X_3, X_4 \rangle \oplus \langle X_4 \rangle$ and
\begin{equation*}
    \operatorname{ad}(X_4) = \begin{pmatrix} 1 & 1 & 0 \\
    0 & 1 & 0 \\ 0 & 0 & 2
    \end{pmatrix}
\end{equation*}
in the basis $(X_2, X_3, X_4)$.

Omitting the symbol $\sum_{i<j}$ and using $d'(X_k^\ell)(x^{ij} X_{ij}) = - x^{ij} X_k^\ell[X_i, X_j] + \delta_{j\ell} x^{ij} [X_i, X_{k}] - \delta_{i\ell} x^{ij} [X_j, X_k]$ one finds
\[
\begin{array}{ll}
     d' (X_1^1) = 2 X_1^{14} + X_1^{24} 
     & d'(X_1^2) = -X_1^{24} + X_2^{24} \\
      d'(X_1^3) = - X_1^{12} + 3 X_1^{34} 
     & d'(X_1^4) = 0 \\
    d'(X_2^1) = X_2^{14} 
    & d'( X_2^2) = X_1^{24} \\
    d'(X_2^3) = - X_1^{34} - 3 X_2^{34}
    & d'(X_2^4) = 0
    \\
    d'(X_3^1) = - X_3^{14}
    & d'(X_3^2) = - X_3^{24} \\
    d'(X_3^3) = 0
    & d'(X_3^4) = 0 \\
    d'(X_4^1) = X_4^{14} + X_4^{24} + X_1^{14}
    & d'(X_4^2) = X_4^{24} - X_1^{12} + 2 X_3^{23} \\
    d'(X_4^3) = 2X_4^{34} - X_1^{13} - X_1^{23} - X_2^{23}
    & d'(X_4^4) = - X_1^{14} - X_1^{24} - X_2^{24} - X_3^{34}.
\end{array}
\]
All the nonzero co-boundaries obtained are linearly independent, hence
\begin{equation*}
    H^1(\mathfrak s'', \mathfrak s'') = \operatorname{span}(X_3^3, X_1^4, X_2^4, X_3^4). 
\end{equation*}
\end{proof}

\subsection{The Lie algebra $\mathfrak l_{6,7}$ and its nilpotent deformations}

We expand below on the computations needed for Example \ref{exm:L67}.

\subsubsection{Adjoint cohomology of \texorpdfstring{$\mathfrak l_{6,7}$}{l67}}
\label{subsec:nilpotent-adjoint}

One computes the $2$-coboundaries as:
\[
\begin{array}{lll}
    d'_{\mu}(X_1^1) = X_{3}^{12} + X_4^{13} + X_5^{14} 
    & d'_{\mu}(X_2^1) = 0
    & d'_{\mu}(X_3^1) = 0 \\
    d'_{\mu}(X_1^2) = X_4^{23} + X_5^{24} 
    & d'_{\mu}(X_2^2) =X_3^{12}
    & d'_{\mu}(X_3^2 ) = X_4^{12} \\
    d'_{\mu}(X_1^3) = - X_3^{23} + X_5^{34} 
    & d'_{\mu}(X_2^3) = X_3^{13} 
    & d'_{\mu}(X_3^3 ) = X_4^{13} \\
    d'_{\mu}(X_1^4) = - X_3^{24} - X_4^{34} 
    & d'_{\mu}(X_2^4)  = X_3^{14} 
    & d'_{\mu}(X_3^4) = X_4^{14}  \\
    d'_{\mu}(X_1^5) = -X_3^{25} - X_4^{35} - X_5^{45} 
    & d'_{\mu}(X_2^5) = X_3^{15} 
    & d'_{\mu}(X_3^5) = X_4^{15} \\
    d'_{\mu}(X_1^6) = - X_3^{26} - X_4^{36} - X_5^{46} 
    & d'_{\mu}(X_2^6) = X_3^{16}  
    & d'_{\mu}(X_3^6) = X_4^{16} \\

    d'_{\mu}(X_4^1) = 0
    & d'_{\mu}(X_5^1) = 0
    & d'_{\mu}(X_6^1) = 0 \\
    d'_{\mu}(X_4^2) = X_5^{12}
    & d'_{\mu}(X_5^2) = 0
    & d'_{\mu}(X_6^2) = 0 \\
    d'_{\mu}(X_4^3) = - X_4^{12} + X_5^{13}
    & d'_{\mu}(X_5^3) = - X_5^{12}
    & d'_{\mu}(X_6^3 ) = - X_6^{12} \\
    d'_{\mu}(X_4^4) = - X_4^{14} + X_5^{15}
    & d'_{\mu}(X_5^4)  = - X_5^{13}
    & d'_{\mu}(X_6^4) =  - X_6^{13} \\
    d'_{\mu}(X_4^5) = - X_4^{14} +X_5^{15}
    & d'_{\mu}(X_5^5) = - X_5^{14}
    & d'_{\mu}(X_6^5) = - X_6^{14} \\
    d'_{\mu}(X_4^6) = X_5^{16}
    & d'_{\mu}(X_5^6) = 0
    & d'_{\mu}(X_6^6) = 0.
\end{array}
\]

This justify the assertion that $\omega$, $\xi_1$, $\xi_2$, $\xi_3$ and $\xi_1 + \xi_2$ are not coboundaries. We now check that they are cocycles. In the computation below, we omit the symbols $\sum_{i<j<k}$, and get rid of the terms that can be checked to equal $0$ by direct inspection.

\begin{align*}
    d'_\mu \omega (x^{ijk} X_{ijk}) & =  d'_\mu (X_2^{16}+X_1^{62})  (x^{ijk} X_{ijk}) \\
    & = - x^{ijk} [X_j, X_2^{16} X_{ik}] - x^{ijk} [X_i, X_1^{26} X_{jk}] + x^{ijk}  [X_j, X_1^{26} X_{ik}] = 0;
\end{align*}
\begin{align*}
    d'_\mu \xi_1 (x^{ijk} X_{ijk}) 
    & =  d'_\mu (X_5^{23})  (x^{ijk} X_{ijk}) \\
    & =  -x^{ijk}  X_5^{23}([X_i, X_j] \wedge X_k) + x^{ijk} X_5^{23}([X_i, X_k] \wedge X_j)   \\
    & \quad  - x^{ijk}  X_5^{23}([X_j, X_k] \wedge X_i)  + x^{ijk} [X_i, X_5^{23} X_{jk}] - x^{ijk} [X_j, X_5^{23} X_{ik}] \\
    & \quad + x^{ijk} [X_k, X_5^{23} X_{ij}] \\
    & = x^{123} [X_1, X_5^{23} X_{23}] + x^{234}[X_4,X_5] + x^{235} [X_5,X_5] = 0; 
\end{align*}
\begin{align*}
    d'_\mu \xi_2 (x^{ijk} X_{ijk}) & =  d'_\mu (X_5^{26})  (x^{ijk} X_{ijk}) \\
    & = x^{ijk} [X_i, X_5^{26} X_{jk}] - x^{ijk} [X_j, X_5^{26} X_{ik}] + x^{ijk} [X_k, X_5^{26} X_{ij}] = 0; \\
    d'_\mu \xi_3 (x^{ijk} X_{ijk}) & =  d'_\mu (X_4^{26} + X_5^{36})  (x^{ijk} X_{ijk}) \\
    & = 
    x^{126}[X_1, X_4^{26} X_{26}]
    -
    x^{126} X_5^{36} ([X_1, X_2] \wedge X_6) \\
    & =  x^{126}(X_5 - X_5) = 0  .
\end{align*}

Note that $\dim H^2(\mu, \mu) = 18$, by the computer-produced \cite[Table 11]{Magnin08}.

\subsubsection{Cohomology rings}
\label{cohom-ring}
Let $d_i$ denote the derivative of $C^\ast(\mathfrak l_{6,i}, \mathbf R)$.
Then
\[
\begin{array}{lll}
     d_7 X^3 = - X^{12} & d_7 X^4 = - X^{13} & d_7 X^5 = - X^{14} \\
     d_{11} X^3 = - X^{12} & d_{11} X^4 = - X^{13} & d_{11} X^5 = - X^{14} - X^{23} \\
     d_{12} X^3 = - X^{12} & d_{12} X^4 = - X^{13} & d_{12} X^5 = - X^{14} - X^{26}. 
\end{array}
\]
In the notation of \cite{magnin2007adjoint}, $\mathfrak l_{6,11} \otimes \mathbf C$ is $\mathfrak g_{6,12}$ while $\mathfrak l_{6,12} \otimes \mathbf C$ is $\mathfrak  g_{6,11}$. We compute that 
\begin{align*}
    H^2 (\mathfrak l_{6,7},\, \mathbf R) 
    & = \langle [X^{15}],\, [X^{16}],\, [X^{23}],\, [X^{25} - X^{34}], [X^{26}] \rangle \\
    H^2 (\mathfrak l_{6,11},\, \mathbf R) 
    & = \langle [X^{13}],\, [X^{15}],\, [X^{23}],\, [X^{16} +  X^{25} - X^{34}], [X^{26} - X^{45}] \rangle \\
    H^2 (\mathfrak l_{6,12},\, \mathbf R) 
    & = \langle [X^{13}],\, [X^{15}],\, [X^{16} - X^{34}],\, [X^{26} - X^{45}], [X^{24}] \rangle
\end{align*}
The computations for $\mathfrak l_{6,11}$ and $\mathfrak l_{6,12}$ can be checked with the help of the derivative $d_i$ written down with computer assistance in \cite{magnin2007adjoint} on p.44 and p.72 respectively (\cite{magnin2007adjoint} uses the $\mathfrak g_{6,i}$ notation recalled above for the Lie algebras and writes $\omega^{i,j}$ for $X^{ij}$).
Moreover,
\begin{align*}
    B^4 (\mathfrak l_{6,7},\, \mathbf R) 
    & = \langle X^{1234},\, X^{1235},\, X^{1236},\, X^{1245}, X^{1246},\, X^{1256},\, X^{1356} \rangle \\
    B^4 (\mathfrak l_{6,11},\, \mathbf R) 
    & = \langle X^{1234},\, X^{1235},\, X^{1245},\, X^{1246}, X^{1236} - X^{1345},\, X^{1346},\, X^{1256} + X^{2345} \rangle \\
    B^4 (\mathfrak l_{6,12},\, \mathbf R) 
    & = \langle X^{1234},\, X^{1235},\, X^{1245},\, X^{1246}, X^{1236} - X^{1345},\, X^{1346},\, X^{1256} + X^{2345} \rangle
\end{align*}
If $\pi_i$ denotes the cup product $H^2(\mathfrak l_{6,i}, \mathbf R) \times H^2(\mathfrak l_{6,i}, \mathbf R) \to H^4(\mathfrak l_{6,i}, \mathbf R)$, then
 \begin{align*}
     \operatorname{Im}(\pi_7) & = \langle [X^{1345}],[X^{2346}] \rangle \\
     \operatorname{Im}(\pi_{11}) & = \langle [X^{1236}],[X^{2345}],[X^{1456} +X^{2346}] \rangle \\
     \operatorname{Im}(\pi_{12}) & = \langle [X^{1236}],[X^{1456}+2X^{2346}], [X^{1256}] \rangle
 \end{align*}
 One checks using the co-boundaries spaces $B^4$ listed above that these vectors are linearly independent, completing the proof that the cohomology rings of $\mathfrak l_{6,11}$ and $\mathfrak l_{6,12}$ are not isomorphic to that of $\mathfrak l_{6,7}$.

\end{appendix}

\newcommand{\etalchar}[1]{$^{#1}$}

\end{document}